\documentclass[a4paper,11pt,reqno]{amsart}

\usepackage{hyperref}
\usepackage{xcolor}
\hypersetup{
    colorlinks,
    linkcolor={red!50!black},
    citecolor={blue!50!black},
    urlcolor={blue!80!black},
    pdftitle={Intersection theory on derived stacks},
    pdfauthor={Adeel A. Khan}
}

\usepackage[T1]{fontenc}
\usepackage{microtype}
\usepackage{MnSymbol}


\makeatletter
\newtheoremstyle{plain}
  {.5\baselineskip\@plus.2\baselineskip\@minus.2\baselineskip}
  {.5\baselineskip\@plus.2\baselineskip\@minus.2\baselineskip\@plus.5em}
  {\slshape}
  {}
  {\bfseries}
  {.}
  { }
  {}
\newtheoremstyle{definition}
  {.5\baselineskip\@plus.2\baselineskip\@minus.2\baselineskip}
  {0.8\baselineskip\@plus.2\baselineskip\@minus.2\baselineskip\@plus.5em}
  {}
  {}
  {\bfseries}
  {.}
  { }
  {}
\makeatother
\theoremstyle{plain}

\makeatletter \@addtoreset{equation}{section} \makeatother
\renewcommand{\theequation}{\thesection.\arabic{equation}}
\newtheorem{defn}[equation]{Definition}
\newtheorem{thm}[equation]{Theorem}
\newtheorem*{thm*}{Theorem}
\newtheorem{lem}[equation]{Lemma}
\newtheorem{cor}[equation]{Corollary}
\newtheorem{prop}[equation]{Proposition}
\newtheorem*{defthm*}{Definition/Theorem}

\theoremstyle{definition}
\newtheorem{rmk}[equation]{Remark}
\newtheorem{exam}[equation]{Example}
\newtheorem{constr}[equation]{Construction}

\newtheorem*{exam*}{Example}

\newcommand\arXiv[1]{\href{http://arxiv.org/abs/#1}{arXiv:#1}}

\usepackage[utf8]{inputenc}

\usepackage{xspace}

\newcommand{\changelocaltocdepth}[1]{%
  \addtocontents{toc}{\protect\setcounter{tocdepth}{#1}}%
  \setcounter{tocdepth}{#1}}

\usepackage{xpatch}
\makeatletter
\xpatchcmd{\@tocline}
  {\hfil\hbox to\@pnumwidth{\@tocpagenum{#7}}\par}
  {\ifnum#1<0\hfill\else\dotfill\fi\hbox to\@pnumwidth{\@tocpagenum{#7}}\par}
  {}{}
\def\l@subsection{\@tocline{2}{0pt}{4pc}{6pc}{}}
\def\l@subsubsection{\@tocline{3}{0pt}{8pc}{8pc}{}}
\makeatother

\newcommand{\nc}{\newcommand}
\nc{\renc}{\renewcommand}
\nc{\ssec}{\subsection}
\nc{\sssec}{\subsubsection}
\nc{\on}{\operatorname}
\nc{\term}[1]{#1\xspace}

\usepackage{tikz}
\usetikzlibrary{matrix}
\usepackage{tikz-cd}

\tikzset{
  commutative diagrams/.cd,
  arrow style=tikz,
  diagrams={>=latex}}
\makeatletter
\tikzset{
  column sep/.code=\def\pgfmatrixcolumnsep{\pgf@matrix@xscale*(#1)},
  row sep/.code   =\def\pgfmatrixrowsep{\pgf@matrix@yscale*(#1)},
  matrix xscale/.code=%
    \pgfmathsetmacro\pgf@matrix@xscale{\pgf@matrix@xscale*(#1)},
  matrix yscale/.code=%
    \pgfmathsetmacro\pgf@matrix@yscale{\pgf@matrix@yscale*(#1)},
  matrix scale/.style={/tikz/matrix xscale={#1},/tikz/matrix yscale={#1}}}
\def\pgf@matrix@xscale{1}
\def\pgf@matrix@yscale{1}
\makeatother

\makeatletter
\g@addto@macro\normalsize{%
  \setlength\abovedisplayskip{8pt plus 3pt minus 2pt}%
  \setlength\belowdisplayskip{8pt plus 3pt minus 2pt}%
  \setlength\abovedisplayshortskip{4pt plus 3pt}%
  \setlength\belowdisplayshortskip{6pt plus 3pt minus 2pt}%
}
\makeatother

\usepackage{enumitem}
\setlist[enumerate,1]{label={(\alph*)},itemsep=\parskip,leftmargin=*,align=left,widest=vii}
\newlist{thmlist}{enumerate}{1}
\setlist[thmlist,1]{label={\em(\roman*)},ref={\upshape{(\roman*)}},itemsep=\parskip,leftmargin=*,align=left,widest=vii}
\newlist{defnlist}{enumerate}{1}
\setlist[defnlist,1]{label={(\roman*)},itemsep=\parskip,leftmargin=*,align=left,widest=vii}
\newlist{inlinelist}{enumerate*}{1}
\setlist[inlinelist,1]{label={(\alph*)}}

\nc{\sD}{\ensuremath{\mathcal{D}}\xspace}
\nc{\sE}{\ensuremath{\mathcal{E}}\xspace}
\nc{\sF}{\ensuremath{\mathcal{F}}\xspace}
\nc{\sG}{\ensuremath{\mathcal{G}}\xspace}
\nc{\sL}{\ensuremath{\mathcal{L}}\xspace}
\nc{\sN}{\ensuremath{\mathcal{N}}\xspace}
\nc{\sO}{\ensuremath{\mathcal{O}}\xspace}
\nc{\sS}{\ensuremath{\mathcal{S}}\xspace}
\nc{\sT}{\ensuremath{\mathcal{T}}\xspace}
\nc{\sU}{\ensuremath{\mathcal{U}}\xspace}
\nc{\sV}{\ensuremath{\mathcal{V}}\xspace}
\nc{\sW}{\ensuremath{\mathcal{W}}\xspace}
\nc{\sX}{\ensuremath{\mathcal{X}}\xspace}
\nc{\sY}{\ensuremath{\mathcal{Y}}\xspace}
\nc{\sZ}{\ensuremath{\mathcal{Z}}\xspace}

\nc{\bA}{\ensuremath{\mathbf{A}}\xspace}
\nc{\bD}{\ensuremath{\mathbf{D}}\xspace}
\nc{\bG}{\ensuremath{\mathbf{G}}\xspace}
\nc{\bL}{\ensuremath{\mathbf{L}}\xspace}
\nc{\bP}{\ensuremath{\mathbf{P}}\xspace}
\nc{\bQ}{\ensuremath{\mathbf{Q}}\xspace}
\nc{\bR}{\ensuremath{\mathbf{R}}\xspace}
\nc{\bT}{\ensuremath{\mathbf{T}}\xspace}
\nc{\bV}{\ensuremath{\mathbf{V}}\xspace}
\nc{\bZ}{\ensuremath{\mathbf{Z}}\xspace}

\nc{\mrm}[1]{\ensuremath{\mathrm{#1}}\xspace}
\nc{\mbf}[1]{\ensuremath{\mathbf{#1}}\xspace}
\nc{\mot}{\mrm{mot}}

\renc{\ge}{\geqslant}
\renc{\le}{\leqslant}

\nc{\id}{\mathrm{id}}

\DeclareMathOperator{\rk}{\mathrm{rk}}
\DeclareMathOperator{\Hom}{\on{Hom}}
\nc{\uHom}{\underline{\smash{\Hom}}}
\DeclareMathOperator{\Maps}{\on{Maps}}
\DeclareMathOperator{\Aut}{\on{Aut}}
\nc{\colim}{\varinjlim}
\renc{\lim}{\varprojlim}

\DeclareMathOperator*{\fibprod}{\times}
\renc{\setminus}{\smallsetminus}

\usepackage{mathtools}

\newcommand{\thmref}[1]{Theorem~\ref{#1}}
\newcommand{\secref}[1]{Sect.~\ref{#1}}
\newcommand{\ssecref}[1]{Subsect.~\ref{#1}}
\newcommand{\sssecref}[1]{(\ref{#1})}
\newcommand{\lemref}[1]{Lemma~\ref{#1}}
\newcommand{\propref}[1]{Proposition~\ref{#1}}
\newcommand{\corref}[1]{Corollary~\ref{#1}}
\newcommand{\rmkref}[1]{Remark~\ref{#1}}
\newcommand{\defnref}[1]{Definition~\ref{#1}}
\renewcommand{\eqref}[1]{(\ref{#1})}
\newcommand{\constrref}[1]{Construction~\ref{#1}}
\newcommand{\examref}[1]{Example~\ref{#1}}
\newcommand{\itemref}[1]{\ref{#1}}

\nc{\A}{\bA}
\renc{\P}{\bP}
\nc{\Spec}{\on{Spec}}
\nc{\Cech}{\textnormal{\v{C}}}
\nc{\Perf}{\on{Perf}}
\nc{\cl}{{\mrm{cl}}}
\nc{\Bl}{\on{Bl}}
\nc{\vir}{\mrm{vir}}
\nc{\dimvir}{\mrm{vd}}
\nc{\CH}{\on{A}}
\nc{\fibprodR}{\fibprod^\bR}
\nc{\eul}{\mrm{eul}}
\nc{\et}{\mrm{\acute{e}t}}
\nc{\SH}{\on{SH}}
\renc{\H}{\mrm{H}}
\nc{\BM}{\mrm{BM}}
\nc{\CBM}{\mrm{C}^{\BM}_\bullet}
\nc{\Ccoh}{\mrm{C}^\bullet}
\nc{\un}{\mbf{1}}
\nc{\D}{\bD}
\nc{\Z}{\bZ}
\nc{\Q}{\bQ}
\nc{\KGL}{\mrm{KGL}}
\nc{\MGL}{\mrm{MGL}}
\nc{\rightrightrightarrows}
  {\mathrel{\substack{\textstyle\rightarrow\\[-0.6ex]
   \textstyle\rightarrow \\[-0.6ex]
   \textstyle\rightarrow}}}
\nc{\rightrightrightrightarrows}
  {\mathrel{\substack{\textstyle\rightarrow\\[-0.6ex]
   \textstyle\rightarrow \\[-0.6ex]
   \textstyle\rightarrow \\[-0.6ex]
   \textstyle\rightarrow}}}
\nc{\Res}{\on{Res}}
\nc{\Lis}{\mrm{Lis}}
\nc{\vb}[1]{\langle #1\rangle}
\nc{\K}{\on{K}}
\nc{\G}{\on{G}}
\nc{\Pic}{\on{Pic}}
\renc{\sp}{\mrm{sp}}
\nc{\Td}{\on{Td}}
\nc{\td}{\on{td}}
\nc{\ch}{\on{ch}}
\renc{\th}{\mrm{th}}

\nc{\inftyCat}{\term{$\infty$-category}}
\nc{\inftyCats}{\term{$\infty$-categories}}

\nc{\inftyGrpd}{\term{$\infty$-groupoid}}

\nc{\da}{\term{derived algebraic}}
\nc{\dDM}{\term{derived Deligne--Mumford}}
\nc{\dA}{\term{derived Artin}}

\nc{\lisse}{{\triangleleft}}

\title[Intersection theory on derived stacks]{Intersection theory on derived stacks\vspace{-2mm}}
\author[Adeel A. Khan]{Adeel A. Khan\vspace{-1mm}}
\date{2026-11-01}

\begin{document}

\begin{abstract}
We define integral cycle complexes and Chow homology of derived Artin stacks.
To any quasi-smooth family of derived Artin stacks we associate a bivariant fundamental class, encoding the virtual fundamental classes of its fibres and their deformation invariance up to coherent homotopy.
We extend the Riemann--Roch theorem of Baum--Fulton--MacPherson from quasi-projective schemes to derived Artin stacks, and thereby obtain a virtual Riemann--Roch formula for quasi-smooth derived Artin stacks.
As another application of the formalism, we construct intersection products on regular Artin stacks locally of finite type over a field.
Finally, we prove an integral comparison of our Chow homology with the cycle-theoretic Chow groups of Kresch.
\vspace{-5mm}
\end{abstract}

\maketitle

\renewcommand\contentsname{\vspace{-1cm}}
\tableofcontents

\parskip 0.2cm
\thispagestyle{empty}

\changelocaltocdepth{1}
\section*{Introduction}

The goal of this paper is to develop the foundations for intersection theory on derived stacks.
For classical Artin stacks, at least those of finite type over a field and with affine stabilizers, a Chow homology theory satisfying the desired properties has been developed by Kresch \cite{Kresch}, building on Edidin--Graham \cite{EdidinGraham}.
On the other hand, following Kontsevich's philosophy of hidden smoothness \cite[\S 1.4]{Kontsevich}, one is led to develop intersection theory natively in derived algebraic geometry, so that phenomena like virtual fundamental classes become intrinsic.

We will construct an integral Chow homology theory $(\sX \to \sS) \mapsto \CH_*(\sX/\sS)$, bivariant in the sense of Fulton--MacPherson \cite{FultonMacPherson}, for derived Artin stacks.
Over a field $k$, the groups $\CH_*(\sX) := \CH_*(\sX/k)$ recover Kresch's Chow groups.\footnote{%
  when $\sX$ is a classical $1$-Artin stack of finite type with affine stabilizers, and $k$ is a perfect field.
  For $\sX=X$ a quasi-projective $k$-scheme, $\CH_*(X)$ recovers the usual Chow groups over arbitrary fields $k$.
}
Moreover, the $\CH_n(\sX/\sS)$ are realized as the zeroth homologies of complexes $\mbf{z}_n(\sX/\sS)$, with $\mbf{z}_n(\sX) := \mbf{z}_n(\sX/k)$ extending Bloch's cycle complexes \cite{Bloch} to stacks.

For a \emph{quasi-smooth} morphism of derived Artin stacks $\sX \to \sS$ of relative virtual dimension $d$, we also construct a \emph{bivariant fundamental class}
\[ [\sX/\sS] \in \CH_d(\sX/\sS), \]
which is compatible with composition and base change, and satisfies a Grothendieck--Riemann--Roch formula.
As we explain below, it gives rise in particular to virtual fundamental classes for Artin stacks, generalizing the construction of Behrend--Fantechi \cite{BehrendFantechi} in the Deligne--Mumford case.

We now explain how this formalism relates to previous approaches to intersection theory and virtual fundamental classes, and some of its applications.
We work over a base field $k$ to fix ideas.

\subsection{Quasi-smoothness}

Given a smooth $k$-scheme $M$ of dimension $m$, a section $s$ of a rank $r$ vector bundle $E$ cuts out a derived scheme $X = Z(s)$. Its structure sheaf is the Koszul complex of $s$, and its classical truncation $X_\cl$ is the usual zero locus of $s$, of dimension at least the \emph{virtual dimension} $\dimvir = m - r$ everywhere (with equality exactly when $X$ is already classical).

A derived Artin stack is \emph{quasi-smooth} if it is smooth-locally of this form, and a morphism $\sX \to \sY$ if it is so relative to $\sY$, smooth-locally on $\sX$ and $\sY$ (\ssecref{ssec:normal}).

\subsection{Fundamental cycles and refined Gysin maps}

A classical $k$-scheme $X$ is quasi-smooth if and only if it is lci, of dimension $\dim(X) = \dimvir$, and $[X] := [X/k] \in \CH_{\dimvir}(X)$ is then its \emph{fundamental cycle} in the sense of \cite[\S 1.5]{Fulton}.

Likewise, a morphism $f : X \to Y$ of classical $k$-schemes is quasi-smooth if and only if it is lci of relative dimension $\dim(X/Y) = \dimvir$, and $[X/Y] \in \CH_{\dimvir}(X/Y)$ then gives rise to refined Gysin maps
\[ f^!_V : \CH_*(V) \to \CH_{*+\dimvir}(X \fibprod_Y V) \]
for all $k$-schemes $V$ over $Y$.\footnote{%
  The reader may be familiar with the operational bivariant groups of Fulton \cite[Ch.~17]{Fulton}.
  While they have the same formal structure, the groups $\CH_{*}(X/Y)$ strictly refine those: they admit natural maps to Fulton's, which are typically neither injective nor surjective.
  Indeed, $\CH^*(X) := \CH_{-*}(X/X)$ admits natural cycle class maps e.g. to Betti cohomology, unlike the operational Chow groups (see \cite[Thm.~7]{TotaroLinear}).
}
These coincide with Fulton's refined Gysin maps, defined in \cite[\S 6.6]{Fulton} only for globally factorizable lci morphisms (see \examref{exam:fund/refined Gysin}).

\subsection{Virtual fundamental classes}

For a quasi-smooth derived Artin stack $\sX$ over $k$, the fundamental class $[\sX]$ determines a \emph{virtual fundamental class}
\begin{equation}\label{eq:intro/vir}
  [\sX]^\vir \in \CH_{\dimvir}(\sX_\cl) \simeq \CH_{\dimvir}(\sX)
\end{equation}
in the sense of Kontsevich \cite[\S 1.4]{Kontsevich}.
In the example $\sX=Z(s)$ above, $[Z(s)]^\vir$ is the localized top Chern class $c_{\mrm{top}}(E, s)$ of \cite[\S 14.1]{Fulton} (see \examref{exam:fund/localized Chern}).
For a quasi-smooth morphism $f : \sX \to \sY$, $[\sX/\sY]$ gives rise in particular to \emph{virtual Gysin pullbacks}
\[ f^!_{\vir} : \CH_*(\sY_\cl) \to \CH_{*+\dimvir}(\sX_\cl). \]

In the case of Deligne--Mumford stacks, resp. morphisms representable by Deligne--Mumford stacks, these recover the constructions of virtual fundamental classes and virtual pullbacks done in \cite{BehrendFantechi} and \cite{Manolache} (where the auxiliary perfect obstruction theories there are naturally induced by quasi-smooth derived structures).

When $\sY$ is smooth, the bivariant class $[\sX/\sY]$ is the virtual fundamental class of the total space: under $\CH_{\dimvir}(\sX/\sY) \simeq \CH_{\dimvir+\dim(\sY)}(\sX)$ it corresponds to $[\sX]^\vir$, whose Gysin restrictions to the derived fibres $\sX_y := \sX \fibprodR_\sY \{y\}$ recover the classes $[\sX_y]^\vir$.
When $\sY$ is singular, so that $[\sX]^\vir$ need not exist, the bivariant class $[\sX/\sY]$ still recovers the virtual classes of the fibres $\sX_y$ (see \rmkref{rmk:fund/higher rational equivalence}).

\subsection{Intersection products}

For $\sX$ a \emph{regular} Artin stack which is locally of finite type over a field $k$ and of dimension $d$, we construct in \ssecref{ssec:fund/absolute purity} an intersection product
\[ \CH_n(\sX) \otimes \CH_{n'}(\sX) \to \CH_{n+n'-d}(\sX) \]
generalizing the intersection product for smooth finite type $1$-Artin stacks with affine stabilizers constructed in \cite[\S 5.1]{Kresch}.
When $\sX$ is smooth, the intersection product can be described via fundamental classes of derived intersections (\ssecref{ssec:fund/Bezout}):
  \begin{equation}\label{eq:Bezout}
    [\sY] \cdot [\sZ] = [\sY \fibprodR_{\sX} \sZ]
  \end{equation}
for smooth closed substacks $\sY$ and $\sZ$.
When $\sY$ and $\sZ$ intersect transversely, so that $\sY \fibprodR_\sX \sZ$ is classical of the expected dimension, this formula collapses to \cite[Prop.~3]{GrothendieckIntersection}.
(See also \propref{prop:Bezout nonlci} for the case where $\sY$ and $\sZ$ are arbitrary.)

\subsection{Grothendieck--Riemann--Roch}

We extend the Baum--Fulton--MacPherson formulation of GRR \cite{BaumFultonMacPherson} from quasi-projective schemes to derived Artin stacks (\ssecref{ssec:fund/GRR}).
For $\sX$ quasi-smooth, we derive from this the formula $[\sX] = \Td_{\sX}^{-1} \cap \tau_{\sX}(\sO_\sX)$, answering a question of Toën (see \cite[Question~4.7]{ToenOverview}, \cite[Thm.~2.1.71]{RobaloHDR}).
Under $\CH_*(\sX)_\bQ \simeq \CH_*(\sX_\cl)_\bQ$ and $\G_0(\sX) \simeq \G_0(\sX_\cl)$, this becomes
\stepcounter{equation}%
\begin{equation}\label{eq:Kontsevich VFC}
    [\sX]^\vir = (\Td_{\sX}^\vir)^{-1} \cap \left( \sum_{i\in\Z} (-1)^i \cdot \tau_{\sX_\cl}\bigl(\pi_i(\sO_\sX)\bigr) \right)
\end{equation}
in $\prod_n \CH_n(\sX_\cl)_\Q$, where $\Td^\vir_\sX$ is the Todd class of the tangent complex restricted to $\sX_\cl$.
For schemes, the right-hand side of \eqref{eq:Kontsevich VFC} is precisely the formula proposed by Kontsevich in \cite[1.4.2]{Kontsevich}, and \eqref{eq:Kontsevich VFC} was established in that generality in \cite{CiocanFontanineKapranov} and \cite{FantechiGoettsche} (in characteristic zero).

\subsection{Higher Chow groups}

The Chow group $\CH_n(\sX/\sS)$ is the zeroth homology of a complex $\mbf{z}_n(\sX/\sS)$, whose homology groups define the \emph{higher} Chow groups $\CH_n(\sX/\sS, s)$.\footnote{%
For $\sS=\Spec(k)$, the complex $\mbf{z}_n(\sX) := \mbf{z}_n(\sX/k)$ can be regarded as an extension to stacks of Bloch's cycle complex (see \thmref{thm:Chow Bloch}).
}
As a point of $\mbf{z}_{\dimvir}(\sX/\sY)$, the bivariant fundamental class $[\sX/\sY]$ encodes, in particular, rational equivalences between the fibre classes $[\sX_y]^\vir$, $y\in\sY$, and higher compatibilities between them (see \rmkref{rmk:fund/higher rational equivalence}).

The intersection product on $\CH_*(\sX)$ discussed above lifts to an $E_\infty$-ring structure on $\bigoplus_n \mbf{z}_n(\sX)$ (\constrref{constr:integral product}), extending \cite[Cor.~5.7]{Bloch} from quasi-projective schemes.
Similarly, the Riemann--Roch map is $\pi_0$ of a map of spectra $\G(\sX) \to \prod_{n\in\Z} \mbf{z}_n(\sX)_\Q$ satisfying analogous compatibilities with proper direct image and Gysin pullbacks (\ssecref{ssec:fund/GRR}); this extends \cite[Thm.~4.1]{Gillet}, \cite[Thm.~9.1]{Bloch}, \cite[Part~I, Ch.~IV, Thm.~2.5.3]{LevineMixed}, and \cite[Thm.~6.3.1]{Riou} from the case of quasi-projective schemes.

\subsection{Applications}

This paper first appeared as a preprint in 2019.
We briefly survey some of the applications this formalism and its adaptations to other geometric contexts\footnote{%
  The constructions also apply in derived $C^\infty$-geometry, derived complex analytic geometry, and so on; see Footnote~\ref{fn:contexts}.
} have found since.

In enumerative geometry, the \emph{bivariant} refinement of fundamental classes is used in Pardon's recent proof of the MNOP conjecture \cite{PardonMNOP} (see also \cite{BehrendPardon}).
Our virtual fundamental class is used to define Gromov--Witten invariants for rigid analytic varieties in \cite{PortaYu}.
It is similarly expected to produce symplectic Gromov--Witten invariants, via the representability of derived moduli of pseudo-holomorphic curves \cite{PardonRepr,Steffens} and its anticipated extension to degenerating families of curves \cite[\S 6]{PardonRepr}; another approach to the same using our virtual fundamental class in liquid mathematics has been proposed by Scholze (see \cite{ScholzeLect24}).
The formalism of higher Chow groups of Artin stacks itself is used in \cite{PRSSLogTaut} to study logarithmic tautological rings of moduli of curves.

In geometric representation theory, our formalism of Gysin pullbacks gives a rigorous construction of the 2d cohomological Hall algebras of Kapranov--Vasserot \cite{KapranovVasserot}\footnote{%
  In \cite{KapranovVasserot}, the authors gave a construction of virtual pullbacks in a special situation that is sufficient for the construction of CoHAs, and observed that the construction would follow from a formalism of Borel--Moore homology for derived Artin stacks such as developed here.
  Note that associativity of the CoHA follows from the intrinsic construction of Gysin maps given here (via Theorems~\ref{thm:fund/functoriality} and \ref{thm:fund/base change}); the argument given in \cite{KapranovVasserot} relies on \cite[Prop.~3.3.3]{KapranovVasserot}, whose proof we were unable to follow (Lemma~3.3.4 is applied to a square that need not be cartesian).
}, as performed in \cite{DiaconescuPortaSala,SchiffmannVasserot,MellitMinetsSchiffmannVasserot,DPSSVNilpotent,DPSSVYangians,JindalLim}.
Moreover, the Gysin transformation \eqref{eq:intro/gys} below promotes this to a \emph{relative} or sheaf-level CoHA as in \cite[\S 2.4.5]{DHKSV} and \cite[\S 4.2]{dimredcoha}.

In arithmetic geometry, our virtual fundamental classes have been used to formulate and prove a function field analogue of Kudla's modularity conjecture for generating series of special cycles (see \cite{FYZHigherTheta,FengKhan,FYZModularity}).
They also enter, in the form of \eqref{eq:intro/gys}, the proof of the categorical Grothendieck--Lefschetz trace formula for Artin stacks in \cite{GaitsgoryVarshavsky}, used in \cite{AGKRRV} to recover automorphic functions from the geometric Langlands category as traces of Frobenius.

On the foundational side, \cite{constack} proves equivariant concentration for higher Chow groups of Artin stacks, and \cite{virloc} shows that our virtual classes satisfy the virtual localization formula, without global resolution hypotheses and in arbitrary characteristic.
A K-theoretic analogue of this formalism, producing virtual structure sheaves\footnote{%
  comparing in particular virtual structure sheaves in G-theory with derived structure sheaves; see \cite[\S 6]{kstack}
}, is in \cite{kstack}.
The paper \cite{equilisse} specializes motivic Borel--Moore homology theories to global quotient stacks and compares these with traditional definitions of equivariant Chow groups, G-theory, and algebraic bordism.
Finally, \cite{homology} shows how the machinery here also produces equivariant homology theories for stacks, such as those required in Joyce's vertex algebra formalism for wall-crossing \cite{Joyce,Bojko}.

\subsection{Chow homology as Borel--Moore homology}

We sketch the construction of Chow homology.
Following Grothendieck, the Borel--Moore homology of a locally compact space $X$ is the cohomology of the dualizing complex $a^!\Lambda$, where $a : X \to \mrm{pt}$.
For algebraic varieties we can define an analogous homology theory using the formalism of motivic sheaves.
For $a : X \to \Spec(k)$ locally of finite type we consider the mapping spectra
\[ \mbf{z}_n(X) := \Maps_{\SH(X)}(\un_X, a^!\Z^\mot\vb{-n}), \qquad \vb{-n} = (-n)[-2n], \]
with $\Z^\mot$ the integral motivic cohomology spectrum in the stable motivic homotopy category $\SH(k)$ (see \examref{exam:Q}).
This extends to Artin stacks $\sX$ by \emph{lisse extension} (see \cite{equilisse,weavelisse} and Appendix~\ref{sec:sixops}).
The resulting complexes $\mbf{z}_n(\sX)$ extend Bloch's cycle complexes \cite{Bloch} from quasi-projective schemes (see \thmref{thm:Chow Bloch}), and their zeroth homologies $\CH_n(\sX)$ agree with Kresch's Chow groups \cite{Kresch}, at least over a perfect field (see \thmref{thm:Chow}).\footnote{%
  For schemes, such comparisons were known in characteristic zero, and in characteristic $p>0$ when $X$ is smooth or up to $p$-torsion (see \cite[Thm.~19.1, Prop.~19.18]{MazzaWeibelVoevodsky}, \cite[Cor.~2]{VoevodskyIMRN}, \cite[Cor.~8.13]{CDIntegral}).
}
The desired operations on $\mbf{z}_*(\sX)$, like proper representable direct image, smooth Gysin, and intersection products for smooth Artin stacks, follow right away from properties of the six functor formalism established in Appendix~\ref{sec:sixops}.
With rational coefficients, we extend direct image to proper morphisms that are only representable by Deligne--Mumford stacks (as in \cite[\S 3.3]{Kresch}) or even by tame Artin stacks (see \ref{sssec:BM/proper}).

For $a : \sX \to \sS$, the relative complexes $\mbf{z}_n(\sX/\sS)$ are obtained via the mapping spectra
\[ \Maps_{\SH(\sX)}(\un_\sX, a^!\Z^\mot_\sS\vb{-n}). \]
Taking any motivic spectrum $\sF \in \SH(\sS)$ in place of $\Z^\mot_\sS$ yields more generally spectra of generalized Borel--Moore chains $\CBM(\sX/\sS; \sF)\vb{-n}$.
For example, $\sF = \MGL$ defines a bivariant algebraic bordism theory (\examref{exam:MGL_et}).
$\SH$ also contains quadratic refinements of these theories, such as the motivic spectrum representing Chow--Witt groups (see \cite[Cor.~4.3.8, Ex.~4.3.9]{DegliseFeldJin}), which following \cite{DegliseJinKhan} we twist by K-theory classes $v \in \K(\sX)$, setting $\CBM(\sX/\sS; \sF)\vb{v} := \Maps_{\SH(\sX)}(\un_\sX, a^!(\sF)\vb{v})$.\footnote{%
  For oriented theories like $\Z^\mot$ and $\MGL$, the Thom twists $\vb{v}$ collapse to $\sF\vb{n} \simeq \sF(n)[2n]$ for $n=\rk(v)$.
}

\subsection{Deformation to the normal bundle}

With the spectra $\CBM(\sX/\sS; \sF)\vb{v}$ at hand, bivariant fundamental classes are constructed by the technique of deformation to the normal bundle, as in the approaches of Verdier \cite[\S 3]{VerdierCI} and Fulton \cite[\S 5.2]{Fulton}.\footnote{%
  \label{fn:contexts}%
  The constructions in this paper are valid not just in the algebraic category, but also in derived $C^\infty$-geometry, derived complex analytic geometry, and so on.
  This is because they only rely on the normal deformation $D_{\sX/\sY}$, whose presentation via Weil restrictions in \ssecref{ssec:deformation} makes sense in great generality (cf. \cite[0.5.4]{weilres}), and on the existence of a suitable six functor formalism (that is, a topological weave in the sense of \cite{weaves}) on stacks.
  For example, the six functor formalism on sheaves of complexes or spectra has been extended to topological stacks in \cite[\S 8.3]{weavelisse}, and suffices for the construction of virtual fundamental classes of derived $C^\infty$-stacks or derived complex analytic stacks.
}

For a quasi-smooth morphism $f : \sX \to \sY$, the normal bundle $N_{\sX/\sY}$ is a vector bundle stack over $\sX$ which, when $\sX$ and $\sY$ have perfect cotangent complexes, fits in an exact triangle of derived vector bundles over $\sX$,
\[ T_\sX \to f^*T_\sY \to N_{\sX/\sY}, \]
see \defnref{defn:normal bundle}.
The derived analogue of deformation to the normal bundle is a family $D_{\sX/\sY}$ over $\A^1$ degenerating $\sY$ to $N_{\sX/\sY}$ (\ssecref{ssec:deformation}).
For a \emph{closed immersion} of smooth schemes $X \hookrightarrow Y$, $N_{X/Y}$ is the usual normal bundle and $D_{X/Y}$ is the deformation to the normal bundle of Verdier \cite[\S 2]{VerdierCI}.
For a \emph{smooth} morphism of schemes $X \to Y$, $N_{X/Y}$ is instead the classifying stack of the relative tangent bundle, $[X/T_{X/Y}] = B_X(T_{X/Y})$.

The boundary map of the localization triangle for $N_{\sX/\sY} \subseteq D_{\sX/\sY}$ gives rise to a \emph{specialization} map $\mbf{z}_n(\sY/\sY) \to \mbf{z}_n(N_{\sX/\sY}/\sY)$ (\constrref{constr:fund/sp}), whose target becomes $\mbf{z}_{n+\dimvir}(\sX/\sY)$ by homotopy invariance for vector bundle stacks (\propref{prop:BM/homotopy}).
The specialization of the unit $1 \in \mbf{z}_0(\sY/\sY)$ thus determines the desired class $[\sX/\sY] \in \mbf{z}_{\dimvir}(\sX/\sY)$.

The same construction works for general coefficients $\sF$, taking values in $\CBM(\sX/\sY; \sF)\vb{-\sL_{\sX/\sY}}$ if $\sF$ is not oriented.
For the spectrum representing Chow--Witt groups, which is only $\mrm{SL}$-oriented, and $X$ a quasi-smooth derived scheme of finite type over $k$, one gets a virtual fundamental class in the Chow--Witt group
\[ [X]^{\vir} \in \widetilde{\on{CH}}_{\dimvir}(X_\cl; \det(\sL_X)|_{X_\cl}). \]
For example, for $X$ proper of virtual dimension zero, a trivialization of $\det(\sL_X)|_{X_\cl}$ yields a degree in the Grothendieck--Witt group $\on{GW}(k)$.

In terms of the six operations, $[\sX/\sY] \in \CBM(\sX/\sY; \sF)\vb{-\sL_{\sX/\sY}}$ amounts to the data of a natural transformation
\begin{equation}\label{eq:intro/gys}
	\on{gys}_f : f^*(-) \vb{\sL_{\sX/\sY}} \longrightarrow f^!(-),
\end{equation}
see \rmkref{rmk:purity}, which specializes to the Poincar\'e duality isomorphism $f^*(-)\vb{\sL_{\sX/\sY}} \simeq f^!(-)$ when $f$ is smooth.

\subsection{Prior work}

The yoga of fundamental classes in motivic bivariant theories was developed in \cite{DegliseBivariant} and \cite{DegliseJinKhan}.
This paper extends these constructions on the one hand from classical to derived algebraic geometry, and on the other from schemes to algebraic stacks.
The usage of derived deformation to the normal bundle also allows us to drop the \emph{smoothability} hypotheses on lci morphisms there, and to treat smooth morphisms and closed immersions simultaneously.
In particular, Bloch's higher Chow groups acquire Gysin maps along arbitrary lci morphisms and with integral coefficients (thanks to \thmref{thm:Chow Bloch}); compare with the $\Z[1/p]$-linear statements in \cite[Ex.~3.3.6]{DegliseBivariant} and \cite[Ex.~4.4.4]{DegliseJinKhan}.

Lowrey--Sch\"urg \cite{LowreySchuerg} studied virtual fundamental classes of quasi-projective derived schemes in algebraic cobordism.
In the older language of dg-manifolds over a field of characteristic zero, Ciocan-Fontanine--Kapranov \cite{CiocanFontanineKapranov} studied them in rational Chow groups and G-theory.
The ``non-transverse B\'ezout formula'' \eqref{eq:Bezout} was inspired by a similar formula in Betti cohomology announced by Lurie \cite{LurieThesis}.

A perfect obstruction theory in the sense of Behrend--Fantechi \cite{BehrendFantechi} is a classical shadow of a quasi-smooth derived structure, and suffices for the construction of virtual fundamental classes on Deligne--Mumford stacks.
After this work appeared, the intrinsic normal cone was extended to Artin stacks by Aranha--Pstr{\fontencoding{T1}\selectfont\k{a}}gowski \cite{AranhaPstragowski}.
As explained in \ssecref{ssec:fund/BF}, a perfect obstruction theory on an Artin stack in the sense of \cite[Def.~8.1]{AranhaPstragowski} is enough data to produce bivariant virtual fundamental classes.

\subsection{Changes since 2019 version}

For the reader's convenience, we list the substantial changes compared with the original preprint of this paper (\arXiv{1909.01332v1}):
\begin{itemize}
  \item
  Chow homology is now constructed with integral coefficients.
  Previously, we had only been able to extend the six operations to stacks under an \'etale descent assumption, which holds for rational but not integral motivic sheaves.

  \item
  A comparison between Chow homology and Kresch's Chow groups, with integral coefficients, is now proved in Appendix~\ref{sec:Chow}.
  Previously, the comparison was stated with rational coefficients and with no proof.

  \item
  The construction of virtual fundamental classes in non-oriented motivic Borel--Moore homology theories (hence e.g. in Chow--Witt groups) is new.

  \item
  The construction of intersection products on equicharacteristic regular algebraic stacks (\constrref{constr:integral product}) is new.

  \item
  For schemes, the extension of the main construction of \cite{DegliseJinKhan} to arbitrary lci morphisms is new.
\end{itemize}

The numbering of the first version has been preserved wherever possible, for the sake of existing references.

\subsection*{Conventions}

We use the language of \inftyCats as developed in \cite{LurieHTT}.
By a \emph{complex} of $\Lambda$-modules, we mean an object of the derived \inftyCat $\on{D}(\Lambda)$, rather than a strict chain complex.
By a \emph{point} $x$ of a complex $K$, we mean a morphism $x : \Lambda \to K$ in $\on{D}(\Lambda)$, or equivalently a point of the underlying \inftyGrpd if $K$ is connective; we write simply $x \in K$.

\subsection*{Acknowledgments}

During the very long gestation period of this paper, I benefited from helpful discussions with Denis-Charles Cisinski, Frédéric Déglise, Marc Hoyois, Fangzhou Jin, Maxim Kontsevich, Marc Levine, Mauro Porta, Niranjan Ramachandran, Charanya Ravi, Marco Robalo, David Rydh, and Tony Yue Yu.
Thanks to the organizers of the June 2019 summer school ``New perspectives in Gromov--Witten theory'' in Paris which made some of the above conversations possible and where I was inspired to finally write up these results.
Thanks to the Institute for Advanced Study which hosted me in July 2019 while the first draft of this paper was being finished.
I also acknowledge support from the grants AS-CDA-112-M01 (Academia Sinica) and NSTC 115-2628-M-001-001.

\changelocaltocdepth{2}
\section{Deformation to the normal bundle}
\label{sec:normal}

  \subsection{Stacks}

    In this paper, we define a stack to be a ``higher stack'' in the sense of \cite{HirschowitzSimpson}.
    That is, it is a functor
      \begin{equation*}
        R \mapsto \sX(R)
      \end{equation*}
    assigning to any commutative ring $R$ an $\infty$-groupoid $\sX(R)$ of $R$-valued points, and satisfying hyperdescent with respect to the étale topology (in the sense of $\infty$-category theory, see e.g. \cite[p.~183]{ToenSurvey}).

    We say $\sX$ is \emph{$0$-Artin} if it is (representable by) an algebraic space.
    We define $k$-Artin stacks inductively, following \cite[\S 3.1]{ToenLectures}.

    For an integer $k\ge 0$, a morphism $f : \sX \to \sY$ is \emph{$k$-representable} if for every $k$-Artin stack $\sY'$ and every morphism $\sY' \to \sY$, the fibred product $\sX \fibprod_\sY \sY'$ is $k$-Artin.
    A $k$-representable morphism $f$ is \emph{smooth} if for every scheme $Y$, every morphism $Y \to \sY$, and every smooth atlas $X \to \sX \fibprod_\sY Y$, the composite $X \to \sX\fibprod_\sY Y \to Y$ is a smooth morphism of schemes.
    A stack $\sX$ is \emph{$(k+1)$-Artin} if its diagonal $\sX \to \sX \times \sX$ is representable by $k$-Artin stacks, and there exists a scheme $X$ and a morphism $X \to \sX$ (automatically $k$-representable) which is smooth and surjective.
    The morphism $X \to \sX$ is called a \emph{smooth atlas} for $\sX$.

    A $k$-Artin stack $\sX$ always takes values in $k$-groupoids: for every commutative ring $R$, the $\infty$-groupoid $\sX(R)$ is $k$-truncated.
    We say a stack is \emph{Artin} if it is $k$-Artin for some $k\ge 0$.
    Artin stacks in this sense form an \inftyCat, whose full subcategory spanned by $1$-Artin stacks is equivalent to the $(2,1)$-category of Artin stacks in the usual sense.

    Now replace the category of commutative rings by its nonabelian derived \inftyCat, i.e., the \inftyCat of simplicial commutative rings.
    This is the natural target for derived functors on the nonabelian category of commutative rings, such as the derived tensor product.
    A simplicial commutative ring $R$ has an underlying ordinary commutative ring $\pi_0(R)$ as well as $\pi_0(R)$-modules $\pi_i(R)$.
    We say $R$ is \emph{discrete} if $\pi_i(R) = 0$ for all $i>0$ (i.e., $R \simeq \pi_0(R)$); the discrete simplicial commutative rings span a full subcategory equivalent to the ordinary category of commutative rings.
    The notions of étale and smooth homomorphisms admit natural extensions to simplicial commutative rings.
    See \cite[Chap.~25]{SAG} or \cite[\S 4]{ToenLectures}.

    A \emph{derived stack} $\sX$ is a functor $R \mapsto \sX(R)$, assigning an $\infty$-groupoid of $R$-points to every simplicial commutative ring $R$, that satisfies étale hyperdescent.
    Derived $k$-Artin and Artin stacks are defined following the pattern outlined above, see e.g. \cite[\S 5.2]{ToenLectures} for details.

  \subsection{Vector bundle stacks}

    Let $\sX$ be a \dA stack and $\sE$ a perfect complex on $\sX$ of Tor-amplitude $\le 0$.
    The associated \emph{vector bundle stack}
      \begin{equation*}
        \pi : \bV_\sX(\sE) \to \sX
      \end{equation*}
    is the moduli stack of co-sections of $\sE$.
    That is, for any affine derived scheme $S$ over $\sX$, the $\infty$-groupoid of $\sX$-morphisms $S \to \bV_\sX(\sE)$ is naturally equivalent to the $\infty$-groupoid of $\sO_S$-linear morphisms of perfect complexes $\sE|_S \to \sO_S$.

    If $\sE$ is of Tor-amplitude $[-k,0]$, then $\bV_\sX(\sE)$ is a smooth $k$-Artin derived stack over $\sX$ of relative dimension $d$, where $d$ is the virtual rank of $\sE$.
    See \cite[Subsect.~3.3, p.~201]{ToenSurvey}.

  \subsection{Normal bundles}
  \label{ssec:normal}

    The normal bundle is a derived version of the relative intrinsic normal cone of \cite{BehrendFantechi}.

    A morphism $f : \sX \to \sY$ of \dA stacks is \emph{quasi-smooth} if it is locally of finite presentation and the relative cotangent complex $\sL_{\sX/\sY}$ is of Tor-amplitude $(-\infty, 1]$.
    Note that we use homological grading: this means that, for every discrete quasi-coherent sheaf $\sE$ on $\sX$, we have
      \begin{equation*}
        \pi_{i}(\sL_{\sX/\sY} \otimes^\bL_{\sO_\sX} \sE) = 0
      \end{equation*}
    for $i > 1$.
    If $\sX=X$ and $\sY=Y$ are derived schemes, this is equivalent to the following condition: Zariski-locally on $X$, $f$ factors through a smooth morphism $M \to Y$ and a morphism $X \to M$ which exhibits $X$ as the derived zero-locus of some functions on $M$ \cite[2.3.14]{KhanRydh}.
    If $f : \sX \to \sY$ is $k$-representable, then it is quasi-smooth if and only if for every derived scheme $Y$, every morphism $Y \to \sY$, and every smooth atlas $X \to \sX \fibprod_\sY Y$, the composite $X \to \sX\fibprod_\sY Y \to Y$ is a quasi-smooth morphism of derived schemes.
    The \emph{relative virtual dimension} of a quasi-smooth morphism $f : \sX \to \sY$ is
      \begin{equation*}
        \dimvir(\sX/\sY) := \rk(\sL_{\sX/\sY}),
      \end{equation*}
    the virtual rank (Euler characteristic) of the relative cotangent complex.

    Let $f : \sX \to \sY$ be a $k$-representable quasi-smooth morphism.
    The cotangent complex $\sL_{\sX/\sY}$ is perfect of Tor-amplitude $[-k,1]$, so the associated vector bundle stack $\bV_\sX(\sL_{\sX/\sY}[-1])$ is a smooth $(k+1)$-Artin stack of relative virtual dimension $-\dimvir(\sX/\sY)$.

    \begin{defn}\label{defn:normal bundle}
    Let $f : \sX \to \sY$ be a quasi-smooth morphism of \dA stacks.
    The \emph{normal bundle} is the vector bundle stack
      \begin{equation*}
        N_{\sX/\sY} = \bV_\sX(\sL_{\sX/\sY}[-1]) \to \sX.
      \end{equation*}
    \end{defn}

    If $f$ is a closed immersion, then $\sL_{\sX/\sY}[-1]$ is of Tor-amplitude $[0,0]$, and $N_{\sX/\sY}$ is a vector bundle.
    If $f$ is smooth, then $\sL_{\sX/\sY}[-1]$ is of Tor-amplitude $(-\infty,-1]$, and the normal bundle is the classifying stack of the tangent bundle $T_{\sX/\sY}$.
    If $f$ factors through a closed immersion $i : \sX \to \sY'$ and a smooth morphism $p : \sY' \to \sY$, then the normal bundle is the quotient
      \begin{equation*}
        N_{\sX/\sY} = [N_{\sX/\sY'}/i^*T_{\sY'/\sY}].
      \end{equation*}

    \begin{prop}\leavevmode
    \label{prop:normal bundle}

    \begin{thmlist}
      \item
      The construction $N_{\sX/\sY} \to \sX$ is stable under derived arbitrary base change in $\sX$.
      That is, for any homotopy cartesian square of \dA stacks
        \begin{equation*}
          \begin{tikzcd}
            \sX' \ar{r}{f'}\ar{d}
              & \sY' \ar{d}
            \\
            \sX \ar{r}{f}
              & \sY
          \end{tikzcd}
        \end{equation*}
      with $f$ quasi-smooth, there is a canonical isomorphism
        \begin{equation*}
          N_{\sX/\sY} \fibprodR_\sX \sX' \to N_{\sX'/\sY'}
        \end{equation*}
      of \dA stacks over $\sX'$.

      \item\label{item:normal bundle/presentation}
      Suppose given a commutative square
        \begin{equation*}
          \begin{tikzcd}
            X \ar{r}{i}\ar{d}{p}
              & Y \ar{d}{q}
            \\
            \sX \ar{r}{f}
              & \sY
          \end{tikzcd}
        \end{equation*}
      with $f$ quasi-smooth, $p$ and $q$ smooth surjections with $X$ and $Y$ schematic, and $i$ a quasi-smooth closed immersion.
      Then $N_{\sX/\sY}$ is the quotient of the groupoid
        \begin{equation*}
          N_{\Cech(X/\sX)_\bullet/\Cech(Y/\sY)_\bullet}
            := \left[ \cdots \rightrightrightarrows N_{X\fibprodR_\sX X / Y\fibprodR_\sY Y} \rightrightarrows N_{X/Y} \right],
        \end{equation*}
      i.e., the geometric realization of this simplicial diagram.
    \end{thmlist}
    \end{prop}

    \begin{proof}
    The first claim follows from the fact that the cotangent complex is stable under derived base change \cite[Prop.~3.2.10]{LurieThesis}.
    The second follows from the fact that the cotangent complex satisfies descent for smooth surjections \cite[Cor.~2.7]{Bhatt}.
    \end{proof}

  \subsection{Normal deformations}
  \label{ssec:deformation}

    For any quasi-smooth morphism $f : \sX \to \sY$, there is a canonical $\A^1$-degeneration of $f$ to the zero section $0 : \sX \to N_{\sX/\sY}$, the \emph{normal deformation}, generalizing the classical construction of Verdier.
    This construction is joint with D.~Rydh.

    \begin{thm}\label{thm:deformation/main}
    Let $f : \sX \to \sY$ be a quasi-smooth morphism of \dA stacks.

    \begin{thmlist}
      \item
      There exists a quasi-smooth \dA stack $D_{\sX/\sY}$ over $\sY\times\A^1$, and a quasi-smooth morphism
        \begin{equation*}
          \sX\times\A^1 \to D_{\sX/\sY}
        \end{equation*}
      over $\sY \times \A^1$.
      The fibre over $\bG_m = \A^1\setminus\{0\}$ is the quasi-smooth morphism $\sX\times\bG_m \to \sY\times\bG_m$ and the fibre over $\{0\}$ is the quasi-smooth morphism $0 : \sX \to N_{\sX/\sY}$.

      \item\label{item:deformation/main/base change}
      The construction $D_{\sX/\sY} \to \sY$ is stable under arbitrary derived base change in $\sY$.
    \end{thmlist}
    \end{thm}

    In the case where $f$ is a closed immersion, $D_{\sX/\sY}$ was already constructed in \cite[Thm.~4.1.13]{KhanRydh}.
    For a general quasi-smooth morphism with a presentation as in \propref{prop:normal bundle}\itemref{item:normal bundle/presentation}, it can be described as the quotient of the groupoid
      \begin{equation*}
        D_{\Cech(X/\sX)_\bullet/\Cech(Y/\sY)_\bullet}
          := \left[ \cdots \rightrightrightarrows D_{X\fibprodR_\sX X / Y\fibprodR_\sY Y} \rightrightarrows D_{X/Y} \right].
      \end{equation*}
    Without choosing a presentation, it can be described simply as the Weil restriction
      \begin{equation*}
        D_{\sX/\sY} = \Res_{\sY/\sY\times\A^1}(\sX)
      \end{equation*}
    of $\sX$ along $\sY = \sY\times\{0\} \to \sY\times\A^1$.
    Details on this perspective, and a generalization to the non-quasi-smooth case, have appeared in \cite{weilres}.

\section{Motivic Borel--Moore homology theories}
\label{sec:BM}

    In this section we construct, given a ``coefficient'' $\sF$ over a \dA stack $\sS$, a (relative) Borel--Moore homology theory with coefficients in $\sF$.
    The construction requires a formalism of six operations on \dA stacks such as that developed in Appendix~\ref{sec:sixops}.

    More precisely, we fix a topological weave $\D$ on \da spaces in the sense of \cite[Def.~3.15]{weaves}, extended to \dA stacks by lisse extension \cite[Cor.~7.2.4]{weavelisse}, and an object $\sF \in \D(\sS)$.
    For example, we may take the weave $\D = \SH$ of motivic spectra, discussed in Appendix~\ref{sec:sixops}, and $\sF$ to be the integral motivic cohomology spectrum or the algebraic cobordism spectrum.
    Or, we could take the weave of Betti sheaves or mixed Hodge modules (over the complex numbers) or \'etale sheaves (with $\bZ/\ell^n\bZ$-coefficients, where $\ell$ is coprime to the characteristic of the base field), and $\sF$ to be a constant sheaf; see \cite[7.1.2]{weavelisse}.

    We write $\un_\sS \in \D(\sS)$ for the monoidal unit, and adopt the notation
    \[ \sF\vb{v}, \ v \in \K(\sS), \qquad \sF\vb{n} \simeq \sF(n)[2n], \ n\in\Z, \]
    for Thom twists (see \cite[Def.~2.57]{weaves}, \eqref{eq:Sigma^E}), and similarly for any locally constant function $n \in \H^0_{\mrm{Zar}}(\sS, \bZ)$.

    \subsection{Definition and examples}
    \label{ssec:BM/definition}

    \begin{defn}\label{defn:BM}
    Let $\sX$ be a \dA stack locally of finite type over $\sS$, with structural morphism $a : \sX \to \sS$, and let $\sF \in \D(\sS)$.
    \begin{thmlist}
      \item
      The \emph{Borel--Moore spectrum} with coefficients in $\sF$ is the mapping spectrum
        \begin{equation}
          \CBM(\sX/\sS; \sF) := \Maps_{\D(\sX)}(\un_\sX, a^!\sF).
        \end{equation}
      For $v \in \K(\sX)$ we also set $\CBM(\sX/\sS; \sF)\vb{v} := \Maps_{\D(\sX)}(\un_\sX, a^!(\sF)\vb{v})$.

      \item
      The \emph{cohomology spectrum} with coefficients in $\sF$ is
        \begin{equation}
          \Ccoh(\sX; \sF) := \Maps_{\D(\sX)}(\un_\sX, a^*\sF).
        \end{equation}
      For $v \in \K(\sX)$ we also set $\Ccoh(\sX; \sF)\vb{v} := \Maps_{\D(\sX)}(\un_\sX, a^*(\sF)\vb{v})$.
    \end{thmlist}
    \end{defn}

    For $\sF$ oriented in the sense of \cite[Def.~4.4.1]{DegliseJinKhan}, as in our main examples, a choice of orientation provides Thom isomorphisms $\sF\vb{v} \simeq \sF\vb{\rk(v)}$.

    The identification $\Ccoh(\sX; \sF) \simeq \CBM(\sX/\sX; \sF)$ allows us to pass freely from Borel--Moore homology statements to their cohomological counterparts, which is why we generally stick with the former perspective.
    For an immersion $i : \sY \to \sX$, we have also cohomology with support:
      \begin{equation}\label{eq:cohomology with support}
        \Ccoh_\sY(\sX; \sF) := \CBM(\sY/\sX; \sF).
      \end{equation}

    \begin{defn}
    We write, for all $r,s\in\Z$,
      \begin{gather*}
        \H^\BM_s(\sX/\sS; \sF(-r)) := \pi_s\, \CBM(\sX/\sS; \sF(-r)) \simeq \Hom(\un_\sS(r)[s], a_*a^!\sF),
        \\
        \H^{s}(\sX, \sF(r)) := \pi_{-s}\, \Ccoh(\sX; \sF(r)) \simeq \Hom(\un_\sS, a_*a^*\sF(r)[s]).
      \end{gather*}
    \end{defn}

    \begin{rmk}\label{rmk:Gamma^BM as limit}
    Let $\sS = S$ be a \da space and $\sX$ a locally of finite type \dA stack over $S$.
    The formula \eqref{eq:SHet limit} implies that the Borel--Moore spectrum $\CBM(\sX/S; \sF)$ can be computed as the homotopy limit
      \begin{equation}
        \CBM(\sX/S; \sF) \simeq \lim_u \CBM(X/S; \sF)\vb{-\sL_{u}}
      \end{equation}
    over the $\infty$-category of pairs $(X,u)$ where $X$ is a scheme and $u : X \to \sX$ a smooth morphism with relative cotangent complex $\sL_u$.
    Similarly, the cohomology spectrum is computed as the homotopy limit
      \begin{equation}
        \Ccoh(\sX; \sF) \simeq \lim_u \Ccoh(X; \sF).
      \end{equation}
    \refstepcounter{equation}\label{eq:Gamma^BM as totalization}
    \end{rmk}

    \begin{exam}\label{exam:Q}
    Let $\Z^\mot \in \SH(\Spec(\Z))$ denote the integral motivic cohomology spectrum over $\Spec(\Z)$, constructed in \cite{Spitzweck}.
    More generally for an abelian group $A$, we consider $A^\mot := \Z^\mot \otimes_{\Z} A$ (cf. \cite[Cor.~10.4]{Spitzweck}).
    We write $A^\mot_\sS$ for its inverse image to any \dA stack $\sS$.
    The complexes (= $\Z$-linear spectra)
      \begin{equation*}
        \CBM(\sX/\sS; A^\mot)\vb{-n},
        \quad
        \Ccoh(\sX; A^\mot)\vb{n}
      \end{equation*}
    compute \emph{motivic Borel--Moore homology} and \emph{motivic cohomology} with coefficients in $A$.

    We write, for all $n \in \Z$,
      \begin{equation*}
        \mbf{z}_n(\sX/\sS)_A := \CBM(\sX/\sS; A^\mot)\vb{-n},
        \quad
        \mbf{z}^n(\sX)_A := \Ccoh(\sX; A^\mot)\vb{n},
      \end{equation*}
    and for all $s \in \Z$,
      \begin{equation*}
        \CH_n(\sX/\sS, s)_A := \pi_s (\mbf{z}_n(\sX/\sS)_A),
        \quad
        \CH^n(\sX, s)_A := \pi_s (\mbf{z}^n(\sX)_A).
      \end{equation*}
    We omit $s$ when $s = 0$, $A$ when $A = \Z$, and $/\sS$ when $\sS$ is the spectrum of a field $k$.
    For example, $\CH_n(\sX, s) = \H^\BM_{2n+s}(\sX; \Z^\mot(-n))$ and $\CH^n(\sX, s) = \H^{2n-s}(\sX, \Z^\mot(n))$.
    The graded $E_\infty$-ring structure on the spectra $\Z^\mot\vb{n}$, $n \in \Z$ (see \cite[Thm.~C.2]{Spitzweck}) gives rise to an $E_\infty$-ring structure on $\bigoplus_n \mbf{z}^n(\sX)$.

    We call $\CH_n(\sX/\sS, s)_A$ and $\CH^n(\sX, s)_A$ the higher \emph{bivariant Chow homology} and \emph{Chow cohomology} groups, respectively.
    If $\sX=X$ is a quasi-projective $k$-scheme, then $\mbf{z}_n(X)$ is isomorphic to Bloch's cycle complex $z_n(X, \bullet)$ (\thmref{thm:Chow Bloch}); in particular, $\CH_n(X, s)$ are Bloch's higher Chow groups.
    If $k$ is perfect and $\sX$ is a $1$-Artin stack of finite type over $k$, with quasi-compact separated diagonal and affine stabilizers, then $\CH_n(\sX)$ is canonically isomorphic to Kresch's Chow group of $\sX$ (\thmref{thm:Chow}), and $\CH_n(\sX)_\Lambda \simeq \CH_n(\sX) \otimes \Lambda$ for every commutative ring $\Lambda$ (\propref{prop:Chow coefficients}).

    \end{exam}

    \refstepcounter{equation}\label{exam:Z_et}

    \begin{exam}\label{exam:MGL_et}
    Let $\MGL$ denote Voevodsky's algebraic cobordism spectrum over $\Spec(\Z)$.
    If $X$ is a smooth algebraic space over a perfect field $k$, then the cohomology groups $\H^{2n}(X, \MGL(n))$ are computed for $n\ge 0$ by the Nisnevich hypercohomology of a certain presheaf of spectra built out of finite quasi-smooth derived schemes over $X$ (see \cite{EHKSY3}).
    If $k$ is of characteristic zero, then they are identified with Levine--Morel's algebraic cobordism $\Omega^n(X)$, and moreover the Borel--Moore homology groups $\H^\BM_{2n}(X; \MGL(-n))$ are identified with $\Omega_n(X)$ for all $n\in\Z$, also for $X$ singular, by \cite{LevineMGL}.

    For a \dA stack $\sS$, let $\MGL_\sS \in \SH(\sS)$ denote the inverse image.
    This gives algebraic bordism and cobordism spectra for \dA stacks
      \begin{equation*}
        \CBM(\sX/\sS; \MGL),
        \quad
        \Ccoh(\sX; \MGL).
      \end{equation*}

    The rationalization $\MGL_\Q$ is identified with
      \begin{equation*}
        \MGL_\Q \simeq \Q^\mot[c_1,c_2,\ldots],
      \end{equation*}
    where $c_i$ is a generator of bidegree $(2i,i)$, by \cite{NaumannSpitzweckOstvaer}.
    There are canonical maps
      \begin{equation*}
        \CBM(\sX/\sS; \MGL)
          \to \CBM(\sX/\sS; \MGL_\Q)
          \to \CBM(\sX/\sS; \Q^\mot)
      \end{equation*}
    for all $\sX$ locally of finite type over $\sS$.
    \end{exam}

    \begin{exam}\label{exam:KGL_et}
    Let $\KGL^\lisse_{\sS}$ denote the inverse image of the algebraic K-theory spectrum along the structural morphism $\sS \to \Spec(\Z)$.
    Assuming that $\sS$ is a regular (classical) stack, such as the spectrum of $\Z$ or a field, the Borel--Moore homology represented by $\KGL^\lisse_{\sS}$ coincides with lisse-extended $G$-theory $\G^\lisse$ in the sense of \cite[\S 7]{equilisse}, and the proper covariance and smooth Gysin maps are compatible with the respective intrinsic operations in $G$-theory (see \cite[Cor.~3.3.7]{JinG}, \cite[Thm.~7.9]{equilisse}).
    Note that in this case the formula of \rmkref{rmk:Gamma^BM as limit} simplifies via Bott periodicity isomorphisms $\KGL_\sS^\lisse\vb{v} \simeq \KGL^\lisse_\sS$ for all $v\in\K(\sS)$.
    \end{exam}

    \refstepcounter{equation}\label{rmk:BM for asps}

  \subsection{Basic operations}
  \label{ssec:BM/operations}

    The formalism of six operations (see Appendix~\ref{sec:sixops}) immediately yields the following structure on Borel--Moore spectra.
    Here $\sF$ is any \emph{multiplicative} coefficient defined over $\sS$, i.e. an $E_\infty$-algebra in $\D(\sS)$.
    In particular there is a unit element
      \begin{equation*}
        1 \in \H^\BM_{0}(\sX/\sX; \sF) = \H^0(\sX, \sF)
      \end{equation*}
    induced by the unit $\eta_\sF : \un_\sS \to \sF$.

    \sssec{Proper direct image}\label{sssec:BM/proper}
    If $f : \sX \to \sY$ is a proper representable morphism of \dA stacks locally of finite type over $\sS$, then there are functorial direct image maps of spectra
      \begin{equation*}
        f_* : \CBM(\sX/\sS; \sF) \to \CBM(\sY/\sS; \sF).
      \end{equation*}
    These are induced by the co-unit $f_*f^! \simeq f_!f^! \to \id$.

    For example, for $\sF = \Z^\mot$ this gives maps $f_* : \mbf{z}_n(\sX/\sS) \to \mbf{z}_n(\sY/\sS)$ (\examref{exam:Q}).
    If $\sF$ belongs to $\SH_{\Q,+}(\sS)$, e.g., $\sF$ is $\Q^\mot$ or $\MGL_\Q$, then this extends to proper morphisms $f : \sX \to \sY$ that are representable by \dDM stacks, or by tame Artin stacks when $\sY$ is equicharacteristic (see \thmref{thm:f_!=f_* for X DM} and \examref{exam:finite parametrization}).

    \sssec{Smooth contravariance}\label{sssec:BM/Gysin}
    If $f : \sX \to \sY$ is a smooth morphism between \dA stacks locally of finite type over $\sS$, then there are functorial Gysin pull-back maps
      \begin{equation*}
        f^! : \CBM(\sY/\sS; \sF) \to \CBM(\sX/\sS; \sF)\vb{-\sL_{\sX/\sY}}.
      \end{equation*}
    These are compatible with proper direct images by a base change formula.
    They are induced by the co-trace transformation $\id \to f_*(f^!(-)\vb{-\sL_{\sX/\sY}})$, right transpose of the Poincar\'e duality isomorphism $f^*(-)\vb{\sL_{\sX/\sY}} \simeq f^!$ (\thmref{thm:sixops/stacks/purity}).

    For example, for $\sF = \Z^\mot$, this gives (using the canonical orientation) maps $f^! : \mbf{z}_n(\sY/\sS) \to \mbf{z}_{n+d}(\sX/\sS)$, where $d$ is the relative dimension of $f$.

    \sssec{Change of base}\label{sssec:BM/base change}
    If $f : \sT \to \sS$ is a morphism of \dA stacks and $\sX$ is a \dA stack locally of finite type over $\sS$, there are change of base maps
      \begin{equation*}
        f^* : \CBM(\sX/\sS; \sF) \to \CBM(\sX_\sT/\sT; \sF),
      \end{equation*}
    where $\sX_\sT = \sX \fibprodR_\sS \sT$ is the derived fibred product.\footnote{%
        Note that $f^*$ always denotes contravariant functoriality in the \emph{base} (change of base homomorphisms, \ref{sssec:BM/base change}), while $f^!$ denotes contravariant functoriality in the \emph{source} (Gysin homomorphisms, \ref{sssec:BM/Gysin}).
        Potentially the notation also clashes with that of the six operations (Appendix~\ref{sec:sixops}), but there should be no risk of confusion.
    }
    More generally, for any commutative square
      \begin{equation*}
        \begin{tikzcd}
          \sY \ar{r}\ar{d}\ar[phantom]{rd}{\scriptstyle\Delta}
            & \sT \ar{d}{f}
          \\
          \sX \ar{r}
            & \sS
        \end{tikzcd}
      \end{equation*}
    which is cartesian on classical truncations, there are maps
      \begin{equation*}
        f^*_\Delta : \CBM(\sX/\sS; \sF) \to \CBM(\sY/\sT; \sF).
      \end{equation*}
    These are induced by the unit map $\id \to f_*f^*$ and the base change formula (\thmref{thm:sixops/stacks/main}).

    For example, for $\sF = \Z^\mot$ they give maps $f^*_\Delta : \mbf{z}_n(\sX/\sS) \to \mbf{z}_n(\sY/\sT)$.

    \refstepcounter{equation}

    \sssec{Top Chern class}
    Let $\sE$ be a finite locally free sheaf of rank $r$ on a \dA stack $\sX$ over $\sS$.
    Then there is an Euler class
      \begin{equation*}
        e(\sE) \in \Ccoh(\sX; \sF)\vb{\sE}.
      \end{equation*}
    This is induced by the Euler transformation $\id \to (-)\vb{\sE}$ (\constrref{constr:eul}).

    When $\sF$ is oriented, this becomes the top Chern class $c_r(\sE) \in \Ccoh(\sX; \sF)\vb{r}$.
    Moreover, one also has the Chern classes $c_i(\sE)$, as in \cite[Sect.~2.1]{DegliseOrientation}.

    \sssec{Composition product}\label{sssec:BM/composition product}
    Given a \dA stack $\sT$ locally of finite type over $\sS$ and a \dA stack $\sX$ locally of finite type over $\sT$, there is a pairing
      \begin{equation*}
        \circ
          : \CBM(\sX/\sT; \sF)\vb{v} \otimes \CBM(\sT/\sS; \sF)\vb{w}
          \to \CBM(\sX/\sS; \sF)\vb{v+w}
      \end{equation*}
    where $v \in \K(\sX)$ and $w \in \K(\sT)$.
    This comes from the multiplication map $m : \sF \otimes \sF \to \sF$, see \cite[2.2.7(4)]{DegliseJinKhan} for details.

    For example, for $\sF = \Z^\mot$ this gives maps
      \begin{equation*}
        \mbf{z}_n(\sX/\sT) \otimes \mbf{z}_{n'}(\sT/\sS) \to \mbf{z}_{n+n'}(\sX/\sS).
      \end{equation*}

    Special cases of the composition product are cap and cup products:

    \sssec{Cap product}\label{sssec:BM/cap product}
    Given a \dA stack $\sX$ locally of finite type over $\sS$, there is a pairing
      \begin{equation}
        \cap
          : \Ccoh(\sX; \sF)\vb{v} \otimes \CBM(\sX/\sS; \sF)\vb{w}
          \to \CBM(\sX/\sS; \sF)\vb{v+w}.
      \end{equation}
    For example, for $\sF = \Z^\mot$ this gives maps
      \begin{equation*}
        \mbf{z}^n(\sX) \otimes \mbf{z}_{n'}(\sX/\sS) \to \mbf{z}_{n'-n}(\sX/\sS).
      \end{equation*}

    \sssec{Cup product}\label{sssec:BM/cup product}
    Given a \dA stack $\sX$ over $\sS$, there is a pairing
      \begin{equation}
        \cup
          : \Ccoh(\sX; \sF)\vb{v} \otimes \Ccoh(\sX; \sF)\vb{w}
          \to \Ccoh(\sX; \sF)\vb{v+w}.
      \end{equation}
    For example, for $\sF = \Z^\mot$ this gives maps
      \begin{equation*}
        \mbf{z}^n(\sX) \otimes \mbf{z}^{n'}(\sX) \to \mbf{z}^{n+n'}(\sX).
      \end{equation*}

  \subsection{Basic compatibilities}
  \label{ssec:BM/compatibilities}

    The operations on relative Borel--Moore spectra are subject to the following compatibilities, direct analogues of the axioms of a bivariant theory in the sense of Fulton--MacPherson \cite[Sect.~2.2]{FultonMacPherson}.

    \sssec{Change of base and composition product}
    \label{sssec:BM/compatibilities/base change + composition}

    Suppose given a commutative diagram
      \begin{equation*}
        \begin{tikzcd}
          \sX_\sT \ar{r}\ar{d}
            & \sX \ar{d}
          \\
          \sY_\sT \ar{r}{g}\ar{d}
            & \sY \ar{d}
          \\
          \sT \ar{r}{f}
            & \sS
        \end{tikzcd}
      \end{equation*}
    where the squares are cartesian.
    Then for $\alpha \in \CBM(\sX/\sY; \sF)\vb{v}$ and $\beta \in \CBM(\sY/\sS; \sF)\vb{w}$ we have
      \begin{equation*}
        f^*(\alpha \circ \beta) \simeq g^*(\alpha) \circ f^*(\beta)
      \end{equation*}
    in $\CBM(\sX_\sT/\sT; \sF)\vb{v+w}$.

    \sssec{Change of base and direct image}
    \label{sssec:BM/compatibilities/base change + direct image}

    Suppose given a commutative diagram
      \begin{equation*}
        \begin{tikzcd}
          \sX_\sT \ar{r}\ar{d}{h'}
            & \sX \ar{d}{h}
          \\
          \sY_\sT \ar{r}\ar{d}
            & \sY \ar{d}
          \\
          \sT \ar{r}{f}
            & \sS
        \end{tikzcd}
      \end{equation*}
    where the squares are cartesian.
    Then for any $\alpha \in \CBM(\sX/\sS; \sF)\vb{v}$ we have
      \begin{equation*}
        f^*h_*(\alpha) \simeq h'_*f^*(\alpha)
      \end{equation*}
    in $\CBM(\sY_\sT/\sT; \sF)\vb{v}$.

    \sssec{Direct image and composition product (on the right)}
    \label{sssec:BM/compatibilities/direct image + composition (right)}

    Suppose given a commutative diagram
      \begin{equation*}
        \begin{tikzcd}
          \sX \ar{rr}{f}\ar{rd}
            &
            & \sY \ar{ld}
          \\
            & \sS \ar{r}
            & \sT
        \end{tikzcd}
      \end{equation*}
    with $f$ representable and proper.
    Then for classes $\alpha \in \CBM(\sX/\sS; \sF)\vb{v}$ and $\beta \in \CBM(\sS/\sT; \sF)\vb{w}$ we have
      \begin{equation*}
        f_*(\alpha) \circ \beta \simeq f_*(\alpha \circ \beta)
      \end{equation*}
    in $\CBM(\sY/\sT; \sF)\vb{v+w}$.

    \sssec{Direct image and composition product (on the left)}
    \label{sssec:BM/compatibilities/direct image + composition (left)}

    Suppose given a commutative diagram
      \begin{equation*}
        \begin{tikzcd}
          \sX' \ar{rr}{g}\ar{d}
            &
            & \sY' \ar{d}
          \\
          \sX \ar{rr}{f}\ar{rd}
            &
            & \sY \ar{ld}
          \\
            & \sS
            &
        \end{tikzcd}
      \end{equation*}
    where the square is cartesian.
    Then for all classes $\alpha \in \CBM(\sX/\sS; \sF)\vb{v}$ and $\beta \in \CBM(\sY'/\sY; \sF)\vb{w}$ we have
      \begin{equation*}
        \beta \circ f_*(\alpha) \simeq g_*(f^*(\beta) \circ \alpha)
      \end{equation*}
    in $\CBM(\sY'/\sS; \sF)\vb{v+w}$.

  \subsection{Properties}
  \label{ssec:BM/properties}

    The following two statements follow immediately from \thmref{thm:sixops/stacks/localization}.

    \begin{thm}[Localization]\label{thm:BM/localization}
    Let $i : \sZ \to \sX$ be a closed immersion of \dA stacks over $\sS$, with open complement $j : \sU \to \sX$.
    Then there is an exact triangle of spectra
      \begin{equation*}
        \CBM(\sZ/\sS; \sF)
          \xrightarrow{i_*} \CBM(\sX/\sS; \sF)
          \xrightarrow{j^!} \CBM(\sU/\sS; \sF).
      \end{equation*}
    \end{thm}

    \begin{thm}[Derived invariance]\label{thm:BM/nil-invariance}
    Let $\sX$ be a \dA stack over $\sS$.
    \begin{thmlist}
      \item
      Let $i_\sS : \sS_\cl \to \sS$ denote the inclusion of the classical truncation.
      Then the change of base map
        \begin{equation*}
          i_\sS^* : \CBM(\sX/\sS; \sF) \to \CBM(\sX\fibprodR_\sS \sS_\cl/\sS_\cl; \sF)
        \end{equation*}
      is invertible.
      \item
      Let $i_\sX : \sX_\cl \to \sX$ denote the inclusion of the classical truncation.
      Then the direct image map
        \begin{equation*}
          (i_\sX)_* : \CBM(\sX_\cl/\sS; \sF) \to \CBM(\sX/\sS; \sF)
        \end{equation*}
      is invertible.
    \end{thmlist}
    \end{thm}

    \begin{prop}[Homotopy invariance]\label{prop:BM/homotopy}
    Let $\sX$ be a \dA stack over $\sS$.
    For a perfect complex $\sE$ on $\sX$ of Tor-amplitude $\le 0$, denote by $\pi : \bV_\sX(\sE) \to \sX$ the associated vector bundle stack.
    Then there is a canonical isomorphism
      \begin{equation*}
        \pi^! : \CBM(\sX/\sS; \sF)\vb{v} \to \CBM(\bV_\sX(\sE)/\sS; \sF)\vb{v-\sE}
      \end{equation*}
    for every $v \in \K(\sX)$.
    \end{prop}

    \begin{proof}
    Let $a : \sX \to \sS$ be the structural morphism.
    The map is induced by the natural transformation
      \begin{align*}
        a^!(\sF)\vb{v}
          &\xrightarrow{\mrm{unit}} \pi_*\pi^*\bigl(a^!(\sF)\vb{v}\bigr)
          \\
          &\xrightarrow{\on{gys}_\pi} \pi_*\bigl(\pi^!a^!(\sF)\vb{v-\pi^*\sE}\bigr),
      \end{align*}
    which is invertible by \propref{prop:sixops/stacks/homotopy invariance} and \thmref{thm:sixops/stacks/purity}.
    \end{proof}

    \begin{defn}\label{defn:fund smooth}
    Let $\sX$ be a \dA stack over $\sS$ with structural morphism $a : \sX \to \sS$.
    If $\sX$ is smooth over $\sS$, then its \emph{bivariant fundamental class} is
      \begin{equation*}
        [\sX/\sS] := a^!(1) \in \CBM(\sX/\sS; \sF)\vb{-\sL_{\sX/\sS}},
      \end{equation*}
    the image of the unit $1 \in \Ccoh(\sS; \sF) \simeq \CBM(\sS/\sS; \sF)$ by the Gysin map $a^!$ \sssecref{sssec:BM/Gysin}.
    More explicitly, this class is the morphism
      \begin{equation*}
        \un_\sS \to a_*(a^!(\sF)\vb{-\sL_{\sX/\sS}})
      \end{equation*}
    coming by adjunction from the Poincar\'e duality isomorphism $a^! \simeq a^*(-)\vb{\sL_{\sX/\sS}}$ (\thmref{thm:sixops/stacks/purity}).
    \end{defn}

    \begin{rmk}\label{rmk:fund smooth is classical}
    For $\sX$ smooth over $\sS$ as above, the fundamental class $[\sX/\sS]$ is ``classical'', in the sense that it is insensitive to the derived structure.
    That is, under the canonical isomorphism (\thmref{thm:BM/nil-invariance})
      \begin{equation*}
        \CBM(\sX/\sS; \sF)\vb{-\sL_{\sX/\sS}} \simeq \CBM(\sX_\cl/\sS_\cl; \sF)\vb{-\sL_{\sX_\cl/\sS_\cl}},
      \end{equation*}
    the class $[\sX/\sS]$ corresponds to $[\sX_\cl/\sS_\cl]$, the fundamental class of the morphism $\sX_\cl \to \sS_\cl$.
    Note that this makes sense because the latter is again a smooth morphism, with $\sL_{\sX_\cl/\sS_\cl} \simeq \sL_{\sX/\sS}|_{\sX_\cl}$.
    This is in contrast to the ``virtual'' case of quasi-smooth morphisms (\secref{sec:fund}).
    \end{rmk}

    \begin{thm}[Poincaré duality]\label{thm:poincare}
    Let $\sX$ be a smooth \dA stack over $\sS$.
    Then cap product \sssecref{sssec:BM/cap product} with the fundamental class $[\sX/\sS]$ is an isomorphism
      \begin{equation*}
        (-) \cap [\sX/\sS] : \Ccoh(\sX; \sF) \to \CBM(\sX/\sS; \sF)\vb{-\sL_{\sX/\sS}}.
      \end{equation*}
    \end{thm}

    \begin{proof}
    Unraveling definitions, this follows from the fact that the morphism $\un_\sS \to a_*(a^!(\sF)\vb{-\sL_{\sX/\sS}})$ defining $[\sX/\sS]$ is the ``right transpose'' of an isomorphism.
    See the discussion after \cite[Def.~2.3.11]{DegliseJinKhan}.
    \end{proof}

    \begin{thm}[Descent]\label{thm:BM/descent}
    Let $p : \sY \twoheadrightarrow \sX$ be a smooth surjection of \dA stacks over $\sS$.
    Assume either that $\sF$ satisfies étale descent, or that $p$ is Nisnevich-locally surjective.
    Then there is a homotopy limit diagram of spectra
      \begin{equation*}
        \CBM(\sX/\sS; \sF)
          \xrightarrow{p^!} \CBM(\sY/\sS; \sF)\vb{-\sL_p}
          \rightrightarrows \CBM(\sY\fibprodR_\sX \sY/\sS; \sF)\vb{-\sL_{p_2}}
          \rightrightrightarrows \cdots,
      \end{equation*}
    where $p_2 : \sY\fibprodR_\sX \sY \to \sX$ is the projection.
    \end{thm}

    \begin{proof}
    This follows from \eqref{eq:SHet totalization} as in \rmkref{rmk:Gamma^BM as limit}.
    \end{proof}

    \begin{thm}[Codescent]\label{thm:BM/codescent}
    Let $p : \sY \twoheadrightarrow \sX$ be a representable proper surjection of \dA stacks over $\sS$, where $\sX$ is a qcqs \dDM stack.
    Assume that $\sF$ belongs to $\SH_{\Q,+}(\sS)$ (e.g. $\sF = \Q^\mot_\sS$).
    Then there is a homotopy colimit diagram of spectra
      \begin{equation*}
        \cdots
          \rightrightrightarrows \CBM(\sY\fibprodR_\sX \sY/\sS; \sF)
          \rightrightarrows \CBM(\sY/\sS; \sF)
          \xrightarrow{p_*} \CBM(\sX/\sS; \sF).
      \end{equation*}
    \end{thm}

    \begin{proof}
    By \cite[Thm.~5.1]{propdesc}, $p$ is descendable for $\SH_{\Q,+}$, so by \cite[Rem.~2.11]{propdesc} the simplicial diagram $p_{\bullet,*}p_\bullet^!(a^!\sF)$, where $a : \sX \to \sS$ is the structural morphism, is a colimit diagram preserved by every exact functor.
    The claim follows by applying $\Maps_{\SH(\sX)}(\un_\sX, -)$.
    \end{proof}

\section{Fundamental classes}
\label{sec:fund}

    We develop some basic tools of intersection theory, namely the specialization and Gysin maps, for Borel--Moore homology of \dA stacks.
    We follow the constructions of Déglise--Jin--Khan \cite{DegliseJinKhan} closely, the main difference being the introduction of the derived normal bundle of a general quasi-smooth morphism, to handle the cases of quasi-smooth closed immersions and smooth morphisms simultaneously.

  \subsection{Construction}
  \label{ssec:fund/construction}

    Let $\sS$ be a \dA stack and fix a coefficient $\sF$ as in \secref{sec:BM}.
    Let $f : \sX \to \sY$ be a quasi-smooth morphism of \dA stacks over $\sS$, say of relative virtual dimension $d$.
    Denote by $N_{\sX/\sY}$ the normal bundle (\defnref{defn:normal bundle}).

    \begin{constr}[Specialization map]\label{constr:fund/sp}
    We define a specialization map
      \begin{equation}\label{eq:sp}
        \sp_{\sX/\sY}
          : \CBM(\sY/\sS; \sF)
          \to \CBM(N_{\sX/\sY}/\sS; \sF).
      \end{equation}
    Let $D_{\sX/\sY}$ be the normal deformation (\ssecref{ssec:deformation}).
    Denote by
    \[
        i : N_{\sX/\sY} \to D_{\sX/\sY}, \qquad
        j : \sY \times\bG_m \to D_{\sX/\sY}
    \]
    the inclusion of the exceptional fibre and its complement, respectively.
    The boundary map of the localization triangle
      \begin{equation*}
        \CBM(N_{\sX/\sY}/\sS; \sF)
          \xrightarrow{i_*} \CBM(D_{\sX/\sY}/\sS; \sF)
          \xrightarrow{j^!} \CBM(\sY\times\bG_m/\sS; \sF)
      \end{equation*}
    is a map $\partial : \CBM(\sY\times\bG_m/\sS; \sF)[-1] \to \CBM(N_{\sX/\sY}/\sS; \sF)$.

    The unit section $1 : \sY \to \sY\times\bG_m$ induces a splitting
      \begin{equation}\label{eq:split CBM Gm}
        \CBM(\sY\times\bG_m/\sS; \sF)[-1] \simeq \CBM(\sY/\sS; \sF) \oplus \CBM(\sY/\sS; \sF)\vb{1}[-1],
      \end{equation}
    under which the boundary map for $\sY \times \{0\} \subseteq \sY \times \bA^1$ is the projection onto the first summand.
    The inclusion of the first summand is equivalently the map $\gamma_t : \CBM(\sY/\sS; \sF) \to \CBM(\sY\times\bG_m/\sS; \sF)[-1]$ given by $x \mapsto \{t\} \cap p^!(x)$, where $p : \sY\times\bG_m \to \sY$ is the projection and $\{t\} \in \Ccoh(\sY\times\bG_m; \un)\vb{1}[-1]$ is the class of the coordinate $t$ (cf. \cite[3.2.2]{DegliseJinKhan}).

    We define \eqref{eq:sp} as the composite
      \begin{equation*}
        \CBM(\sY/\sS; \sF)
          \xrightarrow{\gamma_t} \CBM(\sY\times\bG_m/\sS; \sF)[-1]
          \xrightarrow{\partial} \CBM(N_{\sX/\sY}/\sS; \sF).
      \end{equation*}
    \end{constr}

    \begin{constr}[Gysin map]\label{constr:fund/gys}
    We now construct the Gysin map
      \begin{equation}\label{eq:Gysin}
        f^! : \CBM(\sY/\sS; \sF)
          \to \CBM(\sX/\sS; \sF)\vb{-\sL_{\sX/\sY}},
      \end{equation}
    where $f : \sX \to \sY$ is as above.
    Let $\pi : N_{\sX/\sY} \to \sX$ denote the projection.
    The Gysin map \eqref{eq:Gysin} is the composite
      \begin{equation*}
        \CBM(\sY/\sS; \sF)
          \xrightarrow{\sp_{\sX/\sY}} \CBM(N_{\sX/\sY}/\sS; \sF)
          \xrightarrow{(\pi^!)^{-1}} \CBM(\sX/\sS; \sF)\vb{-\sL_{\sX/\sY}},
      \end{equation*}
    where $\pi^!$ is the isomorphism of \propref{prop:BM/homotopy}.
    \end{constr}

    \begin{constr}[Fundamental class]\label{constr:fund}
    The \emph{bivariant fundamental class} of $f : \sX \to \sY$ is the class
      \begin{equation*}
        [\sX/\sY] := f^!(1) \in \CBM(\sX/\sY; \sF)\vb{-\sL_{\sX/\sY}},
      \end{equation*}
    the image of the unit $1 \in \Ccoh(\sY; \sF) \simeq \CBM(\sY/\sY; \sF)$.
    When $f$ is smooth, this is the fundamental class already defined (see \defnref{defn:fund smooth}).
    \end{constr}

    The bivariant \emph{virtual} fundamental class is defined to be the unique class
      \begin{equation*}
        [\sX/\sY]^\vir \in \CBM(\sX_\cl/\sY_\cl; \sF)\vb{-\sL_{\sX/\sY}}
      \end{equation*}
    corresponding to $[\sX/\sY]$ under the canonical isomorphisms of \thmref{thm:BM/nil-invariance}.

    \begin{rmk}\label{rmk:Gysin/fund}
    The Gysin map and bivariant fundamental class are essentially interchangeable data, as we can recover the former via the composition product \sssecref{sssec:BM/composition product} with $[\sX/\sY]$:
      \begin{equation*}
        f^!(x) \simeq [\sX/\sY] \circ x \in \CBM(\sX/\sS; \sF)\vb{v-\sL_{\sX/\sY}}
      \end{equation*}
    for all $v \in \K(\sY)$ and $x \in \CBM(\sY/\sS; \sF)\vb{v}$.
    \end{rmk}

    \begin{rmk}[Gysin transformation]\label{rmk:purity}
    In terms of the six operations, the fundamental class can be interpreted as a canonical natural transformation
      \begin{equation}
        \on{gys}_f : f^*(-)\vb{\sL_{\sX/\sY}} \to f^!(-)
      \end{equation}
    of functors $\D(\sY) \to \D(\sX)$, where $\vb{\sL_{\sX/\sY}}$ is the Thom twist \eqref{eq:Sigma^E}.
    If $\sF$ is oriented, this induces a canonical morphism
      \begin{equation}\label{eq:pur_f(F)}
        f^*(\sF)\vb{d} \simeq f^*(\sF)\vb{\sL_{\sX/\sY}} \to f^!(\sF)
      \end{equation}
    where $d$ is the relative virtual dimension of $f$.
    See \cite[Subsects.~2.5, 4.3]{DegliseJinKhan} for details on this perspective.
    \end{rmk}

    \refstepcounter{equation}\label{varnt:fund for asps}

    {\addtocounter{equation}{-1}%
    \renewcommand{\theequation}{\thesection.\arabic{equation}b}%
    \renewcommand{\theHequation}{\thesection.\arabic{equation}b}%
    \begin{rmk}[Higher rational equivalence]\label{rmk:fund/higher rational equivalence}
    For $\sF = \bZ^\mot$ the motivic cohomology spectrum, $[\sX/\sY]$ is a point of $\mbf{z}_{\dimvir}(\sX/\sY)$.
    For every algebraic $n$-simplex
    \[
        \sigma : \Delta^n_k \to \sY, \qquad
        \Delta^n_k := \Spec(k[t_0,\ldots,t_n]/(t_0+\cdots+t_n-1)),
    \]
    the base change map $\sigma^* : \mbf{z}_{\dimvir}(\sX/\sY) \to \mbf{z}_{\dimvir}(\sX_\sigma/\Delta^n_k)$, where $\sX_\sigma := \sX \fibprodR_\sY \Delta^n_k$, sends $[\sX/\sY]$ to $[\sX_\sigma/\Delta^n_k]$, which is identified under $\mbf{z}_{\dimvir}(\sX_\sigma/\Delta^n_k) \simeq \mbf{z}_{\dimvir+n}(\sX_\sigma)$ with $[\sX_\sigma]$.
    For $n=0$, $\sigma$ is a $k$-point $y$ and $[\sX_\sigma] = [\sX_y]$, and for $n=1$, $[\sX_\sigma]$ witnesses a rational equivalence between the fundamental classes $[\sX_0]$ and $[\sX_1]$ of the fibres over the vertices $0,1 \in \Delta^1_k$; more precisely, its proper representable direct image along a morphism $\sX_\sigma \to S \times \Delta^1_k$ over $\Delta^1_k$ determines a rational equivalence on $S$ between the images of $0^*[\sX_\sigma] \simeq [\sX_0]$ and $1^*[\sX_\sigma] \simeq [\sX_1]$.

    For $n=2$, $[\sX_\sigma]$ witnesses a higher compatibility between the rational equivalences between $[\sX_\epsilon]$, $\epsilon \in \{0,1,2\}$, witnessed by the restrictions $[\sX_{01}]$, $[\sX_{02}]$, $[\sX_{12}]$ of $[\sX_\sigma]$ over the edges $\Delta^1_k \hookrightarrow \Delta^2_k$; more precisely, after proper representable direct image to $S \times \Delta^2_k$, it determines a $2$-simplex in $\mbf{z}_{\dimvir}(S)$ identifying the rational equivalence on $S$ determined by $[\sX_{02}]$ with the composite of those determined by $[\sX_{01}]$ and $[\sX_{12}]$.
    And so on.
    \end{rmk}
    }

    {\addtocounter{equation}{-1}%
    \renewcommand{\theequation}{\thesection.\arabic{equation}c}%
    \renewcommand{\theHequation}{\thesection.\arabic{equation}c}%
    \begin{exam}[Refined Gysin maps]\label{exam:fund/refined Gysin}
    Let $f : X \to Y$ be an lci morphism of classical schemes.
    Though its classical base change $X \fibprod_Y V \to V$ need not be lci, Fulton constructs a refined Gysin map \cite[\S\S 6.2, 6.6]{Fulton} (cf. \cite[Def.~4.2.5]{DegliseJinKhan})
      \begin{equation*}
        f^!_V : \CH_*(V) \to \CH_{*+\dimvir}(X \fibprod_Y V)
      \end{equation*}
    assuming $f$ globally factorizable.
    In our formalism, $f^!_V$ is simply the Gysin map of the \emph{derived} base change $X \fibprodR_Y V \to V$, under the identification $\CH_*(X \fibprodR_Y V) \simeq \CH_*(X \fibprod_Y V)$ (\thmref{thm:BM/nil-invariance}).
    \end{exam}
    }

    {\addtocounter{equation}{-1}%
    \renewcommand{\theequation}{\thesection.\arabic{equation}d}%
    \renewcommand{\theHequation}{\thesection.\arabic{equation}d}%
    \begin{exam}[Localized top Chern classes]\label{exam:fund/localized Chern}
    Let $s$ be a section of a vector bundle $E$ on a smooth $k$-scheme $M$, with derived zero locus $Z(s) := M \fibprodR_{s,E,0} M$.
    Then $[Z(s)]^\vir \in \CH_{\dimvir}(Z(s)_\cl)$, where $\dimvir = \dim(M) - \rk(E)$ is the virtual dimension, is the localized top Chern class $c_{\mrm{top}}(E, s) \cap [M]$ (see \cite[Rem.~3.2.10]{DegliseJinKhan}, \cite[\S 14.1]{Fulton}).
    \end{exam}
    }

  \subsection{Properties}
  \label{ssec:fund/properties}

    We record the basic properties of the fundamental class.
    These could equivalently be stated for the Gysin maps.

    \begin{thm}[Functoriality]\label{thm:fund/functoriality}
    Let $f : \sX \to \sY$ and $g : \sY \to \sZ$ be quasi-smooth morphisms of \dA stacks, of relative virtual dimensions $d$ and $e$, respectively.
    Then there is a canonical homotopy
      \begin{equation*}
        [\sX/\sY] \circ [\sY/\sZ] \simeq [\sX/\sZ]
      \end{equation*}
    in $\CBM(\sX/\sZ; \sF)\vb{-\sL_{\sX/\sZ}}$.
    \end{thm}

    Use the double normal deformation as in \cite[3.2.19]{DegliseJinKhan}.

    \begin{thm}[Base change]\label{thm:fund/base change}
    Suppose given a cartesian square of \dA stacks
      \begin{equation}
        \begin{tikzcd}
          \sX' \ar{r}{g}\ar{d}{p}
            & \sY' \ar{d}{q}
          \\
          \sX \ar{r}{f}
            & \sY
        \end{tikzcd}
      \end{equation}
    over $\sS$, where $f$ is quasi-smooth.
    Then there is a canonical homotopy
      \begin{equation*}
        q^*[\sX/\sY] \simeq [\sX'/\sY'] \in \CBM(\sX'/\sY'; \sF)\vb{-\sL_{\sX'/\sY'}},
      \end{equation*}
    where we identify $\sL_{\sX'/\sY'} \simeq p^*\sL_{\sX/\sY}$.
    \end{thm}

    This follows easily from the stability of the normal deformation $D_{\sX/\sY}$ under base change (\thmref{thm:deformation/main}\itemref{item:deformation/main/base change}).
    More generally:

    \begin{prop}[Excess intersection formula]
    Suppose given a commutative square of \dA stacks
      \begin{equation}
        \begin{tikzcd}
          \sX' \ar{r}{g}\ar{d}{p}\ar[phantom]{rd}{\scriptstyle\Delta}
            & \sY' \ar{d}{q}
          \\
          \sX \ar[swap]{r}{f}
            & \sY
        \end{tikzcd}
      \end{equation}
    over $\sS$, where $f$ and $g$ are quasi-smooth.
    Assume that $\Delta$ is an \emph{excessive square}, i.e., that it is cartesian on classical truncations and that the fibre $\sE$ of the canonical map
      \begin{equation*}
        p^*\sL_{\sX/\sY}[-1] \to \sL_{\sX'/\sY'}[-1]
      \end{equation*}
    is a locally free $\sO_{\sX'}$-module of finite rank.
    Then there is a canonical homotopy
      \begin{equation*}
        q^*_\Delta[\sX/\sY] \simeq e(\sE) \cap [\sX'/\sY'] \in \CBM(\sX'/\sY'; \sF)\vb{-p^*\sL_{\sX/\sY}},
      \end{equation*}
    where $q^*_\Delta$ denotes the change of base map \sssecref{sssec:BM/base change}.
    \end{prop}

    Same as the proof of \cite[Prop.~3.2.8]{DegliseJinKhan}.
    We call $\sE$ the \emph{excess sheaf} associated to $\Delta$.
    Note that its rank is $r = \dimvir(\sX'/\sY') - \dimvir(\sX/\sY)$.

    \begin{cor}[Self-intersection formula]
    Let $i : \sX \to \sY$ be a quasi-smooth closed immersion of relative virtual codimension $n$.
    Consider the self-intersection square
      \begin{equation*}
        \begin{tikzcd}
          \sX \ar[equals]{r}\ar[equals]{d}\ar[phantom]{rd}{\scriptstyle\Delta}
            & \sX \ar{d}{i}
          \\
          \sX \ar[swap]{r}{i}
            & \sY.
        \end{tikzcd}
      \end{equation*}
    We have
      \begin{equation*}
        i^*_\Delta[\sX/\sY] \simeq e(\sN_{\sX/\sY}) \in \CBM(\sX/\sX; \sF)\vb{\sN_{\sX/\sY}} \simeq \Ccoh(\sX; \sF)\vb{\sN_{\sX/\sY}},
      \end{equation*}
    where $\sN_{\sX/\sY} = \sL_{\sX/\sY}[-1]$ is the conormal sheaf.
    \end{cor}

    \begin{cor}[Key formula]
    Let $i : \sX \to \sY$ be a quasi-smooth closed immersion of relative virtual codimension $n$.
    Form the blow-up square \cite[Thm.~4.1.5]{KhanRydh}:
      \begin{equation*}
        \begin{tikzcd}
          \sD \ar{r}\ar{d}{p}\ar[phantom]{rd}{\scriptstyle\Delta}
            & \Bl_{\sX/\sY} \ar{d}{q}
          \\
          \sX \ar[swap]{r}{i}
            & \sY,
        \end{tikzcd}
      \end{equation*}
    where $\sD = \P_\sX(\sN_{\sX/\sY})$ is the virtual exceptional divisor.
    We have
      \begin{equation*}
        q^*_\Delta[\sX/\sY] \simeq e(\sE) \cap [\sD/\Bl_{\sX/\sY}] \in \CBM(\sD/\Bl_{\sX/\sY}; \sF)\vb{p^*\sN_{\sX/\sY}},
      \end{equation*}
    where $\sE$ is the excess sheaf.
    \end{cor}

  \subsection{Comparison with Behrend--Fantechi}
  \label{ssec:fund/BF}

    Let $f : \sX \to \sY$ be a quasi-smooth morphism of derived $1$-Artin stacks.
    Assume that $\sY$ is classical and $f$ is representable by \dDM stacks.
    In this case the virtual fundamental class
      \begin{equation*}
        [\sX/\sY]^\vir \in \CBM(\sX_\cl/\sY; \sF)\vb{-\sL_{\sX/\sY}}
      \end{equation*}
    can also be defined using the approach of Behrend--Fantechi \cite{BehrendFantechi}.\footnote{%
        After this paper appeared, the intrinsic normal cone of \cite{BehrendFantechi} was generalized to (higher) Artin stacks by \cite{AranhaPstragowski}.
        Using this as a drop-in replacement, one can eliminate the hypothesis that $f$ is representable by \dDM stacks.
        Compare \cite[\S 3.14]{weilres}.
    }
    Below, we give a variant of the construction of the Gysin map $f^!$ \eqref{eq:Gysin} which will visibly agree with the ``virtual pullback'' of Manolache \cite{Manolache}.
    By Corollary~3.12 of \emph{op.~cit.}, this will therefore identify our virtual fundamental class $[\sX/\sY]^\vir$ with the construction of Behrend--Fantechi.

    Let $C_{\sX_\cl/\sY}$ denote the relative intrinsic normal cone \cite[Sect.~7]{BehrendFantechi} of the morphism $\sX_\cl \to \sX \to \sY$, and let $D_{\sX_\cl/\sY}$ denote Kresch's deformation to the intrinsic normal cone \cite[Thm.~2.31]{Manolache}.
    There is a commutative diagram
      \begin{equation*}
        \begin{tikzcd}
          C_{\sX_\cl/\sY} \ar{r}\ar{d}{\iota}
            & D_{\sX_\cl/\sY} \ar{d}
            & \sY \times \bG_m \ar{l}\ar[equals]{d}
          \\
          N_{\sX/\sY} \ar{r}
            & D_{\sX/\sY}
            & \sY \times \bG_m \ar{l}
        \end{tikzcd}
      \end{equation*}
    where the vertical arrows are closed immersions (cf. \cite[(3.27), Prop.~3.38]{weilres}).\footnote{%
      Moreover, by \cite[Prop.~3.38]{weilres}, $D_{\sX_\cl/\sY}$ is the schematic closure of $\sY \times \bG_m$ in the classical truncation of $D_{\sX/\sY}$, so that $C_{\sX_\cl/\sY}$ is the special fibre of this closure.
      In this sense the intrinsic normal cone is the ``support'' of the specialization map $\sp_{\sX/\sY}$ to $N_{\sX/\sY}$.
    }
    Using the upper row, one constructs just as in \eqref{eq:sp} a specialization map
      \begin{equation*}
        \sp_{\sX_\cl/\sY}
          : \CBM(\sY/\sS; \sF)
          \to \CBM(C_{\sX_\cl/\sY}/\sS; \sF).
      \end{equation*}
    By naturality of the localization triangle with respect to proper covariance (e.g. \cite[Prop.~2.2.10]{DegliseJinKhan}), we have a canonical homotopy
      \begin{equation*}
        \iota_* \circ \sp_{\sX_\cl/\sY} \simeq \sp_{\sX/\sY}
      \end{equation*}
    of maps $\CBM(\sY/\sS; \sF) \to \CBM(N_{\sX/\sY}/\sS; \sF)$.
    In particular we get
      \begin{equation*}
        f^!
          \simeq (\pi^!)^{-1} \circ \sp_{\sX/\sY}
          \simeq (\pi^!)^{-1} \circ \iota_* \circ \sp_{\sX_\cl/\sY},
      \end{equation*}
    where $\pi : N_{\sX/\sY} \to \sX$ is the projection.
    Now on $\pi_0$, for $\sF = \Z^\mot$ (\examref{exam:Q}), the right-hand side is precisely the virtual pullback $f^!_{N_{\sX/\sY}}$ of \cite[Constr.~3.6]{Manolache} constructed with respect to the vector bundle stack $N_{\sX/\sY}$.

    This discussion can also be interpreted as a construction of the virtual fundamental class using only the data of a closed immersion $\iota : C_{\sX_\cl/\sY} \hookrightarrow E$ into a vector bundle stack $E$ over $\sX_\cl$, rather than a quasi-smooth derived structure.
    This is equivalent to the data of a relative perfect obstruction theory (see \cite[Prop.~3.11]{Manolache}).\footnote{%
      In the Artin case, the notion of relative perfect obstruction theory in \cite{AranhaPstragowski} gives rise in particular to such an $\iota$, by \cite[Cor.~8.3, Constr.~8.7]{AranhaPstragowski}, but a priori is more data.
    }

  \subsection{Non-transverse Bézout theorem}
  \label{ssec:fund/Bezout}

    In this subsection we assume that $\sF$ is oriented.
    Let $f : \sZ \to \sX$ be a morphism of \dA stacks over $\sS$.
    Suppose that $f$ is quasi-smooth of relative virtual dimension $d$.
    The fundamental class $[\sZ/\sX]$ induces a cohomological Gysin map
      \begin{equation}\label{eq:cohom Gysin}
        f_! : \Ccoh(\sZ; \sF) \to \Ccoh_\sZ(\sX; \sF)\vb{-d}
      \end{equation}
    where the target is the cohomology of $\sX$ with support in $\sZ$ \eqref{eq:cohomology with support}.
    This map is the composite
      \begin{equation*}
        \Ccoh(\sZ; \sF) \simeq \CBM(\sZ/\sZ; \sF)
          \xrightarrow{\circ[\sZ/\sX]} \CBM(\sZ/\sX; \sF)\vb{-d}.
      \end{equation*}
    Composing further with the Borel--Moore direct image \sssecref{sssec:BM/proper}
      \begin{equation*}
        f_* : \CBM(\sZ/\sX; \sF)\vb{-d} \to \CBM(\sX/\sX; \sF)\vb{-d} \simeq \Ccoh(\sX; \sF)\vb{-d},
      \end{equation*}
    when it exists, gives rise to the Gysin map
      \begin{equation}
        f_! : \Ccoh(\sZ; \sF) \to \Ccoh(\sX; \sF)\vb{-d}
      \end{equation}
    valued in the cohomology of $\sX$.
    For example, this exists when $f$ is proper and representable, or just proper if $\sX$ is Deligne--Mumford and $\sF=\Q^\mot$ or $\MGL_\Q$.

    In particular we have a \emph{cohomological fundamental class}
      \begin{equation}
        [\sZ] := f_!(1) \in \Ccoh(\sX; \sF)\vb{-d}
      \end{equation}
    under these assumptions.
    For simplicity we'll state \thmref{thm:Bezout} below only for the representable case, but the proof only requires the existence of proper direct images.

    The following is a generalized cohomological Bézout theorem, where no transversality assumptions are imposed.

    \begin{thm}\label{thm:Bezout}
    Let $\sX$ be a \dA stack over $\sS$, and let $f : \sY \to \sX$ and $g : \sZ \to \sX$ be representable proper quasi-smooth morphisms of relative virtual dimensions $-d$ and $-e$, respectively.
    Then there is a canonical homotopy
      \begin{equation*}
        [\sY] \cup [\sZ] \simeq [\sY \fibprodR_\sX \sZ] \in \Ccoh(\sX; \sF)\vb{d+e}.
      \end{equation*}
    \end{thm}

    \begin{proof}
    Consider the homotopy cartesian square
      \begin{equation*}
        \begin{tikzcd}
          \sW \ar{r}{p}\ar{d}{q}\ar{rd}{h}
            & \sZ \ar{d}{g}
          \\
          \sY \ar{r}{f}
            & \sX
        \end{tikzcd}
      \end{equation*}
    where $\sW = \sY \fibprodR_\sX \sZ$.
    Under the identification
      \begin{equation*}
        \Ccoh(\sX; \sF)\vb{d+e} \simeq \CBM(\sX/\sX; \sF)\vb{d+e},
      \end{equation*}
    the desired identity is
      \begin{align*}
        h_*[\sW/\sX]
          &\simeq g_*p_*([\sW/\sZ] \circ [\sZ/\sX]) \\
          &\simeq g_*p_*(g^*[\sY/\sX] \circ [\sZ/\sX]) \\
          &\simeq g_*(p_*g^*[\sY/\sX] \circ [\sZ/\sX]) \\
          &\simeq g_*(g^*f_*[\sY/\sX] \circ [\sZ/\sX]) \\
          &\simeq f_*[\sY/\sX] \circ g_*[\sZ/\sX].
      \end{align*}
    The first and second identities follow from the functoriality and base change properties of the fundamental class (Theorems~\ref{thm:fund/functoriality} and \ref{thm:fund/base change}).
    For the rest we use the formulas \sssecref{sssec:BM/compatibilities/direct image + composition (right)}, \sssecref{sssec:BM/compatibilities/base change + direct image}, and \sssecref{sssec:BM/compatibilities/direct image + composition (left)}, in that order.
    \end{proof}

    More generally, let $Y$ and $Z$ be \emph{arbitrary} subvarieties of a smooth $k$-scheme $X$ of pure dimension $d$, where $k$ is a field.
    The intersection product
    \[
      [Y] \cdot [Z] \in \CBM(X/k; \Z^\mot)\vb{d-m-n},
    \]
    where $m=\dim(Y)$ and $n=\dim(Z)$, is Poincar\'e dual to the cup product in $\Ccoh(X; \Z^\mot)$.
    We describe it in terms of the derived intersection $Y \fibprodR_X Z$; first, we define its fundamental class as follows.
    We have the homotopy cartesian square
      \begin{equation*}\begin{tikzcd}
        Y \fibprodR_X Z \ar{r}{\delta}\ar{d}{i_{YZ}}
        & Y \times_k Z \ar{d}{i_Y \times i_Z}
        \\
        X \ar{r}{\Delta}
        & X \times_k X,
      \end{tikzcd}\end{equation*}
    where the diagonal $\Delta : X \to X \times_k X$ is quasi-smooth of relative virtual dimension $-d$, and $i_{\{Y,Z,YZ\}}$ are the inclusions.
    Thus $Y \fibprodR_X Z$ is quasi-smooth over $Y \times_k Z$, which itself admits a fundamental cycle $[Y \times_k Z] \in \CBM(Y\times_k Z/k; \Z^\mot)\vb{-m-n}$ under $\CH_{m+n}(Y \times_k Z) \simeq \pi_0\CBM(Y\times_k Z/k; \Z^\mot)\vb{-m-n}$ (\examref{exam:Q}), and we set:
    \[ [Y \fibprodR_X Z] := [Y \fibprodR_X Z / Y \times_k Z] \circ [Y \times_k Z] \in \CBM(Y \fibprodR_X Z/k; \Z^\mot)\vb{d-m-n}. \]
    We then have:

    {\addtocounter{equation}{-1}%
    \renewcommand{\theequation}{\thesection.\arabic{equation}b}%
    \renewcommand{\theHequation}{\thesection.\arabic{equation}b}%
    \begin{prop}\label{prop:Bezout nonlci}
    There is a canonical homotopy
    \begin{equation*}
        [Y] \cdot [Z]
        \simeq i_{YZ,*}[Y \fibprodR_X Z]
        \ \in \ \CBM(X/k; \Z^\mot)\vb{d-m-n}.
    \end{equation*}
    \end{prop}
    }
    \begin{proof}
    By \rmkref{rmk:Gysin/fund}, $[Y \fibprodR_X Z] \simeq \delta^![Y \times_k Z]$.
    By base change \sssecref{sssec:BM/compatibilities/base change + direct image}, we thus have:
        \begin{equation*}
            i_{YZ,*}[Y \fibprodR_X Z]
            \simeq \Delta^!(i_Y\times i_Z)_*[Y \times_k Z]
            \simeq \Delta^!(i_{Y,*}[Y] \times i_{Z,*}[Z]).
        \end{equation*}
    This is precisely the intersection product $[Y] \cdot [Z]$, i.e. it corresponds to the cup product under the Poincaré duality isomorphism $\CBM(X/k; \Z^\mot)\vb{-d} \simeq \Ccoh(X; \Z^\mot)$.
    \end{proof}

    {\addtocounter{equation}{-1}%
    \renewcommand{\theequation}{\thesection.\arabic{equation}c}%
    \renewcommand{\theHequation}{\thesection.\arabic{equation}c}%
    \begin{rmk}
        The definition of $[Y \fibprodR_X Z]$ combines fundamental classes (or Gysin morphisms) for quasi-smooth morphisms with fundamental cycles in Chow theory.
        For a $k$-variety $V$ of pure dimension $d$, the fundamental cycle $[V] \in \CH_d(V)$ \cite[\S 1.5]{Fulton} can be regarded as the flat pullback $a^*(1)$ (in the sense of \cite[\S 1.7]{Fulton}) for $a : V \to \Spec(k)$.
        One can thus imagine that in motivic Borel--Moore homology, Gysin maps should exist for some class of morphisms that contains both quasi-smooth morphisms and flat morphisms and is closed under composition.
        (The classical version of this question is raised in \cite[Ex.~17.4.6]{Fulton}, and a positive answer with rational coefficients provided in \cite[Ex.~18.3.17]{Fulton}.)
        This would ensure that the class $[Y \fibprodR_X Z]$ is intrinsic to the derived scheme $Y \fibprodR_X Z$ itself (without remembering the data of $Y$ and $Z$), just as in the previously discussed case where $Y$ and $Z$ are lci in $X$.
    \end{rmk}
    }

  \subsection{Grothendieck--Riemann--Roch}
  \label{ssec:fund/GRR}

    Let $\sF$ and $\sG$ be two oriented multiplicative coefficients over $\sS$ and $\phi : \sF \to \sG$ a ring morphism.
    The morphism $\phi$ induces maps
      \begin{equation*}
        \phi_* : \Ccoh(\sX; \sF) \to \Ccoh(\sX; \sG),
        \qquad
        \phi_* : \CBM(\sX/\sS; \sF) \to \CBM(\sX/\sS; \sG)
      \end{equation*}
    for every $\sX$ over $\sS$.
    Given a quasi-smooth morphism of \dA stacks $f : \sX \to \sY$ of relative virtual dimension $d$, let
    \[
      [\sX/\sY]^\sF \in \CBM(\sX/\sY; \sF)\vb{-d},
      \qquad
      [\sX/\sY]^\sG \in \CBM(\sX/\sY; \sG)\vb{-d}
    \]
    be the bivariant fundamental classes with coefficients in $\sF$ and $\sG$, respectively, under the Thom isomorphisms coming from the respective orientations of $\sF$ and $\sG$.
    The following Grothendieck--Riemann--Roch formula compares these two classes in terms of a certain class $\Td^\phi_{\sX/\sY} \in \Ccoh(\sX; \sG)$.

    \begin{thm}\label{thm:GRR}
    Let $f : \sX \to \sY$ be a quasi-smooth morphism of \dA stacks over $\sS$.
    Then we have
      \begin{equation}\label{eq:GRR}
        \phi_* ([\sX/\sY]^\sF) \simeq \Td^\phi_{\sX/\sY} \cap [\sX/\sY]^\sG
      \end{equation}
    in $\CBM(\sX/\sY; \sG)\vb{-d}$, where $d = \dimvir(\sX/\sY)$ and where $\cap$ denotes the cap product \sssecref{sssec:BM/cap product}.
    \end{thm}

    These immediately give the usual formulas in Borel--Moore homology and cohomology:

    \begin{cor}\label{cor:GRR BM}
    Let $f : \sX \to \sY$ be a quasi-smooth morphism of \dA stacks over $\sS$.
    Then the square
      \begin{equation*}
        \begin{tikzcd}
          \CBM(\sY/\sS; \sF) \ar{rrr}{f^!}\ar{d}{\phi_*}
            &&& \CBM(\sX/\sS; \sF)\vb{-d} \ar{d}{\phi_*}
          \\
          \CBM(\sY/\sS; \sG) \ar{rrr}{\Td^\phi_{\sX/\sY} \cap f^!}
            &&& \CBM(\sX/\sS; \sG)\vb{-d}
        \end{tikzcd}
      \end{equation*}
    commutes.
    If $f$ is moreover proper, then the square
      \begin{equation*}
        \begin{tikzcd}
          \Ccoh(\sX; \sF) \ar{rrr}{f_!}\ar{d}{\phi_*}
            &&& \Ccoh(\sY; \sF)\vb{-d} \ar{d}{\phi_*}
          \\
          \Ccoh(\sX; \sG) \ar{rrr}{f_!(\Td^\phi_{\sX/\sY} \cup -)}
            &&& \Ccoh(\sY; \sG)\vb{-d}
        \end{tikzcd}
      \end{equation*}
    also commutes.
    \end{cor}

    Let $\sE$ be a perfect complex of Tor-amplitude $\le 0$ on $\sX$ of virtual rank $-d$.
    We define the Todd class $\td_\phi(\sE)$.
    Since capping with the Thom class $\th_\sX^\sG(\sE)$ defines an isomorphism $\Ccoh(\sX; \sG) \to \Ccoh_\sX(\bV_\sX(\sE); \sG)\vb{-d}$, there exists a unique class
      \begin{equation*}
        \td_\phi(\sE) \in \H^0(\sX, \sG)^\times
      \end{equation*}
    such that the relation
      \begin{equation}\label{eq:Td(E)}
        \phi_*(\th_\sX^\sF(\sE)) \simeq \td_\phi(\sE) \cap \th_\sX^\sG(\sE)
      \end{equation}
    holds in $\Ccoh_\sX(\bV_\sX(\sE); \sG)\vb{-d}$.
    Exactly as in \cite[Subsect.~5.2]{DegliseOrientation}, this Todd class can be described explicitly using the formalism of formal group laws.
    We set $\Td_{\sX/\sY}^\phi = \td_\phi(\sL_{\sX/\sY}[-1])$ for $f : \sX \to \sY$ quasi-smooth.

    By deformation to the normal bundle (\ssecref{ssec:deformation}), the formula \eqref{eq:GRR} reduces to the case where $f$ is the zero section of a vector bundle stack $\sY = \bV_\sX(\sE)$, where $\sE$ is a perfect complex of Tor-amplitude $\le 0$.
    In this case the fundamental class $[\sX/\sY]^\sF$ is nothing else than the Thom class $\th_\sX^\sF(\sE)$, and similarly for $\sG$, so the formula reduces to \eqref{eq:Td(E)}.
    \thmref{thm:GRR} is proven.

    Let's make this formula slightly more explicit when $\phi$ is the total Chern character.
    This is a morphism of motivic ring spectra
      \begin{equation*}
        \ch : \KGL \to \bigoplus_{n\in\Z} \Q^\mot\vb{n}
      \end{equation*}
    in $\SH(\Spec(\Z))$, which induces an isomorphism
    \[\KGL_\Q \simeq \bigoplus_{n\in\Z} \Q^\mot\vb{n} \simeq \prod_{n\in\Z} \Q^\mot\vb{n}\]
    upon rationalization (see \cite{Riou}, \cite[5.3.3]{DegliseOrientation}, \cite[Cor.~14.2.14]{CisinskiDegliseBook}).
    For any \dA stack $\sS$, we obtain by inverse image along the structural morphism (which commutes with colimits) a canonical Chern character
      \begin{equation*}
        \ch : \KGL^\lisse_{\sS} \to \bigoplus_{n\in\Z} \Q^\mot_\sS\vb{n}
      \end{equation*}
    in $\SH(\sS)$, which induces an isomorphism $\KGL^\lisse_{\Q,\sS} \simeq \bigoplus_{n\in\Z} \Q^\mot_\sS\vb{n}$ (see \cite[Prop.~7.15(i)]{equilisse}).
    The source and target admit canonical orientations such that the Todd class $\Td_{\sX/\sY}$ is the classical Todd class (see \cite[5.3.3]{DegliseOrientation}).
    Suppose that $\sS$ is the spectrum of a regular noetherian ring $k$ of finite dimension.
    Passing to Borel--Moore spectra, we have
    \[
        \G^\lisse(-)
        \simeq \CBM(-; \KGL^\lisse)
        \to \CBM(-; \KGL^\lisse_\Q)
        \simeq \prod_{n\in\Z} \CBM(-; \Q^\mot)\vb{-n}
    \]
    on \dA stacks locally of finite type over $k$, and the proper covariance and Gysin maps are compatible with the respective intrinsic operations in $G$-theory (see \cite[Prop.~7.15(ii)]{equilisse}, whose proof applies verbatim over $k$ using \cite[Thm.~5.3.10, Rem.~5.3.16]{Riou} and \cite[Thm.~7.14]{Spitzweck}, and \examref{exam:KGL_et}).
    Under these identifications the Chern character $\ch$ induces on $\pi_0$ canonical homomorphisms which we denote
      \begin{equation*}
        \tau_{\sX/k}^{\lisse} : \pi_0 (\G^\lisse(\sX)) \to \prod_{n} \CH_n(\sX/k)_\Q.
      \end{equation*}
    We also write $\tau_{\sX/k}$ for its composite with the canonical morphism $\pi_0\G(\sX) \to \pi_0\G^\lisse(\sX)$.
    \corref{cor:GRR BM} now yields the formula
      \begin{equation*}
        \tau_{\sX/k}(\sO_\sX) = \Td_{\sX/k} \cap [\sX/k]
      \end{equation*}
    in $\prod_n \CH_n(\sX/k)_\Q$, or equivalently
      \begin{equation}
        [\sX/k] = \Td_{\sX/k}^{-1} \cap \tau_{\sX/k}(\sO_\sX).
      \end{equation}
    This is an extension to Artin stacks of Kontsevich's original conjectural formula in \cite[1.4.2]{Kontsevich} for the virtual fundamental class.
    For $\sX=X$ a quasi-projective scheme over a field, $\tau_X$ recovers the Riemann--Roch transformation of \cite{BaumFultonMacPherson} (this follows from the uniqueness theorem of \cite[Ch.~III, \S 2]{BaumFultonMacPherson}).

  \subsection{Absolute purity and duality}
  \label{ssec:fund/absolute purity}

    In this subsection we extend Gabber's proof of the absolute cohomological purity conjecture \cite[Exp.~I, 3.1.4]{SGA5} to Artin stacks.

    \begin{thm}[Absolute purity]\label{thm:absolute purity}
    Let $f : \sX \to \sY$ be a locally of finite type morphism between regular\footnote{Recall that an Artin stack $\sS$ is regular if and only if, for every smooth morphism $S \to \sS$ with $S$ a scheme, $S$ is regular.} Artin stacks over $\Z[\tfrac{1}{n}]$, for some integer $n\in\Z$.
    Let $\Lambda=\Z/n\Z$ and denote by $\Lambda^\et_{\sY}$ the constant \'etale sheaf over $\sY$.
    Then the Gysin transformation $\on{gys}_f$ \eqref{eq:pur_f(F)} induces a canonical isomorphism
      \begin{equation}\label{eq:absolute purity}
        \Lambda^\et_\sX(d)[2d] \to f^!(\Lambda^\et_\sY)
      \end{equation}
    of étale sheaves over $\sX$, where $d$ is the relative virtual dimension of $f$.
    \end{thm}

    Note that \thmref{thm:absolute purity} is new even in the case of schemes, since Gabber's statement \cite[Exp.~XVI, Cor.~3.1.2]{ILO} requires the schemes to admit ample line bundles.
    See also \cite{dualet}, where we have elaborated on the statement for \'etale cohomology.

    The new ingredient we use here is the Gysin transformation (\rmkref{rmk:purity}) which generalizes Gabber's construction of Gysin maps \cite[Exp.~XVI, 2.3]{ILO}.
    Neither the statement nor the proof of \thmref{thm:absolute purity} uses any derived geometry, but it is worth recalling that our construction of $\on{gys}_f$ involves the normal bundle $N_{\sX/\sY}$, which is a classical $2$-Artin stack even when $\sX$ and $\sY$ are classical $1$-Artin stacks.

    \begin{lem}\label{lem:absolute purity reduction}
    Let $\D$ be a topological weave on $S$-schemes for a scheme $S$, and fix an object $\sF \in \D(S)$; we write $\sF_\sX := a^*(\sF) \in \D(\sX)$ for every $a : \sX \to S$.
    The following conditions are equivalent:
    \begin{thmlist}
	\item For every closed immersion of regular schemes $i : X \to Y$ over $S$, the Gysin transformation induces an isomorphism $\on{gys}_i : \sF_X\vb{-\sN_{X/Y}} \simeq i^!(\sF_Y)$.

	\item For every locally of finite type morphism of regular Artin stacks $f : \sX \to \sY$ over $S$, the Gysin transformation induces an isomorphism $\on{gys}_f : \sF_\sX\vb{\sL_{\sX/\sY}} \simeq f^!(\sF_\sY)$.
    \end{thmlist}
    \end{lem}

    \begin{proof}
    Let $v : Y \to \sY$ be a smooth surjection with $Y$ schematic, and form the homotopy cartesian square
      \begin{equation*}
        \begin{tikzcd}
          X \ar{r}{f_0}\ar{d}{u}
            & Y \ar{d}{v}
          \\
          \sX \ar{r}{f}
            & \sY.
        \end{tikzcd}
      \end{equation*}
    The upper arrow $f_0$ is a locally of finite type morphism, $X$ is a regular Artin stack and $Y$ is a regular scheme.
    In terms of the Gysin transformation, the base change property of the fundamental class (\thmref{thm:fund/base change}) translates to the commutativity of the diagram (cf.~\cite[Prop.~2.5.4(ii)]{DegliseJinKhan})
      \begin{equation*}
        \begin{tikzcd}
          u^*(f^*(\sF_\sY)\vb{\sL_{\sX/\sY}}) \ar{r}{\on{gys}_f}\ar[equals]{d}
            & u^*f^!(\sF_\sY) \ar{d}{\mrm{Ex}^{*!}}
          \\
          f_0^*v^*(\sF_\sY)\vb{\sL_{X/Y}} \ar{r}{\on{gys}_{f_0}}
            & f_0^!v^*(\sF_\sY).
        \end{tikzcd}
      \end{equation*}
    The right-hand vertical arrow is the isomorphism induced by the exchange transformation $\mrm{Ex}^{*!}$ (\corref{cor:sixops/exchange}).
    Identifying $v^*(\sF_\sY) \simeq \sF_Y$, it will suffice to replace $f$ by $f_0$ and thereby assume that $\sY=Y$ is a regular scheme.

    We can find a smooth surjection $p : U \to X$ such that $U$ is a (regular) scheme and $f\circ p : U \to Y$ is smoothable.
    The functoriality property of the fundamental class (\thmref{thm:fund/functoriality}) translates to the commutativity of the diagram (cf.~\cite[Prop.~2.5.4(i)]{DegliseJinKhan})
      \begin{equation*}
        \begin{tikzcd}
          p^*(f^*(\sF_Y)\vb{\sL_{X/Y}})\vb{\sL_{U/X}} \ar{r}{\on{gys}_f}\ar[equals]{d}
            & p^*f^!(\sF_Y)\vb{\sL_{U/X}} \ar{r}{\on{gys}_p}
            & p^!f^!(\sF_Y) \ar[equals]{d}
          \\
          (f \circ p)^*(\sF_Y)\vb{\sL_{U/Y}} \ar{rr}{\on{gys}_{f\circ p}}
            &
            & (f\circ p)^!(\sF_Y).
        \end{tikzcd}
      \end{equation*}
    Since $p$ is smooth, the upper right-hand arrow $\on{gys}_p$ is invertible (\thmref{thm:sixops/stacks/purity}).
    Therefore, replacing $X$ by $U$ and $f$ by $f\circ p : U \to Y$, we may assume that $f : X \to Y$ is a smoothable morphism between regular schemes.

    Choose a factorization of $f$ through a closed immersion $i : X \to X'$ and a smooth morphism $g : X' \to Y$.
    Since $\on{gys}_g$ is invertible by \thmref{thm:sixops/stacks/purity}, applying the functoriality property again shows that $\on{gys}_f$ is invertible if and only if $\on{gys}_i$ is.
    Since $i$ is a closed immersion between regular schemes, the claim follows.
    \end{proof}

    \begin{proof}[Proof of \thmref{thm:absolute purity}]
    Take $\sF = \Lambda^\et$ the constant étale sheaf over $S = \Spec(\Z[\tfrac{1}{n}])$, where \lemref{lem:absolute purity reduction} applies by \cite[Exp.~XVI, Thm.~3.1.1]{ILO}.
    More precisely, the assertion is that for every closed immersion between regular schemes $i : X \to Y$, $\on{gys}_i$ induces (via the canonical orientation of \'etale sheaves) an isomorphism
      \begin{equation*}
         \Lambda^\et_X\vb{-c} \to i^!(\Lambda^\et_Y),
      \end{equation*}
    or equivalently isomorphisms in \'etale cohomology
      \begin{equation*}
        \H^{2k-2c}(X, \Lambda^\et(k-c)) \to \H^{2k}_{X}(Y, \Lambda^\et(k))
      \end{equation*}
    for all integers $k\in\Z$, where $c$ is the codimension of $i$.
    In this situation the Gysin transformation $\on{gys}_i$ is the same as the one constructed in \cite[4.3.1]{DegliseJinKhan}; by \cite[4.4.3]{DegliseJinKhan} it agrees with the construction of \cite[2.4.6]{DegliseBivariant} when applied to $\Lambda^\et$, which in turn agrees with Gabber's construction in \cite[Exp.~XVI, 2.3]{ILO} by design.
    This latter is the absolute purity isomorphism of \cite[Exp.~XVI, Thm.~3.1.1]{ILO}.
    \end{proof}

	We may also take $\sF \in \SH(S)$ to be the motivic cohomology spectrum $\Q^\mot$ over $S=\Spec(\Z)$, the rational sphere spectrum $\un_\Q$ over $S=\Spec(\Z)$, the minus part $\un^-$ of the sphere spectrum over $S=\Spec(\Z[\tfrac{1}{2}])$, or any object over $S=\Spec(k)$ with $k$ a perfect field.
    In these cases, \lemref{lem:absolute purity reduction}(i) holds by \cite[Thm.~14.4.1]{CisinskiDegliseBook} and \cite[Thms.~3.6, C.1]{minus}, and in particular removes the \emph{smoothability} hypotheses in \emph{op.~cit}.
    We also get (cf. \cite{dualet}):

    \begin{cor}[Absolute Poincar\'e duality]\label{cor:absolute duality}
    Let $k$ be a field with prime field $F$, $\sF \in \SH(\Spec(F))$, and $\sX$ a regular Artin stack locally of finite type over $k$.
    Then for every $v \in \K(\sX)$ there is a canonical isomorphism
      \begin{equation*}
        \CBM(\sX/k; \sF)\vb{-v} \simeq \Ccoh(\sX; \sF)\vb{\sL_{\sX/k}-v}.
      \end{equation*}
    \end{cor}

    \begin{proof}
    Applying the previous paragraph, with $S = \Spec(F)$, to the structural morphism $a : \sX \to \Spec(k)$, we have the Poincar\'e duality isomorphism $a^!(\sF_k)\vb{-v} \simeq \sF_\sX\vb{\sL_{\sX/k} - v}$, whence the claim.
    \end{proof}

    For $\sF=\Z^\mot$, this gives $\mbf{z}_n(\sX) \simeq \mbf{z}^{d-n}(\sX)$ where $d=\dim(\sX)$.
    In particular, the $E_\infty$-ring structure on $\bigoplus_n \mbf{z}^n(\sX)$ (\examref{exam:Q}) passes to an $E_\infty$-ring structure on $\bigoplus_n \mbf{z}_n(\sX)$.
    On homotopy groups this gives:

    \begin{constr}[Intersection products]\label{constr:integral product}
    Let $k$ be a field and $\sX$ a regular Artin stack locally of finite type over $k$ of pure dimension $d$.
    Under \corref{cor:absolute duality} (for $\sF=\Z^\mot$ the motivic cohomology spectrum), the cup product \sssecref{sssec:BM/cup product} defines an intersection product
      \begin{equation*}
        \CH_n(\sX,s) \otimes \CH_{n'}(\sX, s') \to \CH_{n+n'-d}(\sX, s+s')
      \end{equation*}
    on the higher Chow groups.
    \end{constr}

\changelocaltocdepth{1}
\appendix

\section{The six operations}
\label{sec:sixops}

    In this appendix we recall the extension of the six operations to \dA stacks from \cite{weavelisse}.
    For concreteness, the category of coefficients we use is the motivic stable homotopy category $\SH$, but the construction works for any topological weave in the sense of \cite[Def.~3.15]{weaves}, or equivalently any motivic \inftyCat of coefficients in the sense of \cite[Chap.~2]{KhanThesis} (see \cite[Cor.~3.33]{weaves}).
    The notion of ``motivic \inftyCat of coefficients'' is a refinement of that of ``motivic triangulated category'' studied in \cite{CisinskiDegliseBook}, but all natural examples of the latter can in fact be promoted to motivic \inftyCats.
    The $(\infty,1)$-categorical refinement is crucial for the construction below.
    See \cite[Sect.~2]{ToenSurvey} for a quick introduction to the theory of \inftyCats.

    The six operations in the (Nisnevich-local) motivic homotopy category $\SH$ were already constructed by Ayoub and Voevodsky for schemes.
    After reviewing the extension to derived schemes and algebraic spaces from \cite{KhanThesis,six}, we then recall the further extension to \dA stacks constructed in \cite{weavelisse}.

  \ssec{Derived algebraic spaces}
  \label{ssec:sixops/asp}

    \begin{thm}\label{thm:sixops/asp/main}
    The formalism of six operations on $\SH$ extends to \da spaces.
    In particular:
    \begin{thmlist}
      \item
      For every \da space $X$, there is a closed symmetric monoidal structure on $\SH(X)$.
      In particular, there are adjoint bifunctors $(\otimes, \uHom)$.

      \item\label{item:sixops/asp/f^*}
      For any morphism of \da spaces $f : X \to Y$, there is an adjunction
        $$ f^* : \SH(Y) \to \SH(X),
        \quad f_* : \SH(X) \to \SH(Y). $$
      The assignments $f \mapsto f^*$, $f \mapsto f_*$ are $2$-functorial.
      The functor $f^*$ is symmetric monoidal.

      \item\label{item:sixops/asp/f_!}
      For any locally of finite type morphism of \da spaces $f : X \to Y$, there is an adjunction
        $$ f_! : \SH(X) \to \SH(Y),
        \quad f^! : \SH(Y) \to \SH(X). $$
      The assignments $f \mapsto f_!$, $f \mapsto f^!$ are $2$-functorial.

      \item\label{item:sixops/asp/compatibilities}
      The operation $f_!$ satisfies the base change and projection formulas against $f^*$.
      That is, for any cartesian\footnotemark~square
        \begin{equation*}
          \begin{tikzcd}
            X' \ar{r}{f'}\ar{d}{p}
              & Y' \ar{d}{q}
            \\
            X \ar{r}{f}
              & Y
          \end{tikzcd}
        \end{equation*}
      there are identifications
        \begin{equation*}
          q^*f_! = (f')_!p^*
        \end{equation*}
      and
        \begin{equation*}
          f_!(\sF) \otimes \sG = f_!(\sF \otimes f^*(\sG))
        \end{equation*}
      naturally in $\sF$ and $\sG$.
      There is a natural transformation $\alpha_f : f_! \to f_*$, functorial in $f$, which is invertible if $f$ is proper.
      \footnotetext{In fact, item \itemref{item:sixops/asp/localization} below implies that it suffices to assume that the square is cartesian on classical truncations.}

      \item\label{item:sixops/asp/localization}
      Let $i : X \to Y$ be a closed immersion of \da spaces, with open complement $j$.
      Then the operation $i_* = i_!$ induces a fully faithful functor
        \begin{equation*}
          i_* : \SH(X) \to \SH(Y)
        \end{equation*}
      whose essential image is the kernel of $j^*$.
      In particular, if $i$ induces an isomorphism on underlying reduced classical stacks, then $i_*$ is an equivalence.

      \item\label{item:sixops/asp/homotopy}
      Let $X$ be a \da space and $\sE$ a locally free sheaf on $X$.
      If $p : \bV_X(\sE) \to X$ denotes the associated vector bundle, then the unit transformation
        \begin{equation*}
          \id \to p_*p^*
        \end{equation*}
      is invertible.

      \item\label{item:sixops/asp/Thom}
      There is a canonical map of presheaves of $E_\infty$-group spaces on the site of \da spaces,
        \begin{equation}\label{eq:K->PicSH}
          \K(-) \to \Aut_{\SH}(\SH(-)) \simeq \Pic(\SH(-)),
        \end{equation}
      from the algebraic K-theory of perfect complexes to the \inftyGrpd of $\SH$-linear auto-equivalences of $\SH$.
      For a perfect complex $\sE$ on a \da space $X$, we let $\sF \mapsto \sF\vb{\sE}$ denote the induced auto-equivalence of $\SH(X)$ (\emph{Thom twist}), and $\sF \mapsto \sF\vb{-\sE}$ its inverse.
      If $\sE$ is locally free, then we have
        \begin{equation*}
          (-)\vb{\sE} \simeq s^*p^!,
          \qquad
          (-)\vb{-\sE} \simeq s^!p^*,
        \end{equation*}
      where $p : \bV_X(\sE) \to X$ is the projection of the associated vector bundle and $s : X \to \bV_X(\sE)$ the zero section.

      \item\label{item:sixops/asp/purity}
      Let $f : X \to Y$ be a smooth morphism between \da spaces.
      Then there is a Poincar\'e duality isomorphism
        \begin{equation*}
          \on{gys}_f : f^*(-)\vb{\sL_{X/Y}} \simeq f^!
        \end{equation*}
      which is natural in $f$.
    \end{thmlist}
    \end{thm}

    This is proven in \cite{six}.
    More precisely, see \cite[Ex.~2.5]{six} for \itemref{item:sixops/asp/homotopy}, \cite[Thm.~2.35]{six} for \itemref{item:sixops/asp/f_!} and \itemref{item:sixops/asp/compatibilities}, \cite[Thm.~1.36]{six} for \itemref{item:sixops/asp/localization}, \cite[Rem.~1.32]{six} for \itemref{item:sixops/asp/Thom}, and \cite[Thm.~2.45]{six} for \itemref{item:sixops/asp/purity}.

    Strictly speaking, \cite[Thm.~2.35]{six} only gives the operation $f^!$ for morphisms of finite type.
    Using Zariski descent and the homotopy coherence of the six functor formalism, one extends this to locally of finite type morphisms.
    Indeed, the coherence of the data in \itemref{item:sixops/asp/f_!} and \itemref{item:sixops/asp/compatibilities} can be encoded using the formalism of Gaitsgory--Rozenblyum \cite[Part~III]{GaitsgoryRozenblyum} (as done in \cite[Chap.~2, Thm.~5.1.2]{KhanThesis}) or that of Liu--Zheng \cite{LiuZheng} (as done in \cite[Sect.~9.4]{RobaloThesis}); the two formalisms are almost equivalent, as explained in \cite[Part~III, 1.3]{GaitsgoryRozenblyum}.
    Then an easy application of the ``DESCENT'' program \cite[Thm.~4.1.8]{LiuZheng} gives the desired extension.

  \ssec{Derived algebraic stacks}
  \label{ssec:sixops/stack}

    We begin with the presheaf of \inftyCats
      \begin{equation*}
        X \mapsto \SH(X),
        \quad f \mapsto f^*
      \end{equation*}
    on the site of \da spaces.
    We take its lisse extension to \dA stacks in the sense of \cite[\S 3.2]{weavelisse}, and still denote it by $\SH$.

    We can be more explicit.
    If $\sX$ is a \dA stack and $p : X \twoheadrightarrow \sX$ is a Nisnevich-locally surjective\footnote{%
        That is, for every \da space $T$ and every morphism $t : T \to \sX$, there is a Nisnevich covering $T' \twoheadrightarrow T$ such that $t|_{T'}$ lifts along $p$.
        Such a $p$ always exists; see the references in \cite[\S 7.2]{weavelisse}.
    } smooth morphism with $X$ a \da space, then the \inftyCat $\SH(\sX)$ fits into a homotopy limit diagram of \inftyCats
      \begin{equation}\label{eq:SHet totalization}
        \SH(\sX) \xrightarrow{p^*} \SH(X)
          \rightrightarrows \SH(X\fibprodR_\sX X)
          \rightrightrightarrows \SH(X\fibprodR_\sX X \fibprodR_\sX X)
          \rightrightrightrightarrows \cdots,
      \end{equation}
    see \cite[Prop.~3.2.5, \S 7.2]{weavelisse}.
    By definition, $\SH(\sX)$ is the homotopy limit
      \begin{equation}\label{eq:SHet limit}
        \SH(\sX) = \lim \SH(X)
      \end{equation}
    taken over the \inftyCat $\Lis_\sX$ of all smooth morphisms $u : X \to \sX$ with $X$ schematic.
    Roughly speaking, objects $\sF \in \SH(\sX)$ may be viewed as collections $(u^*\sF)_u$, indexed over $(u : X \to \sX) \in \Lis_\sX$, compatible up to coherent homotopies.
    In particular, the family of functors $u^*$ is conservative as $u$ varies in $\Lis_\sX$.

    \begin{thm}\label{thm:sixops/stacks/main}\leavevmode
    \begin{thmlist}
      \item
      For every \dA stack $\sX$, there is a closed symmetric monoidal structure on $\SH(\sX)$.
      In particular, there are adjoint bifunctors $(\otimes, \uHom)$.

      \item\label{item:six operations for Artin stacks/f^*}
      For any morphism of \dA stacks $f : \sX \to \sY$, there is an adjunction
        $$ f^* : \SH(\sY) \to \SH(\sX),
        \quad f_* : \SH(\sX) \to \SH(\sY). $$
      The assignments $f \mapsto f^*$, $f \mapsto f_*$ are $2$-functorial.

      \item\label{item:six operations for Artin stacks/f^!}
      For any locally of finite type morphism of \dA stacks $f : \sX \to \sY$, there is an adjunction
        $$ f_! : \SH(\sX) \to \SH(\sY),
        \quad f^! : \SH(\sY) \to \SH(\sX). $$
      The assignments $f \mapsto f_!$, $f \mapsto f^!$ are $2$-functorial.

      \item\label{item:six operations for Artin stacks/f_!}
      The operation $f_!$ satisfies the base change\footnotemark~and projection formulas against $g^*$, and $f^!$ satisfies base change against $g_*$.
      If $f$ is $1$-representable and has proper diagonal, then there is a natural transformation $\alpha_f : f_! \to f_*$, functorial in $\sF$.
      If $f$ is $0$-representable and proper, then $\alpha_f$ is invertible.
      \footnotetext{From \thmref{thm:sixops/stacks/localization} below it follows that the base change formula applies also to squares that are only cartesian on classical truncations.}
    \end{thmlist}
    \end{thm}

    On the site of \da spaces, we may view $\SH$ as a presheaf valued in the \inftyCat of presentably symmetric monoidal \inftyCats and symmetric monoidal left-adjoint functors.
    Since the forgetful functor to (large) \inftyCats preserves limits \cite[Prop.~5.5.3.13]{LurieHTT}, the right Kan extension can be performed either way without changing the underlying presheaf of \inftyCats.
    In particular, we find that $\SH(\sX)$ is a presentably symmetric monoidal \inftyCat for every \dA stack $\sX$ and that $f^*$ is a symmetric monoidal left-adjoint functor for every morphism $f$.
    We let $\otimes$ denote the monoidal product, $\uHom$ the internal hom, and $f_*$ the right adjoint of $f^*$.

    Similarly, if we restrict the presheaf $\SH$ to \emph{smooth} morphisms between \da spaces, then it takes values in presentable \inftyCats and right adjoint functors (as follows from \thmref{thm:sixops/asp/main}\itemref{item:sixops/asp/purity}).
    By \cite[Thm.~5.5.3.18]{LurieHTT} its right Kan extension to \dA stacks will have the same property; that is, $f^*$ admits a left adjoint $f_\sharp$ for every smooth morphism $f$ of \dA stacks.

    Let $\SH^!$ denote the presheaf on the site of \da spaces, and locally of finite type morphisms, given by
      \begin{equation*}
        X \mapsto \SH(X),
        \qquad f \mapsto f^!
      \end{equation*}
    and take its lisse extension to \dA stacks.
    For every $\sX$ there is then a canonical equivalence
      \begin{equation*}
        \Theta_\sX : \SH(\sX) \to \SH^!(\sX)
      \end{equation*}
    determined by canonical isomorphisms
      \begin{equation*}
        u^!(\Theta_\sX(\sF)) \simeq u^*(\sF)\vb{\sL_{X/\sX}}
      \end{equation*}
    for all $u : X \to \sX$ in $\Lis_\sX$.
    For any morphism $f : \sX \to \sY$, we define $f^! : \SH(\sY) \to \SH(\sX)$ by $f^! =  \Theta_\sX^{-1} \circ f^! \circ \Theta_\sY$.
    More concretely, $f^!$ is determined by the fact that for any commutative square
      \begin{equation*}
        \begin{tikzcd}
          X \ar{r}{f_0}\ar{d}{u}
            & Y \ar{d}{v}
          \\
          \sX \ar{r}{f}
            & \sY
        \end{tikzcd}
      \end{equation*}
    with $u$ and $v$ smooth and $f_0$ a morphism of \da spaces, we have
      \begin{equation*}
        u^*f^!(\sF) \simeq f_0^!(v^*\sF)\vb{f_0^*(\sL_{Y/\sY}) - \sL_{X/\sX}}
      \end{equation*}
    for all $\sF \in \SH(\sY)$, or equivalently
      \begin{equation}\label{eq:f^! commutes with u^L*}
        u^*f^!(\sF)\vb{\sL_{X/\sX}} = f_0^!\bigl(v^*(\sF)\vb{\sL_{Y/\sY}}\bigr).
      \end{equation}
    Moreover, these identifications are subject to a homotopy coherent system of compatibilities as $f$ varies.

    On $\SH^!$, the operation $f^!$ automatically admits a left adjoint $f_!$ for every morphism $f$.
    Indeed, the right Kan extension can be computed in the \inftyCat of presentable \inftyCats and right adjoint functors (as the forgetful functor preserves limits \cite[Thm.~5.5.3.18]{LurieHTT}).
    This induces an operation $f_! : \SH(\sX) \to \SH(\sY)$ by $f_! =  \Theta_\sY^{-1} \circ f_! \circ \Theta_\sX$, so that
      \begin{equation*}
        f_! u_\sharp\bigl((-)\vb{-\sL_{X/\sX}}\bigr) \simeq v_\sharp\bigl((f_0)_!(-)\vb{-\sL_{Y/\sY}}\bigr)
      \end{equation*}
    for all commutative squares as above.
    All the data in \thmref{thm:sixops/stacks/main} are constructed in \cite[Cor.~7.2.4]{weavelisse}.

    Under certain assumptions the identification $f_! \simeq f_*$ can be extended to non-representable proper morphisms:

    \begin{thm}\label{thm:f_!=f_* for X DM}
    Let $f : \sX \to \sY$ be a proper $1$-representable morphism of \dA stacks with proper diagonal.
    Assume that $f$ is representable by \dDM stacks, or that $f$ has tame relative inertia and $\sY$ is equicharacteristic.
    For $\sF \in \SH(\sX)$, consider the morphism
      \begin{equation*}
        \alpha_f : f_!(\sF) \to f_*(\sF)
      \end{equation*}
    induced by the natural transformation $\alpha_f$ (\thmref{thm:sixops/stacks/main}\itemref{item:six operations for Artin stacks/f_!}).
    If $\sF$ belongs to $\SH_{\Q,+}(\sX)$, then this morphism is invertible.
    In particular, this applies to the rational motivic cohomology spectrum $\Q^\mot_\sX$ (\examref{exam:Q}), the rational algebraic cobordism spectrum $\MGL_{\Q,\sX}$ (\examref{exam:MGL_et}), or more generally any $\MGL_{\Q,\sX}$-module.
    \end{thm}

    \begin{proof}
    This is \cite[Thm.~D]{propdesc}.
    \end{proof}

    \begin{exam}\label{exam:finite parametrization}
    Let $f : \sX \to \sY$ be a morphism.
    If $\sX$ and $\sY$ are $1$-Artin, $\sY$ is equicharacteristic, and $f$ is proper and representable by tame stacks in the sense of \cite[Def.~3.1]{AOVtame} (i.e., the classical truncation of $\sX \fibprod_\sY T$ is tame for every scheme $T$ over $\sY$), then $f$ satisfies the assumptions of \thmref{thm:f_!=f_* for X DM} by \cite[Thm.~3.2]{AOVtame}.
    \end{exam}

    \begin{thm}[Localization]\label{thm:sixops/stacks/localization}
    Let $i : \sX \to \sY$ be a closed immersion of \dA stacks, with open complement $j$.
    Then the operation $i_* = i_!$ induces a fully faithful functor
      \begin{equation*}
        i_* : \SH(\sX) \to \SH(\sY)
      \end{equation*}
    whose essential image is the kernel of $j^*$.
    In particular, if $i$ induces an isomorphism on underlying reduced classical stacks, then $i_*$ is an equivalence.
    \end{thm}

    \begin{proof}
    For fully faithfulness it suffices to show that the co-unit $i^*i_* \to \id$ is invertible.
    After base change along a smooth atlas $v : Y \to \sY$ with $Y$ schematic, we get a closed immersion $i_0 : X \to Y$ and an induced atlas $u : X \to \sX$.
    It suffices to show the co-unit becomes invertible after applying $u^*$ on the left, in which case it is identified with $i_0^*(i_0)_*u^*(\sF) \to u^*(\sF)$, by the base change formula (\thmref{thm:sixops/stacks/main}\itemref{item:six operations for Artin stacks/f_!}).
    This is invertible by the localization theorem for derived schemes (\cite[Chap.~1, Cor.~7.4.9]{KhanThesis}).

    Since $\sX \fibprod_\sY (\sY\setminus \sX)$ is empty, the base change formula shows that $j^*i_* = 0$.
    It remains to show that if $\sF \in \SH(\sY)$ satisfies $j^*(\sF) = 0$, then the unit map $\sF \to i_*i^*(\sF)$ is invertible.
    By descent we reduce again to the schematic case which is \cite[Chap.~1, Cor.~7.4.7]{KhanThesis}.
    \end{proof}

    Thanks to David Rydh for the idea of the inductive argument in the proof below.

    \begin{prop}[Homotopy invariance]\label{prop:sixops/stacks/homotopy invariance}
    Let $\sX$ be a \dA stack and $\sE$ a perfect complex on $\sX$ of Tor-amplitude $\le 0$.
    If $\pi : \bV_\sX(\sE) \to \sX$ denotes the associated vector bundle stack, then the unit transformation
      \begin{equation*}
        \id \to \pi_*\pi^*
      \end{equation*}
    is invertible.
    \end{prop}

    \begin{proof}
    First assume that $\sE$ is of Tor-amplitude $[0,0]$, so that $\pi$ is a vector bundle.
    By descent we may assume that $\sX$ is schematic, in which case the claim holds almost by construction (see \cite[Chap.~2, Subsect.~3.2]{KhanThesis}).

    If $\sE$ is of Tor-amplitude $[-1,-1]$, then $\pi$ is the projection of the classifying stack of the vector bundle $\bV_\sX(\sE[1]) \to \sX$, and the canonical section $\sigma : \sX \to \bV_\sX(\sE)$ is a smooth surjection.
    The composite of the two unit maps $\id \to \pi_*\pi^* \to \pi_*\sigma_*\sigma^*\pi^* = \id$ is the identity, so it will suffice to show that the unit $\id \to \sigma_*\sigma^*$ is invertible.
    Since $\sigma$ is a smooth surjection it suffices moreover to show that $\sigma^! \to \sigma^!\sigma_*\sigma^*$ is invertible.
    By the base change formula for the square
      \begin{equation*}
        \begin{tikzcd}
          \bV_\sX(\sE[1]) \ar{r}{p}\ar{d}{p}
            & \sX \ar{d}{\sigma}
          \\
          \sX \ar{r}{\sigma}
            & \bV_\sX(\sE),
        \end{tikzcd}
      \end{equation*}
    we reduce to showing that the unit map $\id \to p_*p^*$ is invertible.
    This holds by the Tor-amplitude $[0,0]$ case already proven above.
    Repeating the same argument inductively shows the case of Tor-amplitude $[-k,-k]$ for all $k\ge 0$.

    For the general case of Tor-amplitude $[-k,0]$, we argue by induction on $k$ to reduce to the $k=0$ case above.
    The question being local on $\sX$, we may find a map $\sE_0[-k] \to \sE$, with $\sE_0$ locally free, which is surjective on $\pi_{-k}$.
    If $\sE'$ is the cofibre of this map, then $\sE'$ is of Tor-amplitude $[-(k-1),0]$, so by inductive assumption we know that the claim holds for $\pi' : \bV_\sX(\sE') \to \sX$ (i.e., that $\id \to (\pi')_*(\pi')^*$ is invertible).
    There is a commutative diagram
      \begin{equation*}
        \begin{tikzcd}
          \sX \ar{r}{\sigma'} \ar[swap]{rd}{\sigma}
            & \bV_\sX(\sE') \ar{r}{\pi'}\ar{d}{\tau}
            & \sX \ar{d}{\sigma_0}
          \\
            & \bV_\sX(\sE) \ar{r}{\rho}\ar[swap]{rd}{\pi}
            & \bV_\sX(\sE_0[-k]) \ar{d}{\pi_0}
          \\
            &
            & \sX,
        \end{tikzcd}
      \end{equation*}
    where the square is cartesian.
    As $\sE_0[-k]$ is of Tor-amplitude $[-k,-k]$, we already know that the unit $\id \to (\pi_0)_*(\pi_0)^*$ is invertible by the above.
    It remains to show that $\id \to \rho_*\rho^*$ is invertible, which can be done after applying $\sigma_0^!$ on the left.
    By the base change formula this follows from invertibility of the unit $\id \to (\pi')_*(\pi')^*$.
    \end{proof}

    The canonical map \eqref{eq:K->PicSH} of \thmref{thm:sixops/asp/main}\ref{item:sixops/asp/Thom} also extends to the site of \dA stacks:
      \begin{equation}\label{eq:K->PicSH on stacks}
        \K(-) \to \Aut_{\SH}(\SH(-)) \simeq \Pic(\SH(-)).
      \end{equation}
    Indeed, $\Pic(-)$ preserves the limit \eqref{eq:SHet limit}, hence is lisse-extended.
    Thus \eqref{eq:K->PicSH on stacks} factors through the lisse extension of \eqref{eq:K->PicSH}.
    We thus also have the (invertible) Thom twist operations
      \begin{equation}\label{eq:Sigma^E}
        (-)\vb{\sE} : \SH(\sX) \to \SH(\sX), \qquad \sF \mapsto \sF\vb{\sE}
      \end{equation}
    for $\sE \in \Perf(\sX)$ (cf. \cite[Constr.~1.29]{six}).

    \begin{thm}[Poincar\'e duality]\label{thm:sixops/stacks/purity}
    Let $f : \sX \to \sY$ be a smooth morphism of \dA stacks.
    Then there is a canonical isomorphism
      \begin{equation*}
        \on{gys}_f : f^*(-)\vb{\sL_{\sX/\sY}} \simeq f^!
      \end{equation*}
    which is natural in $f$ up to homotopy.
    \end{thm}

    \begin{proof}
    This follows immediately from the characterization of $f^!$ given in the proof of \thmref{thm:sixops/stacks/main}.
    \end{proof}

    \begin{exam}\label{exam:suspension}
    Let $\sE$ be a perfect complex of Tor-amplitude $\le 0$ on a \dA stack $\sX$.
    Then $\bV_\sX(\sE)$ is a smooth Artin stack over $\sX$.
    Let $p : \bV_\sX(\sE) \to \sX$ denote the projection and $s : \sX \to \bV_\sX(\sE)$ the zero section.
    By Poincar\'e duality (\thmref{thm:sixops/stacks/purity}) one has the formulas
      \begin{equation*}
        (-)\vb{\sE} \simeq s^*p^!,
        \qquad
        (-)\vb{-\sE} \simeq s^!p^*.
      \end{equation*}
    \end{exam}

    \begin{cor}\label{cor:sixops/exchange}
    Suppose given a commutative square
      \begin{equation*}
        \begin{tikzcd}
          \sX' \ar{r}{f'}\ar{d}{p}
            & \sY' \ar{d}{q}
          \\
          \sX \ar{r}{f}
            & \sY
        \end{tikzcd}
      \end{equation*}
    of \dA stacks which is cartesian on classical truncations.
    If $f$ is locally of finite type, there is a natural transformation
      \begin{equation*}
        \mrm{Ex}^{*!} : p^*f^! \to (f')^!q^*.
      \end{equation*}
    If either $f$ or $q$ is smooth, then $\mrm{Ex}^{*!}$ is invertible.
    \end{cor}

    \begin{proof}
    The natural transformation is defined as the composite
      \begin{equation*}
        p^* f^!
          \xrightarrow{\mrm{unit}} p^* f^! q_* q^*
          \simeq p^* p_* (f')^! q^*
          \xrightarrow{\mrm{counit}} (f')^! q^*
      \end{equation*}
    where the isomorphism in the middle is the base change formula, obtained by passage to right adjoints from the base change formula (\thmref{thm:sixops/stacks/main}\itemref{item:six operations for Artin stacks/f_!}).
    The second statement follows from \thmref{thm:sixops/stacks/purity}.
    \end{proof}

    \begin{constr}[Euler transformation]\label{constr:eul}
    Let $\sE$ be a locally free sheaf on a \dA stack $\sX$.
    There is a natural transformation
      \begin{equation}
        \eul_\sE : \id \to (-)\vb{\sE}
      \end{equation}
    of auto-equivalences of $\SH(\sX)$.
    More generally for any surjection $\phi : \sE \to \sE'$ of finite locally free sheaves, there is a natural transformation
      \begin{equation*}
        \vb{\phi} : (-)\vb{\sE'} \to (-)\vb{\sE}
      \end{equation*}
    constructed as follows.
    Consider the commutative triangle
      \begin{equation*}
        \begin{tikzcd}
          \bV_\sX(\sE') \ar{rr}{i}\ar[swap]{rd}{q}
            &
            & \bV_\sX(\sE)\ar{ld}{p}
          \\
            & \sX
        \end{tikzcd}
      \end{equation*}
    and let $t$ and $s$ be the respective zero sections.
    Then $\vb{\phi}$ is the composite
      \begin{equation*}
        t^*q^!
          = t^*i^!p^!
          \xrightarrow{\mrm{Ex}^{*!}} t^*i^*p^!
          = s^*p^!
      \end{equation*}
    under the identifications $s^*p^! \simeq (-)\vb{\sE}$ and $t^*q^! \simeq (-)\vb{\sE'}$ (\examref{exam:suspension}), where $\mrm{Ex}^{*!} : i^! \to i^*$ is the exchange transformation (\corref{cor:sixops/exchange}) for the self-intersection square of the closed immersion $i$.
    \end{constr}

\section{Comparison of Chow groups}\label{sec:Chow}

Let $k$ be a field.
Recall that for $\sX$ an Artin stack locally of finite type over $k$, $A$ an abelian group, and $n,s \in \Z$, we have defined the complexes $\mbf{z}_n(\sX)_A$ and the (higher) Chow groups $\CH_n(\sX, s)_A = \pi_s (\mbf{z}_n(\sX)_A)$, via motivic Borel--Moore spectra (\examref{exam:Q}).

We denote by $\CH_n^{\mrm{cyc}}(\sX)$ the cycle-theoretic Chow groups of Kresch \cite{Kresch}, defined when $\sX$ is $1$-Artin and of finite type with quasi-compact and separated diagonal.
When $\sX$ moreover has affine stabilizers, $\CH_n^{\mrm{cyc}}(-)$ satisfies homotopy invariance for vector bundle stacks (\cite[Thm.~2.1.12(x)]{Kresch}).

We compare $\CH_n(\sX)$ and $\CH_n^{\mrm{cyc}}(\sX)$, assuming further that $k$ is a perfect field:

\begin{thm}\label{thm:Chow}
Let $k$ be a perfect field and $\sX$ a $1$-Artin stack of finite type over $k$, with quasi-compact separated diagonal and affine stabilizers.
Then there is a canonical isomorphism
\[ \mrm{cl}_\sX : \CH_n^{\mrm{cyc}}(\sX) \to \CH_n(\sX) \]
for all $n\in\Z$.
\end{thm}

\begin{rmk}
For any commutative ring $\Lambda$, we have $\CH_n(\sX)_\Lambda \simeq \CH_n(\sX) \otimes \Lambda$ (\propref{prop:Chow coefficients}), so \thmref{thm:Chow} yields a $\Lambda$-linear isomorphism between $\CH_n(\sX)_\Lambda$ and $\CH_n^{\mrm{cyc}}(\sX)_\Lambda := \CH_n^{\mrm{cyc}}(\sX) \otimes \Lambda$.
\end{rmk}

We will say that an Artin stack $\sX$ over $k$ is \emph{good} if $k$ is perfect and $\sX$ is $1$-Artin, of finite type over $k$, with quasi-compact separated diagonal and affine stabilizers.
A \emph{good algebraic space} is thus a quasi-separated algebraic space of finite type over a perfect field.

When $\sX=X$ is a quasi-separated scheme (resp. algebraic space), Kresch's $\CH_n^{\mrm{cyc}}(X)$ recovers the usual Chow groups as defined in e.g. \cite{Fulton} (resp. \cite[\S 6.1]{EdidinGraham}).
For schemes, comparison isomorphisms as in \thmref{thm:Chow} were known in characteristic zero, and in characteristic $p>0$ when $X$ is smooth or up to $p$-torsion (see \cite[Thm.~19.1, Prop.~19.18]{MazzaWeibelVoevodsky}, \cite[Cor.~2]{VoevodskyIMRN}, \cite[Cor.~8.13]{CDIntegral}, \cite[Cor.~3.9]{JinBM}).
When $\sX=[X/G]$ is a global quotient stack, $\CH_n^{\mrm{cyc}}(\sX)$ recovers the equivariant Chow groups of Totaro--Edidin--Graham \cite{Totaro,EdidinGraham}, and the comparisons just stated are generalized to the equivariant case in \cite[Cor.~6.5]{equilisse}.
We also note that a proof of \thmref{thm:Chow} with rational coefficients, by a different method (private communication), was announced by Y.~Bae and H.~Park in 2021 (see \cite[Rem.~2.23]{BaeSchmittII}).

We briefly sketch the proof of \thmref{thm:Chow}.
The first task is the construction of the canonical homomorphisms $\mrm{cl}_\sX : \CH_n^{\mrm{cyc}}(\sX) \to \CH_n(\sX)$, which is done in \ssecref{ssec:Chow/cl}.
To prove that they are isomorphisms, we will repeatedly apply the criterion of \lemref{lem:Chow devissage}: if $i : \sZ \hookrightarrow \sX$ is a closed immersion such that $\mrm{cl}_\sZ$ and $\mrm{cl}_{\sX\setminus\sZ}$ are invertible, then $\mrm{cl}_\sX$ is also invertible under the following additional condition:
\begin{enumerate}
  \item[(K)] For every $\beta \in \CH_n(\sX\setminus\sZ, 1)$ there exists $\zeta \in \CH_n^{\mrm{cyc}}(\sZ)$ with $i_*\zeta = 0$ and $\mrm{cl}_\sZ(\zeta) = \partial(\beta)$.
\end{enumerate}
For an algebraic space $\sX=X$, we show that (K) holds for all closed subspaces $\sZ=Z$ (\corref{cor:Chow K}).
In fact, the boundary of any class in $\CH_n(X\setminus Z, 1)$ is the class of a cycle on $Z$ whose direct image to $X$ is rationally equivalent to zero (\lemref{lem:Chow K cycle}).
By induction on the dimension, we are thus reduced to the case where $X$ is a smooth separated scheme, where we check that $\mrm{cl}_X$ is invertible via a comparison of motivic Borel--Moore homology with Bloch's higher Chow groups (see \lemref{lem:Chow smooth}).
See \propref{prop:Chow spaces}.

We then consider the case of a global quotient $\sX = [X/\mrm{GL}_N]$.
There, $\CH_n(\sX)$ can be described as the Chow group of a finite-dimensional approximation to the Borel construction just as in \cite{EdidinGraham}, so the claim follows from the case of algebraic spaces (see \lemref{lem:Chow quotient}).
We then show that (K) holds when the open complement $\sX\setminus\sZ$ is a global quotient stack, which allows us to prove the case of a general $\sX$ using a stratification by global quotients \sssecref{sssec:proof of thm:Chow}.

\ssec{Bloch's cycle complexes}\label{ssec:Chow/Bloch}

For a $k$-scheme $X$ of finite type, denote by $z_n(X, \bullet)$ Bloch's cycle complexes (see \cite{Bloch}, \cite[\S 1.8]{LevineLocalization}).
In this subsection we prove the following comparison, generalizing \cite[Thm.~19.1]{MazzaWeibelVoevodsky} and \cite[Cor.~2]{VoevodskyIMRN} beyond the smooth case:

\begin{thm}\label{thm:Chow Bloch}
Let $k$ be a field and $X$ a $k$-scheme that admits a closed immersion into a smooth separated $k$-scheme of finite type, e.g. a quasi-projective $k$-scheme.
Then for every $n \in \Z$ there is an isomorphism of complexes\footnote{%
  by a \emph{complex} we mean an object in the derived \inftyCat of $\Z$-modules
}
\[ z_n(X, \bullet) \simeq \mbf{z}_n(X). \]
\end{thm}

\begin{lem}\label{lem:Chow Psi}
  Let $Y$ be a smooth separated $k$-scheme of finite type, of pure dimension $d$, and let $n \in \Z$.
  For every open subscheme $V \subseteq Y$ there is an isomorphism of complexes
  \[ \Psi_V : z_n(V, \bullet) \simeq \mbf{z}_n(V). \]
  Moreover, these isomorphisms commute with Gysin pull-back along smooth morphisms of smooth separated $k$-schemes of finite type.
\end{lem}
\begin{proof}
By Poincar\'e duality (\thmref{thm:poincare}) we have the isomorphisms $\mbf{z}_n(V) = \CBM(V; \Z^\mot)\vb{-n} \simeq \mbf{z}^{d-n}(V)$, exchanging Gysin pull-back and inverse image along smooth morphisms.
Writing $q := d-n$, so that $z_n(V, \bullet) = z^q(V, \bullet)$ is Bloch's complex of codimension $q$ cycles, it will thus suffice to construct an isomorphism $z^q(-, \bullet) \simeq \mbf{z}^q(-)$ of presheaves on the category of smooth separated $k$-schemes of finite type and smooth morphisms.

The spectrum $\Z^\mot_k \in \SH(k)$ is isomorphic to Voevodsky's motivic cohomology spectrum over $k$ (combine \cite[Thms.~7.6, 7.18, 8.22, Lem.~8.23]{Spitzweck}).
If $q < 0$, both sides vanish: $z^q(-, \bullet) = 0$ by construction of the Bloch cycle complexes, while $\mbf{z}^q(-) = 0$ (see \cite[Cor.~4.26(i)]{Hoyois}).
For $q \ge 0$, Spitzweck identifies Voevodsky's spectrum with a motivic spectrum $\mrm{M}_k$ whose $q$th term is the Friedlander--Suslin complex $\tilde z^q := C_*(z_{\mrm{equi}}(\bA^q, 0))$ of Zariski sheaves on smooth $k$-schemes (see \cite[Prop.~5.32, Prop.~6.1, Thm.~6.7]{Spitzweck}).
It is related to Bloch's $z^q(-, \bullet)$ by a zigzag of maps of presheaves which are Zariski-local quasi-isomorphisms (see \cite[Prop.~12.1]{FriedlanderSuslin}, whose proof applies over possibly non-perfect fields by the argument in the proof of \cite[Thm.~1, p.~352]{VoevodskyIMRN}).
On smooth $k$-schemes, the presheaf $z^q(-, \bullet)$ is $\bA^1$-invariant by \cite[Thm.~2.1]{Bloch} and satisfies Nisnevich excision by the localization theorem (see \cite[Thm.~1.7, \S 7]{LevineLocalization}).
Thus $\tilde z^q$ is $\bA^1$-invariant and Nisnevich-local, and we have isomorphisms of presheaves
\[ \mbf{z}^q(-) \simeq \Maps_{\SH(k)}(\Sigma^\infty_\bT (-)_+, \mrm{M}_k\vb{q}) \simeq \mrm{R}\Gamma_{\mrm{Zar}}(-, \tilde z^q) \simeq z^q(-, \bullet) \]
as claimed.
\end{proof}

\begin{proof}[Proof of \thmref{thm:Chow Bloch}]
Let $i : X \hookrightarrow P$ be a closed immersion into a smooth separated $k$-scheme of finite type, and let $U := P \setminus X$.
Every connected component of $P$ is of pure dimension, and both $z_n(-, \bullet)$ and $\mbf{z}_n(-)$ send finite disjoint unions to direct sums, so \lemref{lem:Chow Psi} provides isomorphisms $\Psi_P$ and $\Psi_U$ which commute with restriction along $U \subseteq P$.
Thus we get an induced isomorphism on fibres
\[ z_n(X, \bullet) \simeq \mbf{z}_n(X), \]
by the localization triangles of \cite[Thm.~1.7, \S 7]{LevineLocalization} and \thmref{thm:BM/localization}.
\end{proof}

\ssec{Vanishing statements}\label{ssec:Chow/vanishing}

In this subsection we begin by checking that $\CH_n(\sX, s)$ vanishes for $n > \dim(\sX)$ (\lemref{lem:Chow dim}) and for $s < 0$ (\lemref{lem:Chow conn}), and is unchanged by excising a closed substack of dimension $< n$ (\corref{cor:Chow delete}).
In particular, this will imply that the formation of $\CH_n(\sX)_\Lambda$ is compatible with extensions of scalars (\propref{prop:Chow coefficients}).

For an Artin stack $\sX$ locally of finite type over $k$, $\dim(\sX)$ denotes the supremum of $\dim(U) - e$ over all smooth morphisms $U \to \sX$ of relative dimension $e$, where $U$ is an algebraic space.

\begin{lem}\label{lem:Chow smooth MC}
  Let $X$ be a smooth separated scheme of finite type over $k$, of pure dimension $d$, and let $q := d-n$.
  Then $x \mapsto x \cap [X]$ is an isomorphism $\H^{2q-s}(X, \Z^\mot(q)) \simeq \CH_n(X, s)$ for all $s \in \Z$.
  In particular, $\CH_n(X, s)$ vanishes if $q<0$, and is isomorphic to the motivic cohomology group $\H^{2q-s,q}(X, \Z)$ of \cite[Def.~3.4]{MazzaWeibelVoevodsky} if $q \ge 0$.
\end{lem}
\begin{proof}
The isomorphism is Poincar\'e duality (\thmref{thm:poincare}).
The remaining claims follow from \cite[Thm.~4.24, Cor.~4.26(i)]{Hoyois} and the identification $\Z^\mot_k \simeq \mrm{M}\Z_k$ of \cite[Thms.~7.6, 7.18, 8.22, Lem.~8.23]{Spitzweck} (see the proof of \lemref{lem:Chow Psi}).
\end{proof}

\begin{lem}\label{lem:Chow generic}
  Let $\sX$ be a reduced good stack with irreducible components $\sX_1, \ldots, \sX_r$.
  Then there exist substacks $\sU_j \subseteq \sX_j$, each smooth, integral, and open in $\sX$, such that $\dim(\sX_j \setminus \sU_j) < \dim(\sX_j)$ for each $j$.
  Moreover, if $\sX=X$ is an algebraic space, then the $\sU_j=U_j$ may be taken to be affine.
\end{lem}
\begin{proof}
Let $\sU_j$ be the intersection of the smooth locus of $\sX$ with the complement of $\bigcup_{l \ne j} \sX_l$.
Since the smooth locus is dense (as $k$ is perfect) and $\sX_j \not\subseteq \bigcup_{l \ne j} \sX_l$, $\sU_j$ is a nonempty open of $\sX_j$.
Hence $\sU_j$ is integral, and its complement in $\sX_j$ is a proper closed substack of the integral $\sX_j$.
Finally, if $\sX=X$ is an algebraic space, then it contains a dense open scheme $X^\circ$ and we may replace each $\sU_j$ by a nonempty affine open $U_j$ of its intersection with $X^\circ$.
\end{proof}

\begin{lem}\label{lem:Chow dim}
  Let $\sX$ be a good stack and let $n > \dim(\sX)$.
  Then $\CH_n(\sX, s) = 0$ for all $s \in \Z$.
\end{lem}
\begin{proof}
This is equivalent to the vanishing $\CBM(\sX; \Z^\mot)\vb{-n} \simeq 0$.
Assume first that $\sX=X$ is a scheme.
Then by localization and induction on the dimension, we may replace $X$ by a dense open which is smooth of pure dimension $d < n$, in which case $\CH_n(X, s) = 0$ by \lemref{lem:Chow smooth MC}.
In general, by \rmkref{rmk:Gamma^BM as limit}, $\CBM(\sX; \Z^\mot)\vb{-n}$ is the limit of the spectra $\CBM(X; \Z^\mot)\vb{-n-d_u}$ over smooth morphisms $u : X \to \sX$ with $X$ a scheme, where $d_u$ is the relative dimension of $u$.
Each of these vanishes by the scheme case, since $\dim(X) \le \dim(\sX) + d_u < n + d_u$.
\end{proof}

\begin{cor}\label{cor:Chow delete}
  Let $\sX$ be a good stack and let $\sD \subseteq \sX$ be a closed substack with $\dim(\sD) < n$.
  Then restriction $\CH_n(\sX, s) \to \CH_n(\sX\setminus\sD, s)$ is invertible for all $s \in \Z$.
\end{cor}
\begin{proof}
Follows from \lemref{lem:Chow dim} and the long exact localization sequence
\[
  \cdots \to
  \CH_n(\sD, s) \xrightarrow{i_*}
  \CH_n(\sX, s) \xrightarrow{j^!}
  \CH_n(\sU, s) \xrightarrow{\partial}
  \CH_n(\sD, s-1) \xrightarrow{i_*} \cdots,
\]
where $\sU := \sX \setminus \sD$ (see \thmref{thm:BM/localization}).
\end{proof}

\begin{lem}\label{lem:Chow approx}
  Let $X$ be a good algebraic space with an action of $\mrm{GL}_N$, and let $\sX = [X/\mrm{GL}_N]$.
  For every $n \in \Z$ there exists a vector bundle $p : \sV \to \sX$ of rank $r$, and an open substack $j : W \hookrightarrow \sV$ which is an algebraic space, such that the composite
  \[
    \CH_n(\sX, s) \xrightarrow{p^!} \CH_{n+r}(\sV, s) \xrightarrow{j^!} \CH_{n+r}(W, s)
  \]
  is invertible for all $s \in \Z$.
\end{lem}
\begin{proof}
Choose $m \ge N$ with $m-N+1 > \dim(\sX)-n$, and let $M^\circ \subseteq M := \Hom(k^m, k^N)$ be the open subscheme of surjections.
By the proof of \cite[Lem.~9]{EdidinGraham}, $M^\circ$ is a $\mrm{GL}_N$-torsor over a Grassmannian, and its complement in $M$ is of codimension $m-N+1$.
Take $\sV := [X \times M/\mrm{GL}_N]$, a vector bundle of rank $r := Nm$ over $\sX$, and $W := [X \times M^\circ/\mrm{GL}_N] \subseteq \sV$.
Then $p^!$ is invertible by homotopy invariance, and $j^!$ by \corref{cor:Chow delete}, since $\dim(\sV\setminus W) \le \dim(\sX) + r - (m-N+1) < n+r$.
\end{proof}

\begin{lem}\label{lem:Chow conn}
  Let $\sX$ be a good stack.
  Then $\CH_n(\sX, s) = 0$ for all $n \in \Z$ and all $s < 0$.
\end{lem}
\begin{proof}
By \cite[Prop.~3.5.9]{Kresch}, $\sX$ admits a finite stratification by locally closed substacks of the form $[X/\mrm{GL}_N]$ with $X$ as in \lemref{lem:Chow approx}.
By localization and induction on the number of strata, we may assume that $\sX$ is of this form, and then by \lemref{lem:Chow approx} that $\sX = X$ is a quasi-separated algebraic space.
By localization, induction on the dimension, and \lemref{lem:Chow generic}, we may further assume that $X$ is a smooth integral scheme.
Then by \lemref{lem:Chow smooth MC}, $\CH_n(X, s)$ vanishes if $q := \dim(X)-n < 0$, and otherwise is the motivic cohomology group $\H^{2q-s,q}(X, \Z)$, which also vanishes since $2q-s > 2q$ (see \cite[Thm.~19.3]{MazzaWeibelVoevodsky}).
\end{proof}

\begin{prop}\label{prop:Chow coefficients}
  Let $\sX$ be a good stack.
  Then for every commutative ring $\Lambda$, there is a canonical isomorphism $\CH_n(\sX) \otimes \Lambda \simeq \CH_n(\sX)_\Lambda$ for all $n \in \Z$.
\end{prop}
\begin{proof}
We first show that for every set $I$, the canonical map $\bigoplus_I \mbf{z}_n(\sX)_\Z \to \mbf{z}_n(\sX)_{\Z^{(I)}}$ is invertible.
Using the localization triangle as in the proof of \lemref{lem:Chow conn}, we may assume that $\sX = [X/\mrm{GL}_N]$.
By \lemref{lem:Chow approx}, whose proof applies verbatim with coefficients in $\Z^{(I)}$, we may further assume $\sX=X$ is an algebraic space.
Finally, by localization, induction on the dimension, and \lemref{lem:Chow generic}, we are reduced to the case where $X$ is a smooth affine scheme of pure dimension $d$.
In this case, \thmref{thm:poincare} identifies $\mbf{z}_n(X)_{(-)}$ with $\Maps_{\SH(k)}(\Sigma^\infty_{\bT} X_+, (-)\vb{d-n})$, which commutes with direct sums since $\Sigma^\infty_{\bT}(X_+)$ is compact.

Now choose a short exact sequence $0 \to P_1 \to P_0 \to \Lambda \to 0$ with $P_0$ and $P_1$ free abelian groups.
The preceding discussion shows that the induced exact triangle on $\mbf{z}_n(\sX)_{(-)}$ takes the form
\[\mbf{z}_n(\sX)_\Z \otimes P_1 \to \mbf{z}_n(\sX)_\Z \otimes P_0 \to \mbf{z}_n(\sX)_\Lambda.\]
Since $\pi_{-1}\, \mbf{z}_n(\sX)_\Z = \CH_n(\sX, -1) = 0$ (\lemref{lem:Chow conn}), we find that $\CH_n(\sX)_\Lambda$ is identified with $\on{Coker}(\CH_n(\sX) \otimes P_1 \to \CH_n(\sX) \otimes P_0) \simeq \CH_n(\sX) \otimes \Lambda$.
\end{proof}

\ssec{Cycle classes}\label{ssec:Chow/classes}

In this subsection we construct the cycle class $[\sV] \in \CH_d(\sX)$ of an integral closed substack $\sV \subseteq \sX$ of dimension $d$ (\defnref{defn:Chow cycle class}).
We also show that $\CH_n(\sX)$ is freely generated by these classes when $n \ge \dim(\sX)$ (\lemref{lem:Chow top}), and describe their behaviour under proper direct image (\lemref{lem:Chow deg}).

\begin{lem}\label{lem:Chow fund}
  Let $\sV$ be a good stack, integral of dimension $d$.
  There exists a unique class $[\sV] \in \CH_d(\sV)$ whose restriction to every dense open $\sV^\circ \subseteq \sV$ smooth over $k$ is the fundamental class $[\sV^\circ]$.
\end{lem}
\begin{proof}
Choose a smooth atlas $p : V \twoheadrightarrow \sV$ by a scheme $V$ and let $V_{\mrm{sm}}$ denote the smooth locus of $V$, which is dense since $k$ is perfect.
Its image in $\sV$ is an open substack $\sV_{\mrm{sm}} \subseteq \sV$, which is dense since $p$ is surjective, and contains every open substack of $\sV$ smooth over $k$.
Since $\sV$ is integral, $\dim(\sV \setminus \sV_{\mrm{sm}}) < d$, so by \corref{cor:Chow delete} the restriction $j^! : \CH_d(\sV) \to \CH_d(\sV_{\mrm{sm}})$ is bijective.
We denote by $[\sV] \in \CH_d(\sV)$ the unique class satisfying $j^![\sV] = [\sV_{\mrm{sm}}]$ (\defnref{defn:fund smooth}).
\end{proof}

\begin{defn}\label{defn:Chow cycle class}
  For a closed substack $\sV \subseteq \sX$ as in \lemref{lem:Chow fund}, we also write $[\sV] \in \CH_d(\sX)$ for the direct image of $[\sV] \in \CH_d(\sV)$, and call it the \emph{cycle class} of $\sV$.

  More generally, for a formal sum $\alpha = \sum_\sV m_\sV \sV$ of such closed substacks of dimension $d$, we write $[\alpha] := \sum_\sV m_\sV [\sV] \in \CH_d(\sX)$.
\end{defn}

\begin{rmk}\label{rmk:p^! cycle class}
  Given a smooth morphism $p : \sY \to \sX$, we have $p^![\sV] = [\sV \fibprod_\sX \sY]$ for every integral closed substack $\sV$.
  By construction, this reduces to the case where $\sV$ is smooth, in which case it follows from the functoriality of smooth Gysin maps.
\end{rmk}

\begin{lem}\label{lem:Chow top}
  Let $\sX$ be a good stack and let $n \ge \dim(\sX)$.
  Then $\CH_n(\sX, s) = 0$ for all $s > 0$, and $\CH_n(\sX)$ is the free abelian group on the classes $[\sV]$, where $\sV \subseteq \sX$ ranges over integral closed substacks of dimension $n$.
\end{lem}
\begin{proof}
By \thmref{thm:BM/localization} we may assume that $\sX$ is reduced.
Let $\sU_j \subseteq \sX_j$ be as in \lemref{lem:Chow generic}, and let $\sU \subseteq \sX$ be the union of those $\sU_j$ with $\dim(\sX_j) = n$.
Then the reduced complement $\sZ := \sX \setminus \sU$ has $\dim(\sZ) < n$, and the $\sX_j$ of dimension $n$ are exactly the integral closed substacks of $\sX$ of dimension $n$.
By \corref{cor:Chow delete}, the restriction maps induce a bijection $\CH_n(\sX, s) \to \bigoplus_j \CH_n(\sU_j, s)$ for all $s$, which by \lemref{lem:Chow fund} sends $[\sX_j]$ to $[\sU_j]$.
We may thus assume that $\sX$ is smooth and integral of dimension $n$, and show that $\CH_n(\sX, s) = 0$ for $s > 0$ and $\CH_n(\sX) = \Z \cdot [\sX]$.

Since $\sX$ is smooth of dimension $n$, we have the isomorphism $\CH_n(\sX, s) \simeq \H^{-s}(\sX, \Z^\mot)$, sending $[\sX]$ to $1$ (\thmref{thm:poincare}).
By \rmkref{rmk:Gamma^BM as limit}, $\Ccoh(\sX; \Z^\mot)$ is the limit of $\Ccoh(T; \Z^\mot)$ over smooth morphisms $T \to \sX$ with $T$ a scheme.
For each such $T$, $\H^p(T, \Z^\mot)$ vanishes for $p \ne 0$ and is the group of locally constant $\Z$-valued functions on $T$ for $p = 0$ (see \cite[Thm.~4.24]{Hoyois}, \cite[Thms.~7.6, 7.18, 8.22, Lem.~8.23]{Spitzweck}, \cite[Cor.~4.2]{MazzaWeibelVoevodsky}).
In particular, $\Ccoh(T; \Z^\mot)$ is coconnective, hence so is $\Ccoh(\sX; \Z^\mot)$, and $\H^0(\sX, \Z^\mot)$ is the limit of the groups $\H^0(T, \Z^\mot)$, i.e., it is the group of locally constant functions on the underlying topological space of $\sX$.
Since $\sX$ is irreducible, this latter is $\Z$.
\end{proof}

\begin{constr}\label{constr:Chow psi}
Let $k$ be a perfect field and $X$ a smooth separated scheme of finite type over $k$, integral of dimension $d$, and let $n \le d$.
\lemref{lem:Chow Psi} provides an isomorphism of spectra
\[
  \Psi_X : z_n(X, \bullet) \simeq \mbf{z}_n(X) = \CBM(X; \Z^\mot)\vb{-n},
\]
compatible with restriction to opens.
On $\pi_0$, it induces an isomorphism
\[
  \psi_X : \CH_n^{\mrm{cyc}}(X) \simeq \CH_n(X).
\]
\end{constr}

\begin{lem}\label{lem:Chow deg}
  Let $f : \sY \to \sX$ be a proper representable morphism of good stacks, with $\sY$ integral of dimension $n$, and let $\sV \subseteq \sX$ be its reduced image.
  Then we have:
  \begin{thmlist}
  \item
  If $\dim(\sV) = n$, then $f_*[\sY] = \deg(\sY/\sV) \cdot [\sV]$ in $\CH_n(\sX)$.\footnote{%
    Here, $\deg(\sY/\sV) := \deg(\sY \fibprod_\sV V / V)$ for any smooth morphism $V \to \sV$ from an integral scheme.
    This is independent of $V$ by the argument of \cite[Lem.~1.16]{Vistoli}.
  }

  \item
  Otherwise, $f_*[\sY] = 0$.
  \end{thmlist}
\end{lem}
\begin{proof}
We use the argument of \cite[Prop.~3.11]{JinBM}.
If $\dim(\sV) < n$, then $\CH_n(\sV) = 0$ by \lemref{lem:Chow dim}.
Otherwise $\sV$ is integral of dimension $n$, so $f_*[\sY] = m \cdot [\sV]$ for some $m \in \Z$ (\lemref{lem:Chow top}), and it remains to compute $m$.
Let $p : V \to \sV$ be a smooth morphism of relative dimension $e$ from an integral scheme $V$, and $f' : Y' := \sY \fibprod_\sV V \to V$ the base change of $f$.
Then $f'_*[Y'] = p^!f_*[\sY] = m \cdot p^![\sV] = m \cdot [V]$ in $\CH_{n+e}(V)$, which is freely generated by $[V]$ (\lemref{lem:Chow top}).

Writing $[Y']$ as the sum of the classes of its irreducible components $Y_i$, those not dominating $V$ contribute zero by \lemref{lem:Chow dim}.
For the others, $Y_i \to V$ is a dominant morphism with $\dim(V) = \dim(Y_i) = n+e$, so by generic finiteness and \corref{cor:Chow delete} we may assume $Y_i \to V$ is finite and $Y_i$ and $V$ are smooth.
In this case, Poincar\'e duality (\thmref{thm:poincare}) identifies $f'_* : \CH_{n+e}(Y_i) \to \CH_{n+e}(V)$ with the Gysin map $\H^0(Y_i, \Z^\mot) \to \H^0(V, \Z^\mot)$ of \cite[\S 4.3.3(iii)]{DegliseJinKhan}, which is multiplication by $[k(Y_i):k(V)]$ on $\H^0(-, \Z^\mot) \simeq \Z$ by \cite[Ex.~3.2.9(1)]{DegliseOrientation}; hence $f'_*[Y_i] = [k(Y_i):k(V)] \cdot [V]$.
Taking the sum over $i$ gives $m = \deg(\sY/\sV)$.
\end{proof}

\ssec{Cycle class map}\label{ssec:Chow/cl}

In this subsection we construct the cycle class map $\mrm{cl}_\sX$ (\propref{prop:cycle class map}) and prove the criterion of \lemref{lem:Chow devissage}.
The key input is the formula of \lemref{lem:Chow res}, which computes the boundary of the class of a unit as the cycle class of its divisor; in particular, we use this to show that the cycle class map kills cycles rationally equivalent to zero (\lemref{lem:Chow div}).

Recall from \eqref{eq:split CBM Gm} the splitting $\CBM(\bG_m/k; \Z^\mot)[-1] \simeq R \oplus R(1)[1]$, where $R := \Ccoh(\Spec(k); \Z^\mot)$.
By Poincar\'e duality (\thmref{thm:poincare}), cap product with $[\bG_m]$ determines an isomorphism $\Ccoh(\bG_m; \Z^\mot) \simeq \CBM(\bG_m/k; \Z^\mot)(-1)[-2] \simeq R(-1)[-1] \oplus R$.
We denote by \[t \in \Ccoh(\bG_m; \Z^\mot(1))[1]\] the unique class with $t \cap [\bG_m] \simeq 1_{R}$, where $1_R \in R$ is the unit (via the inclusion of the first summand), and we write also $t \in \H^1(\bG_m, \Z^\mot(1))$ for its image in $\pi_0$.

We begin with the following key statement, which in the case of schemes refines \cite[Part~I, Ch.~VI, Prop.~1.1.11]{LevineMixed} and \cite[Lem.~3.8]{JinBM} to singular schemes and integral coefficients.

\begin{lem}\label{lem:Chow res}
  Let $\sX$ be a reduced good stack, and let $\sU \subseteq \sX$ be an open substack, smooth of pure dimension $n+1$ over $k$, with complement $\sZ = \sX\setminus\sU$ of dimension $\le n$.
  Let $f \in \Gamma(\sU, \sO_\sU)^\times$ and for each connected component $\sU_j$ of $\sU$, consider the cycle class $[\mrm{div}(f|_{\sU_j})] \in \CH_n(\sZ)$ of the Weil divisor of $f|_{\sU_j}$ on the closure of $\sU_j$.
  Then the boundary $\partial : \CH_n(\sU, 1) \to \CH_n(\sZ)$ satisfies
  \[ \partial\big(f^*(t) \cap [\sU]\big) = \sum_j [\mrm{div}(f|_{\sU_j})], \]
  where $f^*(t)$ denotes the inverse image of $t$ along $f : \sU \to \bG_m$.
\end{lem}
\begin{proof}
Let $p : X \twoheadrightarrow \sX$ be a smooth surjection of relative dimension $e$ with $X$ affine, and let $U$ and $Z$ be the inverse images of $\sU$ and $\sZ$.
By \lemref{lem:Chow top}, $\CH_n(\sZ)$ is freely generated by the classes of the irreducible components $\sZ_i$, and $p^! : \CH_n(\sZ) \to \CH_{n+e}(Z)$ sends $[\sZ_i]$ to the sum of the classes of the irreducible components of $Z$ lying over $\sZ_i$; in particular, it is injective on $\CH_n(\sZ)$.

Now $p^!$ commutes with $\partial$, sends $f^*(t) \cap [\sU]$ to $(f \circ p)^*(t) \cap [U]$, and sends the cycle class of the Weil divisor of $f|_{\sU_j}$ to that of $f \circ p$ on the components of $U$ that lie over $\sU_j$ (by \rmkref{rmk:p^! cycle class}).
Replacing $\sX$ by $X$ and $\sU$ by $U$ (which is smooth of pure dimension $n+1+e$) and $\sZ$ by $Z$ (which has $\dim(Z) \le n+e$), we may thus assume that $\sX = X$ is an affine scheme.

We first treat the case where $X$ is smooth, $Z$ is a smooth integral divisor cut out by a function $g$ with $g : X \to \bA^1$ smooth, and $f = ug^m$ with $u \in \Gamma(X, \sO_X)^\times$ and $m \in \Z$; we show that $\partial(f^*(t) \cap [U]) = m \cdot [Z]$.
Since $f \mapsto f^*(t)$ is a homomorphism $\Gamma(U, \sO_U)^\times \to \H^1(U, \Z^\mot(1))$ (see \cite[Thm.~4.1]{MazzaWeibelVoevodsky}, \cite[Lems.~6.5, 6.6]{Spitzweck}), we have $f^*(t) = u^*(t) + m \cdot g^*(t)$.
The class $u^*(t) \cap [U]$ is the restriction of $u^*(t) \cap [X]$, so it is killed by the boundary $\partial$.
Since $g^!$ commutes with $\partial$, \rmkref{rmk:p^! cycle class} yields $\partial(g^*(t) \cap [U]) = g^!\partial(t \cap [\bG_m]) = g^![0] = [Z]$, where $\partial(t \cap [\bG_m]) = [0]$ in $\CH_0(\{0\})$ by definition of $t$.

Next suppose $X$ is normal and $U$ is dense in $X$, so that the closures of the $U_j$ are the connected components $X_j$ of $X$.
The desired equality is in $\CH_n(Z)$, which is free on the classes of the $n$-dimensional irreducible components $Z_i$ of $Z$ (\lemref{lem:Chow top}), and the coefficient of $[Z_i]$ may be computed after restriction to any open $W \subseteq X$ meeting $Z_i$ with $W \cap Z = W \cap Z_i$.
Since $X$ is normal and $k$ is perfect, $X$ is smooth at its codimension one points, including the generic points $\eta_i$ of the $Z_i$.
In particular, each $\eta_i$ admits a smooth open neighbourhood $W_i$ in $X$.
Since $\sO_{X,\eta_i}$ is a discrete valuation ring, with uniformizer $g$ say, after shrinking $W_i$ we may assume that $g : W_i \to \bA^1$ is smooth with $g^{-1}(0) = W_i \cap Z = W_i \cap Z_i$, and that $f|_{W_i \cap U} = ug^m$ with $u \in \Gamma(W_i, \sO_{W_i})^\times$ and $m \in \Z$ the multiplicity of $Z_i$ in $\sum_j \mrm{div}(f|_{U_j})$.
Now by the special case treated above, the coefficient of $[Z_i]$ in $\partial(f^*(t) \cap [U])$ is $m$, whence the claim in this case.

In general, both sides are direct images from $\overline{U} \cap Z$, so replacing $X$ by $\overline{U}$ we may assume that $U$ is dense in $X$.
Let $\nu : \widetilde{X} \to X$ be the normalization, which is finite and an isomorphism over $U$.
Then, writing $\widetilde{\partial}$ for the boundary of $\nu^{-1}(Z) \subseteq \widetilde{X}$, we have by the normal case
\begin{align*}
  \partial(f^*(t) \cap [U])
  &= \nu_*\widetilde{\partial}(f^*(t) \cap [U]) \\
  &= \nu_*\textstyle\sum_j [\mrm{div}_{\widetilde{X}}(f|_{U_j})] \\
  &= \textstyle\sum_j [\nu_*\mrm{div}_{\widetilde{X}}(f|_{U_j})] && \text{(\lemref{lem:Chow deg})} \\
  &= \textstyle\sum_j [\mrm{div}(f|_{U_j})] && \text{(\cite[Ex.~1.2.3]{Fulton})}.\qedhere
\end{align*}
\end{proof}

\begin{lem}\label{lem:Chow div}
  Let $\sX$ be a good stack, let $\sY \subseteq \sX$ be an integral closed substack of dimension $n+1$, and let $f$ be a nonzero rational function on $\sY$.
  Write $\mrm{div}(f) = \sum_\sV m_\sV \sV$ for its Weil divisor.
  Then $[\mrm{div}(f)] = 0$ in $\CH_n(\sX)$.
\end{lem}
\begin{proof}
By \lemref{lem:Chow generic}, there is a smooth dense open $\sU \subseteq \sY$.
Shrinking it if necessary, we may assume that $f \in \sO(\sU)^\times$, and hence that $\mrm{div}(f)$ is supported on the reduced complement $\sZ := \sY \setminus \sU$.
Since $\sZ$ is of dimension $\le n$, \lemref{lem:Chow res} yields $[\mrm{div}(f)] = \partial(f^*(t) \cap [\sU])$ in $\CH_n(\sZ)$, where $\partial : \CH_n(\sU, 1) \to \CH_n(\sZ)$ is the boundary map.
Applying direct image along $\sZ \hookrightarrow \sY$, this implies $[\mrm{div}(f)] = 0$ in $\CH_n(\sY)$, hence also in $\CH_n(\sX)$.
\end{proof}

\begin{prop}\label{prop:cycle class map}
Let $\sX$ be a good stack.
There exists a unique homomorphism $\mrm{cl}_\sX : \CH^{\mrm{cyc}}_n(\sX) \to \CH_n(\sX)$ satisfying the following properties:
\begin{thmlist}
  \item The homomorphisms $\mrm{cl}$ commute with direct image along projective representable morphisms and with Gysin maps along smooth morphisms.
  \item For every integral closed substack $\sV \subseteq \sX$ of dimension $n$, $\mrm{cl}_\sX$ sends the class of $\sV$ in $\CH_n^{\mrm{cyc}}(\sX)$ to $[\sV]$.
\end{thmlist}
\end{prop}
\begin{proof}
Recall from \cite[\S 2.1]{Kresch} that
\[
  \CH^{\mrm{cyc}}_n(\sX) := \colim_{h : \sT \to \sX} \widehat{A}_n(\sT)/\widehat{B},
  \qquad
  \widehat{A}_n(\sT) := \colim_{p : \sE \to \sT} A^\circ_{n+r}(\sE),
\]
where $A^\circ_{n+r}(\sE)$ is the group of $(n+r)$-cycles on $\sE$ modulo rational equivalence, and $p$ ranges over vector bundles of rank $r$, along surjections of vector bundles; $\widehat{B} \subseteq \widehat{A}_n(\sT)$ is the subgroup of relations of \cite[Def.~2.1.9]{Kresch}, and $h$ ranges over projective morphisms, along inclusions of components.
We write $(h, p, \alpha)$ for the class of an $(n+r)$-cycle $\alpha$ on $\sE$.

We set
\[
  \mrm{cl}_\sX(h, p, \alpha) := h_*(p^!)^{-1}[\alpha] \in \CH_n(\sX),
\]
where $p^! : \CH_n(\sT) \to \CH_{n+r}(\sE)$ is invertible by homotopy invariance.
We check that it descends to each quotient and is compatible with the transition maps in each colimit:
\begin{enumerate}
  \item \emph{Rational equivalence on $\sE$} \cite[Def.~2.1.4(i)]{Kresch}.
  This is killed by \lemref{lem:Chow div}.

  \item \emph{Surjections of vector bundles} \cite[Def.~2.1.4(iii)]{Kresch}.
  For a surjection $q : \sF \to \sE$ of vector bundles on $\sT$, we have $(p \circ q)^!(p^!)^{-1}[\alpha] = q^![\alpha] = [q^{-1}\alpha]$ by \rmkref{rmk:p^! cycle class}.
  Hence $(p^!)^{-1}[\alpha]$ only depends on the image of $\alpha$ in $\widehat{A}_n(\sT)$.

  \item \emph{The relations of \cite[Def.~2.1.9]{Kresch}.}
  Let $q : \sS \to \sT$ be projective, $p_\sS : \sE_\sS \to \sS$ the base change of $p$, $q_\sE : \sE_\sS \to \sE$ the projection, and $\theta$ an $(n+r)$-cycle on $\sE_\sS$.
  Then $[q_{\sE,*}\theta] = q_{\sE,*}[\theta]$ (\lemref{lem:Chow deg}) and $p^!q_* = q_{\sE,*}p_\sS^!$ (base change) give
  \[ (p^!)^{-1}[q_{\sE,*}\theta] = q_*(p_\sS^!)^{-1}[\theta]. \]
  That is, $(p^!)^{-1}[-]$ intertwines the restricted direct image of \cite[Def.~2.1.4(iv--v)]{Kresch} with $q_*$.
  The relations are differences $q_{2*}(\eta_2) - q_{1*}(\eta_1)$ of such restricted direct images, where $q_1, q_2 : \sS \to \sT$ are projective with $h\circ q_1 \simeq h\circ q_2$ and $\eta_1$, $\eta_2$ have the same image in $\widehat{A}_n(\sS)$.
  Hence $h_*(p^!)^{-1}[-]$ sends them to $(h\circ q_2)_*(y) - (h\circ q_1)_*(y) = 0$, where $y \in \CH_n(\sS)$ is the common value on $\eta_1$ and $\eta_2$.

  \item \emph{Inclusions of components} \cite[Def.~2.1.11]{Kresch}.
  For the inclusion $\iota : \sT' \hookrightarrow \sT$ of a union of connected components, the transition map is direct image along $\iota$.
  Since $\iota$ induces an isomorphism $\widehat{A}^\iota_n(\sT') \simeq \widehat{A}_n(\sT')$ \cite[Rem.~2.1.8]{Kresch}, the identity of the previous item with $q = \iota$ gives $\mrm{cl}_\sX(h \circ \iota, p_{\sT'}, \theta) = \mrm{cl}_\sX(h, p, \iota_{\sE,*}\theta)$.
\end{enumerate}

Thus $\mrm{cl}_\sX$ descends to a homomorphism $\CH^{\mrm{cyc}}_n(\sX) \to \CH_n(\sX)$.
Property (ii) holds since the class of $\sV$ is $(\mrm{id}_\sX, \mrm{id}_\sX, \sV)$.
For (i), direct image along a projective $g : \sX \to \sY$ sends $(h, p, \alpha)$ to $(g\circ h, p, \alpha)$, and we have $\mrm{cl}_\sY(g\circ h, p, \alpha) = g_*\mrm{cl}_\sX(h, p, \alpha)$ since $(g\circ h)_* = g_*h_*$.
Inverse image along a smooth $g : \sY \to \sX$ sends $(h, p, \alpha)$ to $(h_\sY, p_\sY, g_\sE^{-1}(\alpha))$, where $h_\sY$, $p_\sY$ and $g_\sE : \sE_\sY \to \sE$ are the base changes, and we have
\[
  g^!\on{cl}_\sX(h, p, \alpha) = g^!h_*(p^!)^{-1}[\alpha] = h_{\sY,*}(p_\sY^!)^{-1}g_\sE^![\alpha] = \mrm{cl}_\sY(h_\sY, p_\sY, g_\sE^{-1}(\alpha))
\]
by base change and \rmkref{rmk:p^! cycle class}.

For uniqueness, let $\mrm{cl}'$ be another such family of homomorphisms.
In $\CH^{\mrm{cyc}}_n(\sX)$ we have $(h, p, \alpha) = h_*(p^!)^{-1}[\alpha]$, where $[\alpha] \in \CH^{\mrm{cyc}}_{n+r}(\sE)$ is the class of $\alpha$ and $p^!$ is invertible on $\CH^{\mrm{cyc}}_*$ by \cite[Cor.~2.4.9]{Kresch}.
(Indeed, $p^!(\mrm{id}_\sT, p, \alpha)$ and $[\alpha]$ are represented by the inverse images of $\alpha$ along the two projections $\sE \fibprod_\sT \sE \to \sE$, which are surjections of vector bundles on $\sT$ and hence induce the same map on $A^\circ_*$ by \cite[Rem.~2.1.5]{Kresch}.)
Hence (i) and (ii) give $\mrm{cl}'_\sX(h, p, \alpha) = h_*(p^!)^{-1}\mrm{cl}'_\sE[\alpha] = h_*(p^!)^{-1}[\alpha] = \mrm{cl}_\sX(h, p, \alpha)$.
\end{proof}

\begin{lem}\label{lem:Chow devissage}
  Let $\sX$ be a good stack and let $i : \sZ \hookrightarrow \sX$ be a closed immersion with $\sU=\sX\setminus\sZ$.
  Suppose that $\mrm{cl}_{\sZ}$ and $\mrm{cl}_{\sU}$ are invertible, and the following condition holds:
  \begin{enumerate}
	\item [(K)] Every element of the kernel of $i_* : \CH_n(\sZ) \to \CH_n(\sX)$ is the image under $\mrm{cl}_\sZ$ of an element of the kernel of $i_* : \CH_n^{\mrm{cyc}}(\sZ) \to \CH_n^{\mrm{cyc}}(\sX)$.
	Equivalently, for every $\beta \in \CH_n(\sU, 1)$ there exists $\zeta \in \CH_n^{\mrm{cyc}}(\sZ)$ with $i_*\zeta = 0$ and $\mrm{cl}_\sZ(\zeta) = \partial(\beta)$.
  \end{enumerate}
  Then $\mrm{cl}_\sX$ is invertible.
\end{lem}
\begin{proof}
Write $j : \sU \hookrightarrow \sX$ for the open immersion.
By \propref{prop:cycle class map} we have the commutative diagram
\[\begin{tikzcd}[column sep=small]
    & \CH^{\mrm{cyc}}_n(\sZ) \ar{r}{i_*}\ar{d}{\cl_{\sZ}}
    & \CH^{\mrm{cyc}}_n(\sX) \ar{r}{j^!}\ar{d}{\cl_{\sX}}
    & \CH^{\mrm{cyc}}_n(\sU) \ar{r}\ar{d}{\cl_{\sU}}
    & 0
    \\
    \CH_n(\sU, 1) \ar{r}{\partial}
    & \CH_n(\sZ) \ar{r}{i_*}
    & \CH_n(\sX) \ar{r}{j^!}
    & \CH_n(\sU) \ar{r}
    & 0.
\end{tikzcd}\]
The upper row is exact by \cite[Prop.~2.3.6]{Kresch}.
The lower row is part of the long exact localization sequence (\thmref{thm:BM/localization}), whose next term $\CH_n(\sZ, -1)$ vanishes by \lemref{lem:Chow conn}.
Now consider the diagram obtained from this one by adding the respective kernels of $i_*$ in the left-most column (in place of $\CH_n(\sU,1)$ in the lower row), and the vertical map given by the restriction of $\mrm{cl}_\sZ$.
The latter is surjective by (K), so the claim follows from the five lemma.
\end{proof}

\ssec{Algebraic spaces}\label{ssec:Chow/spaces}

\begin{lem}\label{lem:Chow smooth}
  Let $k$ be a perfect field and $X$ a smooth separated scheme of finite type over $k$.
  Then $\mrm{cl}_X : \CH_n^{\mrm{cyc}}(X) \to \CH_n(X)$ is invertible for all $n \in \Z$.
\end{lem}
\begin{proof}
Both sides take finite disjoint unions to direct sums, compatibly with $\mrm{cl}$, so we may assume $X$ connected, hence integral, of dimension $d$.
If $n > d$, both sides vanish (\lemref{lem:Chow dim}), so let $q := d-n \ge 0$, and let $\psi_X : \CH_n^{\mrm{cyc}}(X) \simeq \CH_n(X)$ be the isomorphism of \constrref{constr:Chow psi}, and $\Psi_X$ its spectrum-level lift discussed there.
It will suffice to show that $\cl_X = \psi_X$, for which we need to check that $\psi_X$ sends the class of an $n$-dimensional integral closed subvariety $V$ to the cycle class $[V]$ in the sense of \defnref{defn:Chow cycle class}.
As in the proof of \thmref{thm:Chow Bloch}, the isomorphisms $\Psi_X$ and $\Psi_{X\setminus V}$ give, on fibres, an isomorphism of spectra
\[
  z_n(V, \bullet) \simeq \CBM(V; \Z^\mot)\vb{-n}.
\]
On $\pi_0$ this gives an isomorphism $\psi_V : \CH_n^{\mrm{cyc}}(V) \to \CH_n(V)$, which is intertwined with $\psi_X$ under the respective pushforward maps.
Now $\CH_n^{\mrm{cyc}}(V)$ is the free abelian group generated by the class of $V$, and by \lemref{lem:Chow top}, $\CH_n(V)$ is the free abelian group generated by the cycle class $[V]$.
Thus $\psi_V$, hence $\psi_X$, sends the class of $V$ to $\pm[V]$.

It remains to pin down the sign.
To do so we may shrink $X$ around the generic point of $V$.
Indeed, for an open $X' \subseteq X$ meeting $V$, restriction intertwines $\psi_V$ and $\psi_{V \cap X'}$ (\lemref{lem:Chow Psi}), and sends the generators $V$ and $[V]$ to $V \cap X'$ and $[V \cap X']$ (\lemref{lem:Chow fund}).
Since $V$ is generically smooth, we may thus assume that $V = g^{-1}(0)$ for a smooth morphism $g : X \to \bA^q_k$.

By the compatibility of the isomorphisms of \lemref{lem:Chow Psi} with inverse image along $g$, the Gysin maps $g^!$ intertwine $\psi_{\{0\}}$, formed for $\{0\} \subseteq \bA^q_k$, with $\psi_V$.
Since $g^!$ sends $[\{0\}]$ to $[V]$ (\rmkref{rmk:p^! cycle class}), we are reduced to the case $X = \bA^q_k$ and $V = \{0\}$.
In this case, $\psi_V$ factors via the isomorphisms
\[
  \psi_V : \CH_0^{\mrm{cyc}}(\{0\}) \simeq \CH^q_{\{0\}}(\bA^q_k) \xrightarrow{(-)\cap[\bA^q_k]} \CH_0(\{0\}),
\]
where $\CH^q_{\{0\}}(\bA^q_k)$ is $\pi_0$ of the fibre of $\mbf{z}^q(\bA^q_k) \to \mbf{z}^q(\bA^q_k\setminus \{0\})$.

The first isomorphism is induced on fibres by the isomorphisms $z^q(-, \bullet) \simeq \mbf{z}^q(-)$ of the proof of \lemref{lem:Chow Psi}, and sends $\{0\}$ to the Thom class of the trivial bundle of rank $q$, i.e., the $q$-fold $\bT$-suspension of the unit, by \cite[Prop.~6.4, Lems.~6.5, 6.6]{Spitzweck} (note that the comparisons of \cite[Thms.~7.14, 7.18]{Spitzweck} preserve the unit).
The second isomorphism $(-)\cap[\bA^q_k]$ sends this Thom class to $[\{0\}]$ by \cite[Lem.~3.2.15, Ex.~3.2.18, Thm.~3.3.2(i), \S 4.4.3]{DegliseJinKhan}.
Thus $\psi_V$ sends the class of $V$ to $[V]$, as claimed.
\end{proof}

In the case of quasi-projective schemes, the following lemma can be deduced from the coniveau spectral sequence of \cite[(10.0.1)]{Bloch} (see also \cite[Ex.~5.5.4]{DegliseOrientation} for regular schemes and with rational coefficients).

\begin{lem}\label{lem:Chow units}
  Let $X$ be a good algebraic space and let $\beta \in \CH_n(X, 1)$.
  Then there exist a reduced closed subspace $i : V \hookrightarrow X$, a dense open $V^\circ \subseteq V$, a class $\gamma \in \CH_n(V, 1)$, and a unit $f \in \Gamma(V^\circ, \sO_{V^\circ})^\times$ such that:
  \begin{thmlist}
    \item $\dim(V) \le n+1$, and $i_*\gamma = \beta$.
    \item $V^\circ$ is a smooth affine scheme of pure dimension $n+1$, its complement $V \setminus V^\circ$ has dimension $\le n$, and $\gamma|_{V^\circ} = f^*(t) \cap [V^\circ]$.
    \item The cycle $\sum_j \mrm{div}(f|_{V^\circ_j})$ on $V$ is zero, where $V^\circ_j$ ranges over the connected components of $V^\circ$ and $\mrm{div}(f|_{V^\circ_j})$ is the Weil divisor of $f|_{V^\circ_j}$ on the closure of $V^\circ_j$.
  \end{thmlist}
\end{lem}
\begin{proof}
We argue by induction on $d = \dim(X)$.
Note that $X$ can be assumed reduced.

Suppose first that $d > n+1$.
Let $U \subseteq X$ be a dense open which is a finite disjoint union of smooth integral affine schemes $U_j$ (\lemref{lem:Chow generic}), and let $q_j := \dim(U_j) - n$.
By \lemref{lem:Chow smooth MC}, $\CH_n(U_j, 1) \simeq \H^{2q_j-1}(U_j, \Z^\mot(q_j))$ is identified with motivic cohomology as defined in \cite[Def.~3.4]{MazzaWeibelVoevodsky}.
If $\dim(U_j) = d$, then $2q_j-1 > q_j$, so the colimit of $\H^{2q_j-1}(U', \Z^\mot(q_j))$ over nonempty opens $U' \subseteq U_j$ vanishes (see \cite[Def.~3.1, Lem.~3.9, Lem.~5.2]{MazzaWeibelVoevodsky}), and hence the restriction of $\beta$ to some nonempty open $U'_j \subseteq U_j$ is zero.
The reduced complement $X'$ of the union of these $U'_j$ has $\dim(X') < d$, so by the localization triangle $\beta$ is the direct image of some $\beta' \in \CH_n(X', 1)$.
Hence the claim follows from the induction hypothesis applied to $\beta'$.

Now suppose that $d \le n+1$.
Let $i : V \hookrightarrow X$ be the reduced union of the irreducible components of $X$ of dimension $n+1$.
Since $\dim(X \setminus V) \le n$, we have $\CH_n(X \setminus V, 1) = 0$ by \lemref{lem:Chow top}, so by the localization triangle, $\beta = i_*\gamma$ for some $\gamma \in \CH_n(V, 1)$.
Let $V^\circ \subseteq V$ be a dense open which is a finite disjoint union of smooth integral affine schemes $V^\circ_j$ (\lemref{lem:Chow generic}); each has dimension $n+1$, and $\dim(V \setminus V^\circ) \le n$.
By Poincar\'e duality (\thmref{thm:poincare}), $\CH_n(V^\circ_j, 1) \simeq \H^1(V^\circ_j, \Z^\mot(1))$, and $f \mapsto f^*(t)$ defines an isomorphism $\Gamma(V^\circ_j, \sO_{V^\circ_j})^\times \simeq \H^1(V^\circ_j, \Z^\mot(1))$ by \cite[Thm.~4.1]{MazzaWeibelVoevodsky} and \cite[Lems.~6.5, 6.6]{Spitzweck}.
Thus there is a unit $f \in \Gamma(V^\circ, \sO_{V^\circ})^\times$ with $\gamma|_{V^\circ} = f^*(t) \cap [V^\circ]$.
Since $\gamma|_{V^\circ}$ is restricted from $V$, its boundary $\partial(\gamma|_{V^\circ})$ on $Z := V \setminus V^\circ$ vanishes by the localization triangle.
Then \lemref{lem:Chow res} yields $\sum_j [\mrm{div}(f|_{V^\circ_j})] = 0$ in $\CH_n(Z)$.
Hence the cycle $\sum_j \mrm{div}(f|_{V^\circ_j})$ itself is zero by \lemref{lem:Chow top} (since $\dim(Z) \le n$).
The claim is proven.
\end{proof}

\begin{lem}\label{lem:Chow K cycle}
  Let $\sX$ be a good stack and $i : \sZ \hookrightarrow \sX$ a closed immersion whose open complement $U := \sX \setminus \sZ$ is an algebraic space.
  Then for every $\beta \in \CH_n(U, 1)$ there exists an $n$-cycle $\alpha$ on $\sZ$ such that:
  \begin{thmlist}
    \item\label{item:Chow K cycle/div}
    $i_*\alpha$ is rationally equivalent to zero; and
    \item\label{item:Chow K cycle/cl}
    $[\alpha] = \partial(\beta)$ in $\CH_n(\sZ)$.
  \end{thmlist}
\end{lem}
\begin{proof}
By \lemref{lem:Chow units} applied to $\beta$, there exist a reduced closed subspace $V \subseteq U$, a dense open $V^\circ \subseteq V$, a class $\gamma \in \CH_n(V, 1)$ whose direct image to $U$ is $\beta$, and a unit $f \in \Gamma(V^\circ, \sO_{V^\circ})^\times$ satisfying the conditions there.
Let $\overline{V} \subseteq \sX$ be the reduced closure of $V$, so that $\overline{V} \cap U = V$, and let $c : C \hookrightarrow \sZ$ be the inclusion of the reduced complement $C := \overline{V} \setminus V$.
By naturality of the localization triangle with respect to proper covariance (e.g. \cite[Prop.~2.2.10]{DegliseJinKhan}), we have
\[
  \partial(\beta) = c_*\partial_{\overline{V}}(\gamma) \in \CH_n(\sZ),
\]
where $\partial_{\overline{V}} : \CH_n(V, 1) \to \CH_n(C)$ is the boundary map for the closed immersion $C \hookrightarrow \overline{V}$.

The open $\overline{V}^\circ := \overline{V} \setminus \overline{V \setminus V^\circ}$ of $\overline{V}$, where the closure is taken in $\overline{V}$, satisfies $\overline{V}^\circ \cap U = V^\circ$.
Let $D := C \setminus \overline{V}^\circ$, so that $\overline{V}^\circ \setminus V^\circ = C \setminus D$.
Each irreducible component of $\overline{V \setminus V^\circ}$ has dimension $\le n$ and generic point in $V$, hence not in $C$, so $\dim(D) < n$.
By naturality of the localization triangle with respect to restriction to opens (e.g. \cite[Prop.~2.2.10]{DegliseJinKhan}), the restriction of $\partial_{\overline{V}}(\gamma)$ to $C \setminus D$ is $\partial_{\overline{V}^\circ}(\gamma|_{V^\circ})$, where $\partial_{\overline{V}^\circ} : \CH_n(V^\circ, 1) \to \CH_n(C \setminus D)$ is the boundary map for the closed immersion $C \setminus D \hookrightarrow \overline{V}^\circ$.

Since $V^\circ$ is smooth of pure dimension $n+1$, $\gamma|_{V^\circ} = f^*(t) \cap [V^\circ]$, and $\dim(C \setminus D) \le n$ (as $C$ contains no generic point of $\overline{V}$ and $\dim(V) \le n+1$), \lemref{lem:Chow res} applied to $\overline{V}^\circ$ gives
\[
  \partial_{\overline{V}^\circ}(\gamma|_{V^\circ}) = \sum_j [\mrm{div}(f|_{V^\circ_j})] \in \CH_n(C \setminus D),
\]
where $V^\circ_j$ ranges over the connected components of $V^\circ$ and $\mrm{div}(f|_{V^\circ_j})$ is the Weil divisor of $f|_{V^\circ_j}$ on the closure of $V^\circ_j$ in $\overline{V}^\circ$.

Let $\alpha$ be the cycle on $\overline{V}$ given by the sum over $j$ of the Weil divisors of $f|_{V^\circ_j}$ on the closures of $V^\circ_j$ in $\overline{V}$.
Then its direct image on $\sX$ is rationally equivalent to zero, it is supported on $C$ (since its restriction to $V$ is zero), and its restriction to $C \setminus D$ is $\sum_j \mrm{div}(f|_{V^\circ_j})$.
In particular, $[\alpha] \in \CH_n(C)$ restricts to $\sum_j [\mrm{div}(f|_{V^\circ_j})] = \partial_{\overline{V}^\circ}(\gamma|_{V^\circ})$.
Since the restriction map $\CH_n(C) \to \CH_n(C \setminus D)$ is invertible as $\dim(D) < n$ (\corref{cor:Chow delete}), it follows that $[\alpha] = \partial_{\overline{V}}(\gamma)$ in $\CH_n(C)$.
Its direct image on $\sZ$ thus satisfies $[\alpha] = c_*[\alpha] = c_*\partial_{\overline{V}}(\gamma) = \partial(\beta)$ in $\CH_n(\sZ)$.
\end{proof}

\begin{cor}\label{cor:Chow K}
  Let $X$ be a good algebraic space.
  Then every closed immersion $i : Z \hookrightarrow X$ satisfies condition~(K) of \lemref{lem:Chow devissage}.
\end{cor}
\begin{proof}
Let $\beta \in \CH_n(X \setminus Z, 1)$ and let $\alpha$ be as in \lemref{lem:Chow K cycle}.
Then $[\alpha] \in \CH_n^{\mrm{cyc}}(Z)$ satisfies $i_*[\alpha] = 0$ by \lemref{lem:Chow K cycle}\itemref{item:Chow K cycle/div} and the definition of rational equivalence, and $\mrm{cl}_Z[\alpha] = \partial(\beta)$ by \lemref{lem:Chow K cycle}\itemref{item:Chow K cycle/cl}.
\end{proof}

\begin{prop}\label{prop:Chow spaces}
  Let $X$ be a good algebraic space.
  Then $\mrm{cl}_X : \CH_n^{\mrm{cyc}}(X) \to \CH_n(X)$ is invertible for all $n \in \Z$.
\end{prop}
\begin{proof}
We argue by induction on $d = \dim(X)$.
By \lemref{lem:Chow devissage} for the closed immersion $X_{\mrm{red}} \hookrightarrow X$, where (K) holds by \corref{cor:Chow K}, we may assume that $X$ is reduced.
By \lemref{lem:Chow generic}, it admits a dense open $U \subseteq X$ which is a finite disjoint union of smooth integral affine schemes.
Since the reduced complement $Z \subseteq X$ of $U$ has $\dim(Z) < d$, the maps $\mrm{cl}_Z : \CH_n^{\mrm{cyc}}(Z) \to \CH_n(Z)$ are invertible for all $n$ by the induction hypothesis.
Since $\mrm{cl}_U : \CH_n^{\mrm{cyc}}(U) \to \CH_n(U)$ is also invertible for all $n$ by \lemref{lem:Chow smooth}, the claim now follows by \lemref{lem:Chow devissage} and \corref{cor:Chow K}.
\end{proof}

\ssec{Artin stacks}\label{ssec:Chow/stacks}

\begin{lem}\label{lem:Chow quotient}
  Let $X$ be a good algebraic space with an action of $\mrm{GL}_N$, and let $\sX = [X/\mrm{GL}_N]$.
  Then $\mrm{cl}_\sX : \CH_n^{\mrm{cyc}}(\sX) \to \CH_n(\sX)$ is invertible for all $n \in \Z$.
\end{lem}
\begin{proof}
Let $p : \sV \to \sX$ and $j : W \hookrightarrow \sV$ be as in the proof of \lemref{lem:Chow approx}.
Then $j^!p^!$ is invertible on $\CH^{\mrm{cyc}}_n(\sX)$ by \cite[Def.-Prop.~1]{EdidinGraham}, and on $\CH_n(\sX)$ by \lemref{lem:Chow approx}.
Since this isomorphism commutes with the cycle class map, the claim then follows from \propref{prop:Chow spaces} for $W$.
\end{proof}

\begin{lem}\label{lem:Chow K quotient}
  Let $\sX$ be a good stack and $i : \sZ \hookrightarrow \sX$ a closed immersion whose open complement is of the form $[X/\mrm{GL}_N]$, with $X$ as in \lemref{lem:Chow quotient}.
  Then $i$ satisfies condition~(K) of \lemref{lem:Chow devissage}.
\end{lem}
\begin{proof}
We adapt the argument of \cite[Prop.~4.2.1]{Kresch}.
Let $\sU := \sX \setminus \sZ$ and $\beta \in \CH_n(\sU, 1)$, and let $p : \sV \to \sU$ and $j : W \hookrightarrow \sV$ be as in the proof of \lemref{lem:Chow approx}, so that $W$ is an algebraic space and $\dim(\sV \setminus W) < n+r$ where $r=\rk_\sU(\sV)$.
By \cite[Prop.~2.3.3]{Kresch}, there exists a projective morphism $h : \sX' \to \sX$ which is an isomorphism over $\sU$, and a vector bundle $p' : \sV_{\sX'} \to \sX'$ of rank $r$ restricting to $p$ over $\sU$.
Let $\sZ' \subseteq \sX'$ be the reduced inverse image of $\sZ \subseteq \sX$ and $\sV_{\sZ'} := \sV_{\sX'} \fibprod_{\sX'} \sZ'$, so that $\sV_{\sZ'}$ is closed in $\sV_{\sX'} \setminus W$ with complement $\sV \setminus W$:
\[
\begin{tikzcd}[row sep=small]
  \sV \ar[hook]{r}\ar{d}{p} & \sV_{\sX'} \ar{d}{p'} & \sV_{\sZ'} \ar[hook']{l}\ar{d}{p'_{\sZ'}}
  \\
  \sU \ar[hook]{r}\ar[equal]{d} & \sX' \ar{d}{h} & \sZ' \ar[hook']{l}\ar{d}{h_{\sZ'}}
  \\
  \sU \ar[hook]{r} & \sX & \sZ. \ar[hook']{l}{i}
\end{tikzcd}
\]
By \lemref{lem:Chow K cycle} applied to $j^!p^!(\beta) \in \CH_{n+r}(W, 1)$ and the reduced complement of the open $W \subseteq \sV_{\sX'}$, there is an $(n+r)$-cycle $\alpha$ on $\sV_{\sX'} \setminus W$ with $[\alpha] = \partial(j^!p^!(\beta))$ in $\CH_{n+r}(\sV_{\sX'} \setminus W)$, and whose direct image to $\sV_{\sX'}$ is rationally equivalent to zero.
Since $\dim(\sV \setminus W) < n+r$, $\alpha$ is supported on $\sV_{\sZ'}$, and the direct image map $\CH_{n+r}(\sV_{\sZ'}) \to \CH_{n+r}(\sV_{\sX'} \setminus W)$ is invertible by localization and \lemref{lem:Chow dim}.
We may thus regard $\alpha$ as a cycle on $\sV_{\sZ'}$ satisfying:
\begin{enumerate}
  \item[(a)] the direct image $[\alpha] \in \CH_{n+r}^{\mrm{cyc}}(\sV_{\sX'})$ vanishes, by definition of rational equivalence;
  \item[(b)] $[\alpha] = p'^!_{\sZ'}\partial_{\sZ'}(\beta)$ in $\CH_{n+r}(\sV_{\sZ'})$, where $\partial_{\sZ'} : \CH_n(\sU, 1) \to \CH_n(\sZ')$ denotes the boundary map for $\sZ' \hookrightarrow \sX'$.
  Indeed, $\partial(j^!p^!(\beta))$ is the direct image of $\partial_{\sV_{\sZ'}}(p^!(\beta))$, where $\partial_{\sV_{\sZ'}} : \CH_{n+r}(\sV, 1) \to \CH_{n+r}(\sV_{\sZ'})$ is the boundary map for the closed immersion $\sV_{\sZ'} \hookrightarrow \sV_{\sX'}$, whence $[\alpha] = \partial_{\sV_{\sZ'}}(p^!(\beta)) = p'^!_{\sZ'}\partial_{\sZ'}(\beta)$ in $\CH_{n+r}(\sV_{\sZ'})$.
\end{enumerate}
By homotopy invariance of $\CH_*^{\mrm{cyc}}(-)$, there exists a unique $\eta \in \CH^{\mrm{cyc}}_n(\sZ')$ such that $p'^!_{\sZ'}(\eta) = [\alpha]$ in $\CH^{\mrm{cyc}}_{n+r}(\sV_{\sZ'})$, and we set $\zeta := h_{\sZ'*}(\eta) \in \CH^{\mrm{cyc}}_n(\sZ)$.
The direct image of $\eta$ along $\sZ' \hookrightarrow \sX'$ is zero: its image under the injective map $p'^!$ is $[\alpha] \in \CH^{\mrm{cyc}}_{n+r}(\sV_{\sX'})$, which vanishes by (a); hence also $i_*(\zeta) = 0$.
By (b), $p'^!_{\sZ'}\mrm{cl}_{\sZ'}(\eta) = [\alpha] = p'^!_{\sZ'}\partial_{\sZ'}(\beta)$ in $\CH_{n+r}(\sV_{\sZ'})$, hence $\mrm{cl}_{\sZ'}(\eta) = \partial_{\sZ'}(\beta)$ by homotopy invariance of $\CH_*(-)$.
Since $h$ is an isomorphism over $\sU$, it follows that $\mrm{cl}_\sZ(\zeta) = h_{\sZ'*}\partial_{\sZ'}(\beta) = \partial(\beta)$.
\end{proof}

\sssec{Proof of \thmref{thm:Chow}}\label{sssec:proof of thm:Chow}

By \lemref{lem:Chow devissage} for the closed immersion $\sX_{\mrm{red}} \hookrightarrow \sX$, where (K) holds vacuously, we may assume that $\sX$ is reduced.
By \cite[Prop.~3.5.9]{Kresch}, $\sX$ admits a finite stratification by reduced locally closed substacks, each of the form $[X/\mrm{GL}_N]$ with $X$ as in \lemref{lem:Chow quotient}.
We may thus assume that $\sX$ is nonempty, with a nonempty open stratum $\sU$ whose reduced complement $\sZ$ is a good stack stratified by the remaining strata.
By induction on the number of strata, we may assume that $\mrm{cl}_\sZ$ is invertible.
Since $\mrm{cl}_\sU$ is invertible by \lemref{lem:Chow quotient}, and (K) holds for $\sZ \hookrightarrow \sX$ by \lemref{lem:Chow K quotient}, it then follows that $\mrm{cl}_\sX$ is invertible by \lemref{lem:Chow devissage}.
\qed

\bigskip \noindent {\href{mailto:adeelkhan@as.edu.tw}{adeelkhan@as.edu.tw}} \medskip

{\scriptsize
\noindent Institute of Mathematics, Academia Sinica, Taipei 106319, Taiwan\\
\noindent National Center for Theoretical Sciences, National Taiwan University, Taipei 106319, Taiwan
}

\end{document}